\documentclass[11pt,a4paper]{article}
\usepackage[english]{babel}
\usepackage[utf8]{inputenc}
\UseRawInputEncoding
\usepackage{setspace}
\usepackage[left=2.5cm,right=2.5cm,top=2.5cm,bottom=2.5cm]{geometry}
\usepackage{epsfig}
\usepackage{titlesec}
\usepackage{sectsty}
\usepackage{enumerate,amsmath,amsthm,amsfonts,amssymb, calrsfs, cases,graphicx, comment, color, soul,todonotes,bbm, bm}
\usepackage{float,placeins}

\usepackage{booktabs} 

\newcommand{\myfrac}[2]{
    \setbox0\hbox{$#1$}        
    \dimen0=\wd0               
    \setbox1\hbox{$#2$}         
    \dimen1=\wd1               
    \ifdim\wd0<\wd1            
        \dfrac{#1\hfill}{#2}  
    \else                      
        \dfrac{#1}{#2\hfill}   
    \fi
}

\usepackage{amscd}
\usepackage{mathtools}

\usepackage{xcolor}
\usepackage[normalem]{ulem}
\usepackage{cancel}
\usepackage[shortlabels]{enumitem}
\usepackage{amstext}
\usepackage{verbatim}
\usepackage{catchfile}
\usepackage{autobreak}

\newtheorem{theo}{Theorem}[section]
\newtheorem{lemma}[theo]{Lemma}
\newtheorem{prop}[theo]{Proposition}

\newtheorem{cor}[theo]{Corollary}
\newtheorem{defi}[theo]{Definition}
\newtheorem{claim}[theo]{Claim}

\newtheorem{rem}[theo]{Remark}

\DeclareMathAlphabet{\pazocal}{OMS}{zplm}{m}{n}

\newcommand{\Ab}{\pazocal{A}}
\newcommand{\Bb}{\pazocal{B}}
\newcommand{\Cb}{\pazocal{C}}
\newcommand{\Db}{\pazocal{D}}

\providecommand{\keywords}[1]
{
	\small	
	\textit{Keywords:} #1
}
\providecommand{\subjclass}[2][2020]
{
	\small
	\emph{AMS Mathematics Subject Classification 2020:} #2
}

\title{Mean field games with terminal state constraints}

\author{Luciano Campi \footnote{Department of Mathematics ``Federigo Enriques'', University of Milan, Via Saldini 50, 20133, Milan, Italy. E-mail address: luciano.campi@unimi.it}
\and Luca Di Persio \footnote{College of Mathematics, 
		Department of Computer Science, University of Verona, strada le Grazie 15, 37134, Verona, Italy. E-mail address: luca.dipersio@univr.it}
	\and Viktorya Vardanyan \footnote{Department of Mathematics, University of Trento, Via Sommarive 14, 38123,  Povo (TN), Italy. E-mail address: viktorya.vardanyan@univr.it}
}

\date{} 
\usepackage{caption}
\usepackage{hyperref}

\begin{document}
	\maketitle
	
\begin{abstract} 
		
We study a mean field game (MFG) of state and control with state dynamics described by stochastic differential equations driven by both idiosyncratic and common noise, and subject to the constraint that the terminal state variable belongs to a nonempty, convex, closed set. The mean-field interaction enters both the dynamics and the costs through state and control. Inspired by the work of El Karoui, Peng and Quenez \cite{Karoui Peng} and Ji and Zhou \cite{Ji Zhou}, we formulate an equivalent auxiliary MFG problem and derive the stochastic maximum principle for an auxiliary optimization problem under fixed flows. By means of a suitable forward-backward stochastic differential equation (FBSDE) of conditional McKean-Vlasov type, we establish a characterization of MFG solutions via such a system. Moreover, we prove the existence and uniqueness of solutions for the FBSDE system in the specific case where the running cost is zero and the state coefficients are linear. Additionally, we obtain some uniqueness results \`a la Lasry-Lions and we apply our findings to MFGs of optimal investment with a quadratic relative performance criterion. 
\end{abstract}
	
\keywords{Mean field games, state constraints, stochastic maximum principle, Ekeland's variational principle, optimal investment.}\smallskip

\subjclass[2020]{49N80, 93E20, 49N90.}

\section{Introduction}
Mean field games (MFGs) were introduced by Lasry and Lions in \cite{Lasry Lions} and, independently, by Huang, Malham\'e, and Caines in \cite{Huang}. MFGs are limiting models of large population symmetric stochastic differential games with mean field type interaction as the number of players $N$ goes to infinity. Moreover, in these games, it is assumed that the influence of each player on the overall system becomes negligible as $N$ increases.

In this paper, we study MFGs characterized by both idiosyncratic and common noise, with a constrained terminal state. Common noise affects all players simultaneously, introducing correlated uncertainty into the system. This is in contrast to idiosyncratic noise, which involves each player independently (see, e.g., \cite{Carmona Delarue I}, \cite{Carmona Delarue II}, \cite{Carmona Delarue P}).\par
 The dynamics of the private states of player $i \in \{1,\ldots,N\}$ in the $N$-player game are given by
\begin{equation*}
\begin{cases}
    \begin{aligned}
       dX_t^i = b(t,X_t^i,  \alpha_t^i, \beta_t^i, \bar{\mu}_t^N)dt & + \sigma(t,X_t^i,  \alpha_t^i,  \bar{\mu}_t^N)dW_t^i + \sigma_0(t,X_t^i,   \beta_t^i, \bar{\mu}_t^N)dW_t^0 
    \end{aligned}\\
X_0^i=a,\  X_T^i \in \mathcal{D} \\
\end{cases}
\end{equation*}
where \( \alpha_t^i \), \( \beta_t^i \) are the controls of player \( i \) at time $t$, and \( \bar{\mu}_t^N \in \mathcal P(\mathbb R^n \times \mathbb R^{n\times d} \times \mathbb R^{n\times d})\) denotes the empirical distribution of the state-control triples:
\[
\bar{\mu}_t^N = \frac{1}{N} \sum_{j=1}^N \delta_{(X_t^j, \alpha_t^j, \beta_t^j)}.
\]
$W^0, W^1,\ldots, W^N$ are $N+1$ independent $d$-dimensional  Brownian motions. $W^0$ is said to be the common source of noise and $(\mathcal{F}^0_t)_{t\in [0,T]}$ is the filtration generated by $W^0$ satisfying the usual conditions of completeness and right-continuity. $\mathcal{D}$ is a nonempty closed convex subset in $\mathbb{R}^n$.\par
Each player chooses a control to minimize the following expected cost functional
\begin{equation*}
   J^i(\alpha^1,\ldots,\alpha^N, \beta^1,\ldots, \beta^N)=\mathbb{E}\left[\int _0^T \bar{\ell}\left(t,X_t^i,  \alpha_t^i, \beta_t^i, \bar{\mu}_t^N \right)dt+\phi \left(X_T^i,\bar{\mu}_T^{X,N} \right)\right],
\end{equation*}
for given functions $\bar \ell$ and $\phi$. In this $N$-player  symmetric stochastic differential game, the players interact through the empirical distribution of their private states and controls, where the empirical distribution of the states  is given by
$
\bar{\mu}_t^{X,N} = \frac{1}{N} \sum_{i=1}^{N} \delta_{X_t^i} .
$
As \( N \to \infty \), we expect the empirical distributions \( \bar{\mu}_t^N \) to converge, at equilibrium, to the conditional distribution \( \mathcal{L}(X_t, \alpha_t, \beta_t \mid \mathcal{F}_t^0) \), given the common noise \( W^0 \). Consequently, we will also have that \( \bar{\mu}_t^{X,N} \) converges to \( \mathcal{L}(X_t \mid \mathcal{F}_t^0) \). Therefore, the limiting  MFG problem  can be roughly described as follows:
\begin{enumerate}
    \item (Optimality condition) For  a given flow $(\mu_t)_{t\in [0, T]}$, adapted to the common noise filtration, one solves the following optimization problem
\begin{equation}
\label{forward CN 1}
     \inf_{(\alpha, \beta) \in \mathcal{A},\, X_T \in \mathcal D} \mathbb{E}\left[\int _0^T \bar{\ell}(t,X_t,  \alpha_t, \beta_t, \mu_t)dt+\phi(X_T,\mu_T^X)\right],
\end{equation}
where $(\mu^X _t)_{t\in [0,T]}$ denotes the flow of first marginals of $(\mu_t)_{t\in [0,T]}$, and
\begin{align*}
    &dX_t  = b(t,X_t,  \alpha_t, \beta_t, \mu_t)dt+\sigma(t,X_t, \alpha_t, \mu_t)dW_t+\sigma_0(t,X_t, \beta_t, \mu_t)dW_t^0, \\ & X_0=a.
 \end{align*}
 $\mathcal A$ denotes the set of all admissible controls, which will be rigorously defined in the next section.
\item (Fixed point or consistency condition) 
Find a stochastic  flow \( (\hat\mu_t)_{t \in [0,T]} \) such that 
\[ \hat\mu_t=\mathcal L(\hat X_t, \hat\alpha_t, \hat\beta_t | \mathcal F_t ^0), \quad t\in [0,T],\]
where $(\hat\alpha_t, \hat\beta_t)$ is the solution of the above optimization problem and $\hat X$ denotes the corresponding state.  
\end{enumerate} 
 
In the case of deterministic dynamics, MFGs with state constraints are studied by Cannarsa, Capuani, and Cardaliaguet  in \cite{CanCap}, \cite{CanCar}, and \cite{Cannarsa}. They solve the problem via a system of partial differential equations (PDEs) associated with the state-constrained MFG. 
Under fairly general assumptions in \cite{CanCap},  the existence of MFG equilibria is shown. In the same paper, the uniqueness of the solution is also derived under an adapted version of the Lasry-Lions monotonicity condition. In \cite{CanCar}, they address the regularity of solutions and obtain solutions that are more regular. In \cite{Cannarsa}, the authors obtain a more useful regularity property of generalized solutions, which is a global semiconcavity property. 
Porretta and Ricciardi have analysed state-constrained MFGs of second order in \cite{Porretta}. They consider a model in which the individual dynamics are affected by additive idiosyncratic noise. The equilibrium is characterized by a solution to a second-order MFG system, in which the value function diverges at the boundary and the players' density is smooth near it. Graber and Mayorga in \cite {GM} extend the results of \cite{CanCap} to the case of an MFG of controls, where the controls are velocities. 
\par
 In \cite{Daudin II} Daudin addresses the convergence problem of mean-field control theory with state constraints and non-degenerate idiosyncratic noise. The pre-limit problem involves a large number of interacting particles subject to symmetric, almost-sure constraints. In contrast, in the limit, the constraint acts on the law of a typical particle.  The limit problem is an optimal control problem of the Fokker-Planck equation under state constraints on the space of probability measures, which has been studied by Daudin  in \cite{Daudin I}.\par
To address the terminal state-constrained problem 
\eqref{forward CN 1}, we use the approach introduced by El Karoui, Peng, and Quenez in \cite{Karoui Peng}. They use the backward approach, proving a version of the stochastic maximum principle by taking the terminal wealth as control. Ji and Zhou, in  \cite{Ji Zhou}, use the backward formulation and Ekeland's variational principle \cite{Ekeland} to derive the stochastic maximum principle (SMP). Peng in \cite{Peng} proves the SMP for stochastic optimal control problems with an endpoint constraint on the state variable using Ekeland's variational principle. Buckdahn, Djehiche, and Li in \cite{Buckdahn Li 1}, \cite{Buckdahn Li 2} study the optimal control for stochastic differential equations (SDEs) of mean-field type and obtain a Peng-type SMP. 

Inspired by the approach developed in those articles, we formulate an auxiliary optimization problem by converting the terminal state constraint into a control constraint. This transformation allows for a more tractable analysis and solution approach. In the auxiliary problem, the initial condition of the original problem becomes a constraint. Using Ekeland's variational principle, we derive a SMP that characterizes the optimal terminal state. Specifically, the transversality condition provides a necessary condition for the optimal terminal state. A similar transversality condition can be observed in the book by Yong and Zhou \cite[Chap. 3, Sec. 6]{Yong Zhou},  where they establish the SMP for optimal control problems with constrained cost functionals. Necessary and sufficient conditions for the SMP are obtained, leading to a suitable FBSDE system of conditional McKean-Vlasov type.
We study the well-posedness of such a system in the one-dimensional linear case. Finally, to illustrate the resolution method developed in this paper, we provide a financial application to an optimal investment problem for a representative risk-neutral agent with a quadratic relative performance criterion in two scenarios with and without common noise. \par
The paper's structure is as follows: Section 2 introduces the problem setting and establishes the equivalence between the original MFG with terminal state constraints and the auxiliary formulation. In Section 3, we derive the SMP (both necessary and sufficient conditions), establish a verification theorem, discuss the uniqueness of MFG solutions, and prove the existence of the solution to the FBSDE for a specific linear case.  In Section 4, we present the financial applications with some numerics and economic interpretations. Finally, the appendix contains the proofs of the necessary and sufficient conditions, as well as statements and proofs of several intermediate results.

\section{Setting}
We aim to analyze MFGs with terminal-state constraints, subject to both idiosyncratic and common noise. We consider a fairly general model in which the coefficients of the controlled stochastic differential equation (SDE) describing the time evolution of the state variable depend on the conditional joint law of the state variable and the controls.

Let $(\Omega, \mathcal{F},(\mathcal{F}_t)_{t\in [0,T]}, \mathbb{P})$  be a given complete filtered probability space supporting two $d$-dimensional independent $(\mathcal F_t)$-Brownian motions $W$ and $W^0$.  The filtration generated by $W^0$ will be denoted  by $(\mathcal{F}^0_t)_{t\in [0,T]}$. Let $p\geq 2$. We denote by $\mathcal H^p _{n\times d}$ the space of all $(\mathcal{F}_t)$-progressively measurable processes $Z$ with values in $\mathbb{R}^{n\times d}$ such that $\|Z\|_{\mathcal H^p _{n\times d}}^p:= \mathbb E [(\int_0^T |Z_t|^2 dt )^{p/2}]< \infty$, and by $\mathcal S^p _n$, the space of all $(\mathcal{F}_t)$-progressively measurable continuous processes  $Z$ with values in $\mathbb{R}^n$ such that $\|Z\|^p _{\mathcal S^p _n} :=\mathbb E [\sup_{0 \leq t \leq T}|Z_t|^p]< \infty$. When $d=n=1$ we will simply write $\mathcal H^p$ and $\mathcal S^p$. All filtrations are assumed to satisfy the usual conditions of completeness and right-continuity.

Let $\mathcal{P}_2(\mathbb{R}^n)$ denote the set of probability measures on $\mathbb{R}^n$ with finite second moments, meaning that for any \( \mu \in \mathcal{P}_2(\mathbb{R}^n) \),
\[
    M_2 (\mu) := \left(\int_{\mathbb{R}^n} |x|^2 d\mu(x)\right)^{1/2} < \infty.
\]

We denote by $\mathbb R^+$ the set of all positive real numbers, that is, $\mathbb R^+:=(0, \infty).$ For ease of notation, we will denote the scalar product between two vectors $x$ and $y$ as $xy$.

The state evolves according to the following SDE
\begin{equation}
    \label{SDE CN}
\begin{cases}
    \begin{aligned}
       dX_t=b(t,X_t,  \alpha_t, \beta_t, \mu_t)dt+\sigma(t,X_t, \alpha_t, \mu_t)dW_t 
+\sigma_0(t,X_t, \beta_t, \mu_t)dW_t^0 
    \end{aligned}\\
X_0=a
\end{cases}
\end{equation}
with the terminal state constraint $X_T \in \mathcal{D}$, where $\mathcal{D}$ is a nonempty closed convex subset in $\mathbb{R}^n$. In the SDE above, $\mu_t$ is a fixed process adapted  to the filtration $(\mathcal F_t ^0)$, $a \in \mathbb{R}^n$ is given, and the following coefficients 
\begin{itemize}
    \item $b:[0, T ] \times\mathbb{R}^n\times  \mathbb{R}^{n\times d}\times \mathbb{R}^{n\times d} \times  \mathcal P_2(\mathbb R^n \times \mathbb R^{n\times d} \times \mathbb R^{n\times d}) \to \mathbb{R}^n$ ,
    \item $(\sigma, \sigma^0):[0, T ] \times\mathbb{R}^n\times  \mathbb{R}^{n\times d} \times  \mathcal P_2(\mathbb R^n \times \mathbb R^{n\times d} \times \mathbb R^{n\times d}) \to \mathbb{R}^{n\times d} \times \mathbb{R}^{n\times d}$,
\end{itemize}
are given measurable functions. More specific assumptions on these coefficients will follow.

The objective is to minimize the following cost functional over admissible control strategies $\alpha$ and $\beta$:
\begin{equation} 
\label{Cost functional CN}
   J(\alpha, \beta)=\mathbb{E}\left[\int _0^T \bar{\ell}(t,X_t,  \alpha_t, \beta_t, \mu_t)dt+\phi(X_T,\mu_T^X) \right],
\end{equation}
where 
\begin{itemize}
    \item $\bar{\ell}:[0, T ] \times\mathbb{R}^n\times  \mathbb{R}^{n\times d}\times \mathbb{R}^{n\times d} \times  \mathcal P_2(\mathbb R^n \times \mathbb R^{n\times d} \times \mathbb R^{n\times d}) \to \mathbb{R}$ is the running cost and
    \item $\phi :\mathbb{R}^n\times\mathcal{P}_2(\mathbb{R}^n) \to \mathbb{R}$ is the terminal cost.
\end{itemize}
Admissible controls $\alpha$ and $\beta$ are defined as $(\mathcal F_t)$-progressively measurable $\mathbb R^{n \times d}$-valued stochastic processes satisfying the square-integrability condition
\[\mathbb{E} \left[\int_0^T \left( |\alpha_t|^2 + |\beta_t|^2 \right) dt \right] < \infty.
    \label{admissibility condition}\]  
The set of admissible controls is denoted by $\mathcal{A}$, i.e., $\mathcal A = \mathcal H^2 _{n\times d} \times \mathcal H^2 _{n\times d}$.

Before giving the definition of solution of the MFG with terminal state constraint, we set the following standing assumptions.

\begin{enumerate}[label= (A\arabic*)]
\item $b, \sigma, \sigma_0$, and $\bar{\ell}$ are continuous in their arguments and continuously differentiable in $(x, \alpha, \beta)$, while $\phi$ is continuous and continuously differentiable in $x$;
\item the first  derivatives of $b$, $\sigma, \sigma_0$ in $(x, \alpha, \beta)$ are bounded;
\item the first derivatives of \( \bar{\ell}(t, x, \alpha, \beta, \mu) \) with respect to \( x, \alpha, \beta \) satisfy 
\[
|\partial_{(x,\alpha,\beta)} \bar{\ell}(t, x, \alpha, \beta, \mu)| \leq C \left(1 + |x| + |\alpha| + |\beta| + M_2(\mu)\right), 
\]
for all $ x \in \mathbb{R}^n$, $\mu \in \mathcal{P}_2(\mathbb R^n \times \mathbb R^{n\times d} \times \mathbb R^{n\times d})$, $t \in [0,T]$ and $\alpha, \beta \in \mathbb{R}^{n \times d}$;
the first derivative of \( \phi(x, \mu) \) with respect to \( x \) satisfies
\[
|\partial_x \phi(x, \mu)| \leq C \left(1 + |x| + M_2(\mu)\right), \quad \forall x \in \mathbb{R}^n, \mu \in \mathcal{P}_2(\mathbb R^n );
\]
\item[(A4)] there exists a constant $C > 0$ such that, for all $t \in [0,T]$ and $\mu \in \mathcal{P}_2(\mathbb R^n \times \mathbb R^{n\times d} \times \mathbb R^{n\times d})$, $b, \sigma$ and $\sigma_0$ satisfy the following growth condition at the origin:
\[
|b(t,0,0,0,\mu)| + |\sigma(t,0,0,\mu)| + |\sigma_0(t,0,0,\mu)| \leq C(1 + M_2(\mu)).
\]
\end{enumerate}
Under these conditions, for given $\alpha, \beta \in \mathcal{A}$ and for a given adapted process $\mu_t$ with values  in the space $\mathcal{P}_2(\mathbb R^n \times \mathbb R^{n\times d} \times \mathbb R^{n\times d})$ such that $\mathbb E \left[\int_0^T M_2(\mu_t)^2 dt\right]<\infty$ , there exists a unique solution  $X \in \mathcal S^2 _n$ to \eqref{SDE CN}, \cite[Theorem 6.16, p.49]{Yong Zhou}.

\begin{defi} \label{def-MFG-constraint} A MFG solution with terminal state constraint is any tuple $(\hat \alpha, \hat\beta, \hat \mu)$ of processes such that
\begin{enumerate}
\item $(\hat\alpha, \hat\beta) \in \mathcal A$ and  $\hat \mu$ is a $(\mathcal F_t ^0)$-progressively measurable process with values in $\mathcal{P}_2(\mathbb R^n \times \mathbb R^{n\times d} \times \mathbb R^{n\times d})$, such that $\mathbb E \left[\int_0^T M_2(\hat \mu_t)^2 dt\right]<\infty$;
\item (Optimality condition) Given the flow $\hat \mu$, the pair $(\hat \alpha, \hat \beta)$ solves the following optimal control problem
\begin{equation}
\label{forward CN}
     \inf_{(\alpha, \beta)\in \mathcal A,  X_T \in \mathcal D} \mathbb{E}\left[\int _0^T \bar{\ell}(t,X_t,  \alpha_t, \beta_t, \hat\mu _t)dt+\phi(X_T,\hat\mu^X_T)\right],
\end{equation}
where
\begin{equation}
\begin{aligned}
    dX_t=b(t,X_t,  \alpha_t, \beta_t, \hat\mu_t)dt+\sigma(t,X_t, \alpha_t, \hat\mu_t)dW_t
    +\sigma_0(t,X_t, \beta_t, \hat\mu_t)dW_t^0, \quad  X_0=a;
 \end{aligned}   
\end{equation}
\item (Consistency condition)  For all $t\in [0,T]$, we have \[\hat\mu_t=\mathcal{L}(\hat X_t, \hat\alpha_t, \hat\beta_t|\mathcal{F}_t^0),\] where $\hat X$ denotes the state controlled by $(\hat\alpha, \hat\beta)$.    \end{enumerate}
\end{defi}

In order to solve the above constrained stochastic optimal control problem \eqref{forward CN}, we will provide an equivalent backward formulation of it as in \cite{Karoui Peng,Ji Zhou}. We will make use of the following extra assumption:
\begin{enumerate}
\item[(A5)] there exists $k_1 >0$ such that 
\[ |\sigma(t,x, \alpha_1, \mu)-\sigma(t,x, \alpha_2, \mu)| + |\sigma_0(t,x, \beta_1, \mu)-\sigma_0(t,x, \beta_2, \mu)| \geq k_1(|\alpha_1-\alpha_2| + |\beta_1-\beta_2|),\] 
for all $ x \in \mathbb{R}^n$, $\mu \in \mathcal{P}_2(\mathbb R^n \times \mathbb R^{n\times d} \times \mathbb R^{n\times d})$, $t \in [0,T]$ and $\alpha_1, \alpha_2, \beta_1,\beta_2 \in \mathbb{R}^{n \times d}$.

\end{enumerate}
Assumptions (A1), (A2), and (A5) imply
that the two mappings 
\[
\alpha \mapsto \sigma(t,x, \alpha, \mu), \quad 
\beta \mapsto \sigma_0(t,x, \beta, \mu)
\]
are  bijections from $\mathbb{R}^{n\times d}$ onto itself for any $(t,x, \mu)$.
Indeed, the inequality in (A5) implies that the mappings
\[
\alpha \mapsto \sigma(t,x,\alpha,\mu), \quad \beta \mapsto \sigma_0(t,x,\beta,\mu)
\]
are injective because distinct inputs \(\alpha_1 \neq \alpha_2\) (or \(\beta_1 \neq \beta_2\)) yield distinct outputs.  By (A1), both maps are continuously differentiable, the  lower bound in (A5) implies that the Jacobian matrices are uniformly nonsingular, so the Inverse Function Theorem applies and gives local invertibility at every point. 
In addition, (A5) implies that these mappings are coercive (i.e., $|\sigma(t,x,\alpha,\mu)| \to \infty$ as $|\alpha| \to \infty$ and similarly for $\sigma_0$), which implies that they are proper maps in $\mathbb{R}^{n \times d}$. According to the Hadamard global inversion theorem, a proper local $C^1$-diffeomorphism between Euclidean spaces is a global diffeomorphism. Thus, the mappings are bijective onto the entire space $\mathbb{R}^{n \times d}$, and their global inverses $\tilde{\sigma}$ and $\tilde{\sigma}_0$ are well-defined.

\medskip

Let $ q \equiv \sigma(t, x, \alpha, \mu)$, $z \equiv \sigma_0(t, x, \beta, \mu)$, and denote the inverse functions by $\alpha=\tilde{\sigma}(t, x, q, \mu)$, $\beta=\tilde{\sigma}_0(t, x, z, \mu)$, respectively.
Therefore the system \eqref{SDE CN} can be rewritten as
\begin{equation*}
      -dX_t=f(t, X_t, q_t, z_t, \mu_t)dt -q_tdW_t-z_tdW_t^0, \quad X_0=a,
\end{equation*}
where  
\[ f(t, x, q, z, \mu) \equiv -b(t, x, \tilde{\sigma}(t, x, q, \mu), \tilde{\sigma}_0(t, x, z, \mu),  \mu).\]
The triple $(a, q, z)$ gives the terminal value $X_T$ (by solving the SDE), while $X_T$ gives the triple $(a,q,z)$ (by solving the BSDE).
Hence we can introduce the following controlled system: 
\begin{equation}
\label{BSDE CN}
\begin{cases}
      -dX_t=f(t, X_t, q_t, z_t, \mu_t)dt-q_tdW_t-z_tdW_t^0\\
      X_T=\xi
\end{cases}
\end{equation}
where the control is now the random variable $\xi$ to be chosen from the following set 
\[
L^2(\mathcal D)=\{ \xi \  : \  \mathbb E[|\xi|^2]< \infty,\, \xi \in \mathcal{D} \text{ a.s.}\}.
\]
For each $\xi \in L^2(\mathcal D)$, consider the following cost
\[
J(\xi)= J(\xi; \mu)=\mathbb E \left[\int_0^T \ell(t, X_t, q_t, z_t, \mu_t)dt+\phi(\xi,\mu_T^X) \right],
\]
where
\[ \ell(t, x, q, z, \mu) \equiv \bar \ell(t, x, \tilde{\sigma}(t, x, q, \mu), \tilde{\sigma}_0(t, x, z, \mu),  \mu),\]
and where we use the same letter $J$ as in the original problem with a little abuse of notation. This gives rise to the following auxiliary optimization problem: 
\begin{equation}
\label{backward CN}
\inf_{\xi \in L^2(\mathcal D),\, X_0 ^\xi =a} 
J(\xi ),
\end{equation}
where $X_0^{\xi}$ is the solution of \eqref{BSDE CN} at time 0 under $\xi$.

As long as the mappings 
$\alpha \mapsto \sigma(t,x, \alpha, \mu)\equiv q$, \ 
$\beta \mapsto \sigma_0(t,x, \beta, \mu)\equiv z $
are bijections, and due to the regularity of $\sigma$ and  $\sigma_0$, the new coefficients $f, \ell$ and $\phi$ satisfy the following properties:
\begin{itemize}
\item[(B1)] $f$ and $\ell$ are  continuous in all their arguments and continuously differentiable in $(x, q, z)$, while $\phi$ is continuous in all its arguments and continuously differentiable in $x$;
\item[(B2)] the first derivatives of $f$ in $(x,q,z)$ are bounded; 
\item[(B3)] the first derivatives of \( \ell(x, q, z, \mu) \) with respect to \( x, q, z\) satisfy
\[
|\partial_{(x,q,z)} \ell(t, x, q, z, \mu)| \leq C \left(1 + |x| + |q| + |z| + M_2(\mu)\right),
\]
for all $ x \in \mathbb{R}^n$, $\mu \in \mathcal{P}_2(\mathbb R^n \times \mathbb R^{n\times d} \times \mathbb R^{n\times d})$, $t \in [0,T]$ and $q, z \in \mathbb{R}^{n \times d}$; 
the first derivative of \( \phi(x, \mu) \) with respect to \( x \) satisfies
\[
|\partial_x \phi(x, \mu)| \leq C \left(1 + |x| + M_2(\mu)\right), \forall  x \in \mathbb{R}^n, \mu \in \mathcal{P}_2(\mathbb R^n);
\]
\item[(B4)] there exists a constant $C > 0$ such that, for all $t \in [0,T]$ and $\mu \in  \mathcal{P}_2(\mathbb R^n \times \mathbb R^{n\times d} \times \mathbb R^{n\times d})$, $f$ satisfies the following growth condition at the origin:
\[
|f(t,0,0,0,\mu)| \leq C(1 + M_2(\mu)).
\]
\end{itemize}
Under  these assumptions BSDE \eqref{BSDE CN} has a unique solution $(X^\xi, q^\xi, z^\xi)  \in \mathcal S^2 _n \times  \mathcal H^2 _{n\times d}\times \mathcal H^2 _{n\times d}$  for  given $(\mathcal F_t ^0)$-adapted process $\mu_t$ with values in $\mathcal{P}_2(\mathbb{R}^n\times \mathbb{R}^{n\times d} \times \mathbb{R}^{n\times d})$ such that $\mathbb E \left[\int_0^T M_2(\mu_t)^2 dt\right]<\infty$  \cite[Theorem 2.1, p.18]{Karoui Peng Quenez}, \cite[Theorem 3.1, p.58]{Pardoux}. 

\begin{defi}
 Let $a\in \mathbb R^n$. A random variable $\xi \in L^2 (\mathcal D)$ is called $a$-feasible if the solution of \eqref{BSDE CN} satisfies $X_0^\xi=a$. We shall denote by $L^2 _a(\mathcal D)$ the set of all $a$-feasible $\xi$'s. Moreover, an $a$-feasible $\xi^*$ is called $a$-optimal or simply optimal if it attains the minimum of $J(\xi)$ over $L^2 _a(\mathcal D)$. 
\end{defi}
For a fixed flow, the original problem \eqref{forward CN} is equivalent to the auxiliary one \eqref{backward CN}.  Since the mappings $\alpha \mapsto \sigma(t,x, \alpha, \mu) \equiv q, \quad \beta \mapsto \sigma_0(t,x, \beta, \mu)\equiv z$ are bijections, the processes $q$ and $z$
can be treated as control variables. Moreover the triple $(a, q, z)$ gives the terminal value $X_T$ (by solving the SDE), while $X_T$ gives the triple $(a,q,z)$ (by solving the BSDE). 

Since $\xi$ now is the control variable, the state constraint becomes a control constraint  in \eqref{backward CN}. The backward formulation makes the problem easier to solve but in  addition we will have the original initial condition as a constraint. We can now provide the definition of the auxiliary MFG problem.

\begin{defi} \label{def-auxMFG} A solution to the \emph{auxiliary} MFG with terminal state constraint is any pair $(\hat  \xi, \hat \mu)$, where $\hat \xi \in L^2 _a (\mathcal D)$, $\hat \mu$ is an $(\mathcal F_t ^0)$-adapted process with values in $\mathcal{P}_2(\mathbb{R}^n\times \mathbb{R}^{n\times d} \times \mathbb{R}^{n\times d})$, such that $\mathbb E \left[\int_0^T M_2(\hat \mu_t)^2 dt\right]<\infty$, and the following two properties are satisfied:
\begin{enumerate}
    \item (Optimality condition) Given $\hat \mu$, the random variable $\hat \xi$ solves the following optimal control problem
\begin{equation}
\label{auxoptimal}
     \inf_{\xi\in L^2 _a (\mathcal D)}\mathbb E \left[\int_0^T \ell(t, X_t, q_t, z_t, \hat \mu_t)dt+\phi(\xi, \hat\mu^X_T)\right],
\end{equation}
where $(X,q,z) \in \mathcal S^2_n \times  \mathcal H^2_{n \times d}\times \mathcal H^2_{n\times d}$ solves
\begin{equation}
-dX_t=f(t, X_t, q_t, z_t, \hat \mu_t)dt -q_tdW_t-z_tdW_t^0, \quad X_T=\xi;
\end{equation}
\item (Consistency condition)  For all $t\in [0,T]$, we have \[\hat \mu_t=\mathcal{L}(\hat X_t, \hat\alpha_t, \hat\beta_t |\mathcal{F}_t^0), \quad t\in [0,T], \] where $\hat X$ denotes the state with terminal condition $\hat X_T  = \hat \xi $ and
\begin{equation} (\hat \alpha_t , \hat \beta_t ) = (\tilde \sigma, \tilde \sigma_0)(t, \hat X_t , \hat q_t , \hat z_t , \hat \mu_t ). \label{eq:xi-to-alpha-beta}\end{equation}
\end{enumerate}
\end{defi}

The following result establishes a bijective correspondence between the original MFG with terminal state constraint and the auxiliary formulation. This equivalence allows us to solve the constrained problem by analyzing the auxiliary MFG in the subsequent sections.
\begin{prop} 
\label{equivalence}
Under assumptions (A1)-(A5), a triple $(\hat \alpha, \hat \beta, \hat \mu)$ is a solution of the MFG with terminal state constraint  (Definition \ref{def-MFG-constraint}) if and only if the pair $(\hat \xi, \hat \mu)$ is a solution of the auxiliary MFG as in Definition \ref{def-auxMFG}.
\end{prop}
\begin{proof}
The proof consists of two directions:\\
\textbf{Direction 1: Original $\implies$ auxiliary.} 
Let $(\hat \alpha, \hat \beta, \hat \mu)$ be a solution of the original MFG. We verify the conditions of Definition \ref{def-auxMFG}:\\
 \textit{Feasibility ($a$-feasibility):} 
    By the optimality condition of the forward problem, the optimal state process $\hat X$ satisfies the initial condition $\hat X_0 = a$ and the terminal constraint $\hat X_T \in \mathcal{D}$ almost surely. By setting $\hat \xi = \hat X_T$, we clearly have $\hat \xi \in L^2_a(\mathcal{D})$. 
    Let $\hat q, \hat z$ be defined by
    \[ \hat q_t = \sigma(t, \hat X_t, \hat \alpha_t, \hat \mu_t), \quad \hat z_t = \sigma_0(t, \hat X_t, \hat \beta_t, \hat \mu_t). \]
    Due to the equivalence between SDE \eqref{SDE CN} and BSDE \eqref{BSDE CN} under assumption (A5), the triplet $(\hat X, \hat q, \hat z)$ solves the BSDE \eqref{BSDE CN} with terminal condition $\hat \xi$. Since $\hat X_0 = a$, the variable $\hat \xi$ belongs to the set of $a$-feasible random variables $L^2_a(\mathcal{D})$.\\
 \textit{Optimality condition:} 
    Suppose, toward a contradiction, that $\hat \xi$ is not optimal for the auxiliary problem \eqref{backward CN}. Then there exists a random variable $\xi' \in L^2_a(\mathcal{D})$ such that
    \begin{equation} \label{absurd_en}
    J(\xi'; \hat \mu) < J(\hat \xi; \hat \mu).
    \end{equation}
    Let $(X', q', z')$ be the solution to the BSDE \eqref{BSDE CN} associated with $\xi'$. Since $\xi' \in L^2_a(\mathcal{D})$, we have $X'_0 = a$. We define the corresponding forward controls using the inverse mappings:
    \begin{equation} \label{eq:inv-control} \alpha'_t = \tilde \sigma(t, X'_t, q'_t, \hat \mu_t), \quad \beta'_t = \tilde \sigma_0(t, X'_t, z'_t, \hat \mu_t). \end{equation}
    We must ensure $(\alpha', \beta') \in \mathcal{A}$. Under assumptions (A1), (A2), and (A5), the inverse mappings $\tilde \sigma$ and $\tilde \sigma_0$ are Lipschitz continuous in $(x, q)$ and $(x, z)$, respectively.
    
We prove the claimed Lipschitz continuity of $\tilde \sigma$ only, as the same property for $\tilde \sigma_0$ can be shown in the same way. We recall first that, under (A1), (A2), (A5), the mapping $\alpha \mapsto \sigma(t, x, \alpha, \mu)$ is bijective. We define its inverse as $\tilde{\sigma}(t, x, q, \mu) = \alpha$. To prove this inverse is Lipschitz in $(x, q)$ uniformly in $t$, we show that Lipschitz continuity holds separately in $x$ and $q$ uniformly in the other variable as well as in $t$. Joint Lipschitz continuity in $(x,q)$ follows easily. 
We start with checking the Lipschitz continuity in $x$. 
Let $x_1, x_2 \in \mathbb{R}^n$ and fix $q$. Let $\alpha_1 = \tilde{\sigma}(t, x_1, q, \mu)$ and $\alpha_2 = \tilde{\sigma}(t , x_2, q, \mu)$. Then $\sigma(t, x_1, \alpha_1, \mu) = \sigma(t, x_2, \alpha_2, \mu) = q$. Thus
    \begin{equation}\label{eq-sigma} \sigma(t, x_1, \alpha_1, \mu) - \sigma(t, x_1, \alpha_2, \mu) = \sigma(t, x_2, \alpha_2, \mu) - \sigma(t, x_1, \alpha_2, \mu). \end{equation}
Since the first derivatives of $\sigma$ are bounded  (Assumption (A2)), $\sigma$ is Lipschitz in $x$ with some constant $L>0$ uniform in the other variables. Therefore, \eqref{eq-sigma} leads to
     \[ k_1 |\alpha_1 - \alpha_2| \leq |\sigma(t, x_1, \alpha_1, \mu) - \sigma(t, x_1, \alpha_2, \mu)| = |\sigma(t, x_2, \alpha_2, \mu) - \sigma(t, x_1, \alpha_2, \mu)| \leq  L |x_1 - x_2|  . \]
 Applying this and (A5) yield
    \[ k_1 |\alpha_1 - \alpha_2| \leq L |x_1 - x_2| \implies |\tilde{\sigma}(t, x_1, q, \mu) - \tilde{\sigma}(t, x_2, q, \mu)| \leq \frac{L}{k_1} |x_1 - x_2|. \]
Now, we turn to the Lipschitz continuity in $q$. Let $q_1, q_2 \in \mathbb{R}^{n \times d}$. Let $\alpha_1 = \tilde{\sigma}(t, x, q_1, \mu)$ and $\alpha_2 = \tilde{\sigma}(t, x, q_2, \mu)$. By (A5), we have
    \[ |q_1 - q_2| = |\sigma(t, x, \alpha_1, \mu) - \sigma(t, x, \alpha_2, \mu)| \geq k_1 |\alpha_1 - \alpha_2|, \]
    which entails
    \[|\tilde{\sigma}(t, x, q_1, \mu) - \tilde{\sigma}(t, x, q_2, \mu)| \leq \frac{1}{k_1} |q_1 - q_2|. \]   
    Notice that $k_1$ is a fixed positive constant from assumption (A5). Therefore, the Lipschitz constants $L/k_1$ and $1/k_1$ above are uniform in, respectively, $(t,q,\mu)$ and $(t,x,\mu)$. Consequently, 
\[ |\tilde{\sigma}(t, x, q, \mu) - \tilde{\sigma}(t, x', q', \mu)| \leq \frac{L}{k_1} |x - x'| + \frac{1}{k_1} |q - q'|, \]
and similarly for $\tilde \sigma_0$.
We must ensure $(\alpha', \beta') \in \mathcal{A}$. We have just shown that, under assumptions (A2) and (A5), the inverse mappings $\tilde{\sigma}$ and $\tilde{\sigma}_0$ are Lipschitz continuous in $(x, q, z)$ uniformly in $t \in [0,T]$. To establish the linear growth of the inverse controls $(\alpha',\beta')$ in \eqref{eq:inv-control}, we use first the Lipschitz continuity of $\tilde{\sigma }$ in \((x, q)\) shown above. For any \((x', q', \mu)\), we have:
\begin{equation}
\label{inversezero}
    |\tilde{\sigma}(t, x', q', \mu)| \leq |\tilde{\sigma}(t, 0, 0, \mu)| + \frac{L}{k_1} |x'| + \frac{1}{k_1} |q'|.
\end{equation}
Let \(\alpha_0 = \tilde{\sigma}(t, 0, 0, \mu)\). By the definition of inverse function, it holds that \(\sigma(t, 0, \alpha_0, \mu) = 0\). We can now bound \(\alpha _{0}\) by comparing it to the origin \(\alpha = 0\). Using Assumption (A5), we obtain:
$$k_1 |\alpha_0 - 0| \leq |\sigma(t, 0, \alpha_0, \mu) - \sigma(t, 0, 0, \mu)|.$$
Substituting \(\sigma(t, 0, \alpha_0, \mu) = 0\), the inequality simplifies to
$|\alpha_0| \leq \frac{1}{k_1} |\sigma(t, 0, 0, \mu)|$.
Applying the growth condition at the origin from Assumption (A4), i.e., \(|\sigma(t, 0, 0, \mu)| \leq C(1 + M_2(\mu))\), yields 
$$|\tilde{\sigma}(t, 0, 0, \mu)| \leq \frac{C}{k_1}(1 + M_2(\mu)).
$$
Finally, substituting this back into \eqref{inversezero}, we conclude that the control \(\alpha'\) satisfies
$$|\alpha'_t| \leq \bar{C}(1 + M_2(\hat\mu_t) + |X_t '| + |q_t '|),$$
where \(\bar{C}\) is a constant independent of \(t \in [0, T]\). Identical arguments lead to the same bound for \(\beta_t ' = \tilde{\sigma}_0(t, X'_t, z_t ', \hat \mu_t)\).
Since the BSDE \eqref{BSDE CN} solution $(X', q', z')$ belongs to $\mathcal S^2_n \times  \mathcal H^2_{n \times d}\times \mathcal H^2_{n\times d}$ and  $\hat \mu$ is an $(\mathcal F_t ^0)$-adapted process with values in $\mathcal{P}_2(\mathbb{R}^n\times \mathbb{R}^{n\times d} \times \mathbb{R}^{n\times d})$, such that $\mathbb E \left[\int_0^T M_2(\hat \mu_t)^2 dt\right]<\infty$, the linear growth just established ensures that:
\[ \mathbb{E} \left[ \int_0^T (|\alpha'_t|^2+ |\beta'_t|^2)dt \right] \leq \bar{C}\mathbb{E} \left[\int_0^T (1 + M_2(\mu_t)^2+ |X'_t|^2 +|q'_t|^2+ |z'_t|^2) dt \right] < \infty. \]
Furthermore, the continuity of $\tilde{\sigma}$ and $\tilde{\sigma}_0$ ensures that the progressive measurability of $(X', q', z')$ is preserved. Thus, we deduce $(\alpha', \beta') \in \mathcal{A}$.

By construction, the state associated to such pair satisfies $X'_0 = a$ and $X'_T = \xi' \in \mathcal{D}$. By the definition of the costs $\ell$ and $\phi$, we have
\[ J(\alpha', \beta'; \hat \mu) = J(\xi'; \hat \mu) \quad \text{and} \quad J(\hat \alpha, \hat \beta; \hat \mu) = J(\hat \xi; \hat \mu). \]
Substituting these into \eqref{absurd_en}, we obtain $J(\alpha', \beta'; \hat \mu) < J(\hat \alpha, \hat \beta; \hat \mu)$, which contradicts the optimality of $(\hat \alpha, \hat \beta)$ in the original MFG problem \eqref{forward CN}. Thus, $\hat \xi$ must be optimal for the auxiliary problem.\\
\textit{Consistency condition:} 
    The consistency condition for the forward problem is $$\hat \mu_t = \mathcal{L}(\hat X_t, \hat \alpha_t, \hat \beta_t | \mathcal{F}_t^0).$$ 
    Since the processes $(\hat X, \hat \alpha, \hat \beta)$ and the flow $\hat \mu$ are identical in both formulations via the bijective mappings, the consistency condition in Definition \ref{def-auxMFG} is satisfied.
Therefore, $(\hat \xi, \hat \mu)$ is a solution for the auxiliary MFG.\\
\textbf{Direction 2: Auxiliary $\implies$ original.} 
Let $(\hat \xi, \hat \mu)$ be a solution of the auxiliary MFG with corresponding state $(\hat X, \hat q, \hat z)$ as in Definition \ref{def-auxMFG}. Going back to the original problem, we have $J(\hat \alpha, \hat \beta)=J(\hat \xi)$ for
\[ (\hat \alpha_t , \hat \beta_t ) = (\tilde \sigma, \tilde \sigma_0)(t, \hat X_t , \hat q_t , \hat z_t, \hat \mu_t), \quad t\in [0,T]. \]
Under assumptions (A1), (A2), (A4) and (A5), the inverse mappings $(\tilde \sigma, \tilde \sigma_0)$ are Lipschitz continuous in $(x, q, z)$ uniformly in $t$, which implies a linear growth condition (see the proof of Direction 1). Since the BSDE solution $(\hat X, \hat q, \hat z)$ belongs to $\mathcal S^2_n \times  \mathcal H^2_{n \times d}\times \mathcal H^2_{n\times d}$, it follows that $(\hat \alpha, \hat \beta) \in \mathcal{A}$.   With this admissibility established, we have $J(\hat \alpha, \hat \beta) \leq J(\alpha,\beta)$ for any admissible control pair $(\alpha, \beta) \in \mathcal{A}$ satisfying $X_T^{\alpha, \beta} \in \mathcal{D}$ a.s. Now, let
\[ \xi = X_T^{\alpha, \beta}, \quad q_t = \sigma(t, X_t, \alpha_t, \hat \mu_t), \quad z_t = \sigma_0(t, X_t, \beta_t, \hat \mu_t). \]
Under the  growth assumptions (A2) and (A4), and since $(X, \alpha, \beta)$ belongs to $\mathcal S^2_n \times  \mathcal H^2_{n \times d}\times \mathcal H^2_{n\times d}$, it follows that $(q, z) \in \mathcal H^2_{n \times d}\times \mathcal H^2_{n\times d}$. 
Consequently, we have $\xi \in L^2_a (\mathcal D)$. Therefore, the optimality condition in Definition \ref{def-MFG-constraint} is satisfied. The consistency condition therein is also trivially fulfilled and we can conclude that $(\hat \alpha, \hat \beta, \hat \mu)$ is a solution for the MFG with terminal constraint.
\end{proof}

\section{Existence of a MFG solution}
\subsection{FBSDE characterization of MFG solutions}

Let $\mu$ be given $(\mathcal F_t^0)$-adapted process with values in $\mathcal P_2(\mathbb R^n \times \mathbb R^{n\times d}\times \mathbb R^{n\times d})$ and $\mu^X$ be the flow of the corresponding first marginal laws. Let $a\in \mathbb R^n$ be a fixed initial condition for the state $X$. The paper by Ji and Zhou \cite{Ji Zhou} provides a characterization of $a$-optimal random variables in $L^2 (\mathcal D)$ via a stochastic maximum principle. We adapt their method to our setting. This will lead to necessary and sufficient conditions for an $a$-feasible $\hat \xi$ to be optimal, which are stated below. Their proofs are postponed to \ref{appendix}, where they are performed in full detail for the sake of completeness. Once $\hat \xi$ is given, $\hat q, \hat z$ are obtained by solving \eqref{BSDE CN} and subsequently $\hat \alpha, \hat \beta$ are found. 

Now, let $\hat \xi$ be an $a$-optimal random variable in $L^2 (\mathcal D)$, with associated state process $(\hat X,\hat q,\hat z)$. 
We define the following adjoint equation:
\begin{equation}
\label{adjoint equation CN}
    \begin{cases}
        dY_t=\left(\hat f_x(t)Y_t+h_0 \hat \ell_x(t)\right)dt+\left(\hat f_q(t) Y_t+h_0 \hat \ell_q(t) \right)dW_t+\left(\hat f_z(t)Y_t+h_0 \hat\ell_z(t)\right)dW_t^0\\
        Y_0=h_1
    \end{cases}
\end{equation}
where we recall that $\hat \ell_{p}(t)=\ell_{p}(t,\hat X_t, \hat q_t, \hat z_t, \mu_t)$ for $p\in \{x,q,z\}$.
By introducing the Hamiltonian as 
$$H(t, x, q, z, \mu,  y, h_0)=  y f(t, x, q, z, \mu)+h_0 \ell(t, x, q, z, \mu),$$
we can rewrite the adjoint equation \eqref{adjoint equation CN} as follows: 
\begin{equation*}
    \begin{cases}
      dY_t=H_xdt+H_qdW_t+H_zdW_t^0\\
        Y_0=h_1  
    \end{cases}
\end{equation*}
where we have omitted the variables $(t, \hat X_t,  \hat q_t, \hat z_t,  \mu_t, Y_t,h_0)$ in $H_x,H_q$ and $H_z$ to ease the notation. 

\begin{theo}
\label{Necessary Sufficient Condition CN}
Assume \emph{(A1)--(A5)}.
\begin{enumerate}
\item[(i)] (Necessary conditions) For a given \( (\mathcal{F}_t^0) \)-progressively measurable process \( \mu \) with values in \( \mathcal{P}_2(\mathbb{R}^n \times \mathbb{R}^{n\times d} \times \mathbb{R}^{n \times d}) \), if $\hat \xi$ is $a$-optimal for \eqref{backward CN} with $(\hat X,\hat q,\hat z)$ being the corresponding state as in \eqref{BSDE CN}, then there exist $h_1 \in \mathbb{R}^n$ and $h_0 \in \mathbb{R}$ with $h_0 \geq 0$ and $|h_0|+|h_1|\neq 0$  such that 
\begin{equation}
    \label{main condition CN}
        \big(Y_T+h_0\phi_x(\hat \xi,\mu_T^X)\big)\big(\xi-\hat \xi\big)  \geq0 \ a.s., \text{ for all } \xi \in L^2(\mathcal{D}),
    \end{equation}
    where $Y$ is the solution of the adjoint equation \eqref{adjoint equation CN} with initial condition $h_1$.
\item[(ii)] (Sufficient conditions) Suppose that the Hamiltonian is convex in  $(x, q,z) \in \mathbb R^n \times \mathbb R^{n \times d} \times \mathbb R^{n \times d}$ and that $\phi$ is convex in $x \in \mathbb R^n$. 
Let $\hat \xi \in L^2_a (\mathcal D)$. Assume that \eqref{BSDE CN} (for $\xi = \hat \xi$) is well-posed, that there exist $h_1 \in \mathbb{R}^n$ and $h_0 \in \mathbb{R}^+$ such that \eqref{adjoint equation CN} is well-posed, and
    \begin{equation}
    \label{optimality condition CN}
        \big( Y_T+h_0\phi_x(\hat \xi, \mu_T^X)\big) \big(\xi-\hat \xi \big)\geq0 \ a.s., \quad \textrm{for all }\xi \in L^2_a(\mathcal{D}),
    \end{equation}
then $\hat\xi$ is $a$-optimal.
\end{enumerate}
\end{theo}

By means of the necessary and sufficient conditions stated above we can obtain the following characterization of the auxiliary MFG solutions. Its proof is omitted since it is a straightforward consequence of those conditions. 

\begin{theo} 
\label{verification}
Suppose that $H(t,\cdot, \cdot, \cdot, \cdot, \mu, y, h_0)$ and $\phi(\cdot, \mu)$ are convex in $(x,q, z)$ and $x$ respectively. 
Then any pair $(\hat \xi, \hat \mu )$ is a solution for the auxiliary MFG with terminal state constraint as in Definition \ref{def-auxMFG} if the tuple $( \hat X,   \hat q,  \hat z,  \hat \mu,   \hat Y,  \hat {h}_0, \hat {h}_1, \hat {\xi})$, where $( \hat X,   \hat q,  \hat z,  \hat Y) \in \mathcal S^2 _n \times \mathcal H^2 _{n\times d} \times \mathcal H^2 _{n\times d} \times \mathcal S^2 _n$, $\hat \mu$ is an $(\mathcal F_t ^0)$-progressively measurable process such that $\mathbb E \left[\int_0^T M_2(\hat \mu_t)^2 dt\right]<\infty$, $( \hat h_0,  \hat h_1) \in \mathbb R \times \mathbb R^n$, and $\hat \xi \in L^2(\mathcal D)$, solves the FBSDE system
\begin{equation}
\label{FBSDE MFG CN}
\begin{cases}
    -dX_t=f(t,X_t, q_t, z_t, \mu_t)dt -q_tdW_t-z_tdW_t^0\\
      X_0 =a, \quad X_T=\xi\\
    dY_t=[f_x(t)Y_t+h_0 \ell_x(t)]dt+[f _q(t)Y_t+h_0 \ell_q(t)]dW_t+[f_z(t)Y_t+h_0 \ell_z(t)]dW_t^0\\
        Y_0=h_1\\
        h_0 > 0\\
    \big( Y_T+h_0\phi_x(\xi,\mu_T^X)\big) \big(\xi^\prime -\xi \big) \geq 0 \ a.s., \ \forall \xi^\prime \in L^2(\mathcal{D})
\end{cases}
\end{equation}
with
\[\hat{\mu}_t=\mathcal{L}(\hat X_t, \hat \alpha_t , \hat {\beta}_t \mid \mathcal F_t ^0),   \quad  t\in [0,T],\]
where 
$ (\hat \alpha_t , \hat \beta_t ) = (\tilde \sigma, \tilde \sigma_0)(t, \hat X_t , \hat q_t , \hat z_t , \hat \mu_t ).$
The reverse also holds true with the condition $h_0 >0$ in the FBSDE system replaced with the weaker $h_0 \geq 0$.
\end{theo}

\subsection{On the uniqueness of the MFG solution}
In this section we prove some uniqueness results for the auxiliary MFG under some further assumptions. Using the bijective correspondence between solutions of the auxiliary problem and those of the original problem, we deduce uniqueness of the original MFG solution. The working assumptions are the following:
\begin{enumerate}
    
\item[(C1)] The coefficients $b, \sigma, \sigma_0$ do not depend on the measure argument.

\item[(C2)] The coefficients $b, \sigma, \sigma_0$ are affine in the state and control variables, i.e.,
    \[ b(t,x,\alpha,\beta)=b_0(t)+B_x(t)x+B_\alpha(t)\alpha+B_\beta(t)\beta, \] 
    \[ \sigma(t,x,\alpha)=\sigma_1(t)+S_x(t)x+S_\alpha(t)\alpha, \] 
    and 
    \[ \sigma_0(t,x,\beta)=\sigma_2(t)+S_x^0(t)x+S_\beta^0(t)\beta, \] 
    for all $(t,x,\alpha,\beta) \in [0, T]\times \mathbb{R}^n \times \mathbb{R}^{n\times d} \times \mathbb{R}^{n\times d}$, where the coefficients
    \[ (b_0, B_x, B_\alpha, B_\beta, \sigma_1, S_x, S_\alpha, \sigma_2, S_x^0, S_\beta^0) \]
    are bounded measurable functions on $[0,T]$. Moreover, for every \(t \in [0,T]\), the matrices \(S_\alpha(t)\), \(S_\beta^0(t)\) are invertible and  the matrix-valued functions
$t \mapsto S_\alpha(t)^{-1},
t \mapsto S_\beta^0(t)^{-1}$
are bounded and measurable on \([0,T]\).

\item[(C3)]  The  running cost $\bar{\ell}$ has a separated structure of the form:
    \[ \bar{\ell}(t,x,\alpha,\beta,\mu) = \bar \ell_{0}(t,x,\mu) + \bar{\ell}_{1}(t,x,\alpha,\beta), \]
    where $\bar{\ell}$ is convex in $(x, \alpha, \beta)$ for all $(t, \mu)$. Furthermore, there exists a constant $C_0 > 0$ such that the following growth conditions hold:
    \[ |\bar{\ell}(t,x,\alpha,\beta,\mu)| \leq C_{0}\left(1+|x|+|\alpha|+|\beta|+M_{2}(\mu)\right)^{2}, \]
    \[ |\phi(x,\mu)| \leq C_{0}\left(1+|x|+M_{2}(\mu)\right)^{2}, \]
    for all $(t,x,\alpha,\beta,\mu) \in [0, T]\times \mathbb{R}^n \times \mathbb{R}^{n\times d} \times \mathbb{R}^{n\times d} \times \mathcal{P}_2(\mathbb{R}^n)$.

\item [(C4)] $\phi$ is
$\lambda$-convex in $x$, uniformly in $\mu$, i.e.
\[
\phi(x,\mu)-\phi(x',\mu)
\geq
\phi_x(x',\mu)(x-x')+\lambda |x-x'|^2,
\qquad x,x'\in\mathbb R^n,
\]
for some constant $\lambda>0$.
\item [(C5)]
The functions \( \bar \ell_0(t, \cdot, \cdot) \)  and \( \phi \) are monotone in the sense of Lasry and Lions (see the following Definition \ref{monotone}).
\end{enumerate}
\begin{defi}
\label{monotone}
    A function $\psi: \mathbb{R}^n \times \mathcal{P}_2(\mathbb{R}^n) \to \mathbb R$
is said to be \emph{monotone} (in the sense of Lasry and Lions), if for all \( \mu \in \mathcal{P}_2(\mathbb{R}^n) \), the function   \( x \mapsto \psi (x, \mu) \) is at most of quadratic growth uniformly in $\mu$,  and for all \( \mu, \mu' \in \mathcal{P}_2(\mathbb{R}^n) \), the following inequality holds:
    \[
    \int_{\mathbb{R}^n} \left( \psi(x, \mu) - \psi(x, \mu') \right) \, d(\mu - \mu')(x) \geq 0.
    \]
\end{defi}

Under the assumptions established above, we  derive several structural  properties of the coefficients associated with the auxiliary MFG problem, which will be useful to prove the uniqueness results.
\begin{lemma}
\label{assumC}
Under assumptions \textnormal{(C1)--(C3)}, the  driver $f$ and the  running cost $\ell$, are defined by
\[
f(t,x,q,z) \equiv -b\big(t, x, \tilde{\sigma}(t, x, q), \tilde{\sigma}_0(t, x, z)\big),
\]
\[
\ell(t, x, q, z, \mu) \equiv \bar{\ell}\big(t, x, \tilde{\sigma}(t, x, q), \tilde{\sigma}_0(t, x, z), \mu\big),
\]
for all $(t,x,q,z,\mu) \in [0,T] \times \mathbb{R}^n \times \mathbb{R}^{n\times d} \times \mathbb{R}^{n\times d} \times \mathcal{P}_2(\mathbb{R}^n)$, satisfy the following properties:
\begin{itemize}
    \item[\rm (i)] The driver $f$ is affine in $(x, q, z)$ and is independent of the measure argument $\mu$.
    \item[\rm (ii)] The running cost $\ell$ is  convex in $(x, q, z)$.
    \item[\rm (iii)] Setting $\ell_0(t,x,\mu) \equiv \bar{\ell}_0(t,x,\mu)$ and $\ell_1(t,x,q,z) \equiv \bar{\ell}_1\big(t, x, \tilde{\sigma}(t, x, q), \tilde{\sigma}_0(t, x, z)\big)$, the function $\ell$ has the separated structure
    \[ \ell(t, x, q, z, \mu) = \ell_0(t, x, \mu) + \ell_1(t, x, q, z). \]
    \item[\rm (iv)] There exists a positive constant $C > 0$ such that the following quadratic growth condition holds globally:
    \[ |\ell(t, x, q, z, \mu)| \leq C \left( 1 + |x| + |q| + |z| + M_2(\mu) \right)^2. \]
    \item[\rm (v)] The function $\ell_0(t, \cdot, \cdot)$ inherits the Lasry--Lions monotonicity property from $\bar{\ell}_0(t, \cdot, \cdot)$.
\end{itemize}
\end{lemma}

\begin{proof}
We prove each item using the structural assumptions provided.

\smallskip
\noindent
\textbf{(i)--(iii).} By assumption \textnormal{(C2)}, the inverse control functions $\alpha = \tilde{\sigma}(t,x,q)$ and $\beta = \tilde{\sigma}_0(t,x,z)$ are explicitly given by the affine relations:
\begin{equation}\label{eq:affine_controls}
\begin{aligned}
\tilde{\sigma}(t, x, q) &= S_\alpha(t)^{-1} \big( q - \sigma_1(t) - S_x(t)x \big), \\
\tilde{\sigma}_0(t, x, z) &= S_\beta^0(t)^{-1} \big( z - \sigma_2(t) - S_x^0(t)x \big).
\end{aligned}
\end{equation}
Substituting \eqref{eq:affine_controls} into the linear expression for $b$ from \textnormal{(C2)} yields:
\[
f(t,x,q,z) = -b_0(t) - B_x(t)x - B_\alpha(t)\tilde{\sigma}(t, x, q) - B_\beta(t)\tilde{\sigma}_0(t, x, z).
\]
Since $\tilde{\sigma}$ and $\tilde{\sigma}_0$ are affine in $(x,q)$ and $(x,z)$ respectively, $f$ is affine in $(x,q,z)$. Moreover, by assumption \textnormal{(C1)}, the coefficients $b, \sigma, \sigma_0$ do not depend on the measure argument $\mu$, which guarantees that the driver $f$ does not depend on $\mu$.

Similarly, substituting \eqref{eq:affine_controls} into the separated form of $\bar{\ell}$ given in \textnormal{(C3)} gives:
\[
\ell(t, x, q, z, \mu) = \bar{\ell}_0(t,x,\mu) + \bar{\ell}_1\big(t, x, \tilde{\sigma}(t, x, q), \tilde{\sigma}_0(t, x, z)\big).
\]
Defining $\ell_0 \equiv \bar{\ell}_0$ and $\ell_1(t, x, q, z) \equiv \bar{\ell}_1(t,x,\tilde{\sigma}(t,x,q),\tilde{\sigma}_0(t,x,z))$, we immediately obtain the desired separated structure.

\noindent
\textbf{(ii)} Let $(x_1, q_1, z_1),(x_2, q_2, z_2) \in \mathbb{R}^n \times \mathbb{R}^{n\times d} \times \mathbb{R}^{n\times d}$, and let $\lambda \in [0,1]$. Since the mappings \eqref{eq:affine_controls} are affine, we have:
\[
\tilde{\sigma}(t, x_\lambda, q_\lambda) = \lambda \tilde{\sigma}(t, x_1, q_1) + (1-\lambda)\tilde{\sigma}(t, x_2, q_2) \equiv \lambda \alpha_1 + (1-\lambda)\alpha_2,
\]
\[
\tilde{\sigma}_0(t, x_\lambda, z_\lambda) = \lambda \tilde{\sigma}_0(t, x_1, z_1) + (1-\lambda)\tilde{\sigma}_0(t, x_2, z_2) \equiv \lambda \beta_1 + (1-\lambda)\beta_2,
\]
where we set $x_\lambda = \lambda x_1 + (1-\lambda)x_2$, $q_\lambda = \lambda q_1 + (1-\lambda)q_2$, and $z_\lambda = \lambda z_1 + (1-\lambda)z_2$. Exploiting convexity of $\bar{\ell}$ in $(x,\alpha,\beta)$ from \textnormal{(C3)}, we deduce:
\begin{align*}
\ell(t, x_\lambda, q_\lambda, z_\lambda, \mu) &= \bar{\ell}\big(t, \lambda x_1 + (1-\lambda)x_2, \lambda \alpha_1 + (1-\lambda)\alpha_2, \lambda \beta_1 + (1-\lambda)\beta_2, \mu\big) \\
&\leq \lambda \bar{\ell}(t, x_1, \alpha_1, \beta_1, \mu) + (1-\lambda)\bar{\ell}(t, x_2, \alpha_2, \beta_2, \mu) \\
&= \lambda \ell(t, x_1, q_1, z_1, \mu) + (1-\lambda)\ell(t, x_2, q_2, z_2, \mu),
\end{align*}
which shows that $\ell$ is  convex in $(x,q,z)$.

\smallskip
\noindent
\textbf{(iv)}  Since the time-dependent coefficients in (C2) and the inverse matrices $S_\alpha (t)^{-1}$, $S_\beta^0(t)^{-1}$ are bounded and measurable in $t$, there exists a uniform constant $K > 0$ such that $|S_\alpha(t)^{-1}| + |S_x(t)| +|\sigma_1(t)| \leq K$, and similarly for the $\sigma_0$ coefficients. Applying the triangle inequality to \eqref{eq:affine_controls} yields the estimates
\[
|\tilde{\sigma}(t, x, q)| \leq K\big(|q| + K|x| + K\big) \leq K_1\big(1 + |x| + |q|\big),
\]
\[
|\tilde{\sigma}_0(t, x, z)| \leq K\big(|z| + K|x| + K\big) \leq K_1\big(1 + |x| + |z|\big),
\]
for a sufficiently large constant $K_1 > 0$. Substituting these estimates into the quadratic growth bound of $\bar{\ell}$ from \textnormal{(C3)} gives:
\begin{align*}

|\ell(t,x,q,z,\mu)| &\leq C_0 \left(1 + |x| + K_1(1 + |x| + |q|) + K_1(1 + |x| + |z|) + M_2(\mu)\right)^2 \\
&\leq C_0 \left((1+2K_1) + (1+2K_1)|x| + K_1|q| + K_1|z| + M_2(\mu)\right)^2 \\
& \leq C \left( 1 + |x| + |q| + |z| + M_2(\mu) \right)^2,
\end{align*}
for some constant $C>0$, which is exactly property (iv).\\
\textbf{(v)} This property is  obvious  since $\ell_0(t,x,\mu) \equiv \bar{\ell}_0(t,x,\mu)$ for all $(t,x,\mu)$.
\end{proof}
Now we  establish uniqueness of the optimal control for the auxiliary MFG problem. This result will play a key role in the analysis and it will also be used in the applications of Section \ref{sec:appl}.
\begin{prop}
\label{uniquexi}
Let assumptions (C1)-(C4) hold. Then the optimal control problem \eqref{auxoptimal} admits at most one minimizer in $L^2_a(\mathcal D)$.
\end{prop}
\begin{proof}
Let $\xi^1,\xi^2\in L^2_a(\mathcal D)$ be two optimal minimizers. We prove that 
\[\xi^1=\xi^2 \quad \text{a.s.}.\] 
Fix $\theta\in(0,1)$ and define
\[
\xi^\theta:=\theta \xi^1+(1-\theta)\xi^2.
\]
Since $\mathcal D$ is convex, under assumptions \textnormal{(C1)--(C2)} $f$ is affine (for the proof see Lemma \ref{assumC}) and $\xi^1,\xi^2\in L^2_a(\mathcal D)$, we first show that $\xi^\theta\in L_a^2(\mathcal D)$. Consequently, $L_a^2(\mathcal D)$ is convex.\\
Let $(X^i,q^i,z^i)$ be the solution of \eqref{BSDE CN} corresponding to
$\xi^i$, for $i=1,2$. Since $\xi^1,\xi^2$ are $a$-feasible, we have
$X_0^1=X_0^2=a$. Define
\[
X_t^\theta:=\theta X_t^1+(1-\theta)X_t^2,
\qquad
q_t^\theta:=\theta q_t^1+(1-\theta)q_t^2,
\qquad
z_t^\theta:=\theta z_t^1+(1-\theta)z_t^2.
\]
Using the affine structure of $f$, we obtain
\[
\theta f(t,X_t^1,q_t^1,z_t^1)
+(1-\theta)f(t,X_t^2,q_t^2,z_t^2)=
f(t, X_t^\theta, q_t^\theta, z_t^\theta).
\]
Therefore, by multiplying the BSDEs for $(X^1,q^1,z^1)$ and
$(X^2,q^2,z^2)$ by $\theta$ and $1-\theta$, respectively, and adding them, we see that $( X^\theta, q^\theta, z^\theta)$ is a unique solution to  \eqref{BSDE CN} with terminal condition $\xi^\theta$. Indeed,
\[
X_T^\theta=\theta X_T^1+(1-\theta)X_T^2= \theta \xi^1+(1-\theta)\xi^2
=\xi^\theta.
\]
In particular,
\[
X_0^\theta= \theta X_0^1+(1-\theta)X_0^2=\theta a+(1-\theta)a=a.
\]
Hence $\xi^\theta\in L^2_a(\mathcal D)$. Thus $L^2_a(\mathcal D)$ is convex.
We now prove convexity of the cost functional. Under assumption (C3)
$\ell$  is convex in $(x,q,z)$ (for the proof see Lemma \ref{assumC}),
\[
\begin{aligned}
\ell(t,X_t^\theta,q_t^\theta,z_t^\theta,\mu_t) 
\leq \theta \ell(t,X_t^1,q_t^1,z_t^1,\mu_t) +(1-\theta)\ell(t,X_t^2,q_t^2,z_t^2,\mu_t).
\end{aligned}
\]
Integrating over $[0,T]$ and taking expectations gives
\begin{equation}
\label{lconvex}
    \begin{aligned}
\mathbb E\left[\int_0^T
\ell(t,X_t^\theta,q_t^\theta,z_t^\theta,\mu_t)\,dt\right]
&\leq
\theta \mathbb E \left[\int_0^T
\ell(t,X_t^1,q_t^1,z_t^1,\mu_t)\,dt \right] \\
& \quad + (1-\theta)\mathbb E \left[\int_0^T
\ell(t,X_t^2,q_t^2,z_t^2,\mu_t)\,dt \right].
\end{aligned}
\end{equation}
On the other hand, by assumption (C4) since  $\phi$ is $\lambda$-convex in $x$, we have
\[
\phi(\xi^\theta,\mu_T^X) \leq \theta \phi(\xi^1,\mu_T^X)+
(1-\theta)\phi(\xi^2,\mu_T^X)-\lambda\theta(1-\theta)|\xi^1- \xi^2|^2.
\]
Taking expectations, we obtain
\begin{equation}
\label{phiconvex}
\mathbb E[\phi(\xi^\theta,\mu_T^X)] \leq \theta \mathbb E[\phi(\xi^1,\mu_T^X)] +(1-\theta)\mathbb E[\phi(\xi^2,\mu_T^X)]
-\lambda\theta(1-\theta)\mathbb E \left[|\xi^1-\xi^2|^2 \right].
\end{equation}
Adding the  two inequalities \eqref{lconvex} and  \eqref{phiconvex} yields
\[
J(\xi^\theta)\leq\theta J(\xi^1)+(1-\theta)J(\xi^2)-
\lambda\theta(1-\theta)\mathbb E \left[ |\xi^1-\xi^2|^2 \right].
\]
Since $\xi^1$ and $\xi^2$ are both minimizers,
\[
J(\xi^1)=J(\xi^2)=\inf_{\xi\in L^2_a(\mathcal D)}J(\xi):=\inf J.
\]
Moreover, since $\xi^\theta\in L^2_a(\mathcal D)$, we also have
\(
J(\xi^\theta)\geq \inf J
\).
Therefore,
\[
\inf J\leq J(\xi^\theta)  \leq\inf J-\lambda\theta(1-\theta)\mathbb E \left[|\xi^1-\xi^2|^2 \right].
\]
Hence
$\mathbb E\left[|\xi^1-\xi^2|^2\right]=0$.
Consequently, $\xi^1=\xi^2$  a.s., and the problem admits at most one minimizer in  $L_a^2(\mathcal D)$.
\end{proof}
\begin{rem}\label{uniq-appl}
In the applications of Section 4,  the terminal cost functional
\[
\phi(x,\bar{\mu}) = -x + \frac{\gamma}{2}(x - \theta\bar{\mu})^2
\]
is uniformly $\lambda$-convex in $x \in \mathbb{R}$, uniformly in $\bar{\mu}$, with any $\lambda \in (0, \gamma / 2]$. 
\end{rem}

The following  proposition yields uniqueness of the mean field flow under the Lasry-Lions monotonicity condition. Despite this condition is not satisfied in the applications of Section 4 (see Remark \ref{Lasry-appl} below), we nevertheless include it as it provides a general uniqueness criterion within the present framework, which is of independent theoretical interest. Since its proof uses the same arguments as in the classical case without constraints and without common noise  (see  \cite[Theorem 3.29, p.169]{Carmona Delarue I}), we omit it.

\begin{prop}
\label{Lasry}
 Let assumptions \textnormal{(C1)--(C3), (C5)} hold,  and  suppose that for any given $(\mathcal{F}_t^0)$-adapted process $\hat \mu^X =(\hat \mu^X _t)_{0\leq t\leq T}$ with values in  $\mathcal{P}_2(\mathbb{R}^n) $, the optimal control problem 
 \begin{equation} \label{eq:opt_problem}
\inf_{\xi \in L^2_a(\mathcal{D})} \, \mathbb{E} \Bigg[  \int_0^T \ell(t, X_t, q_t, z_t ,  \hat \mu_t^X) \, dt + \phi(\xi, \hat \mu^X_T) \Bigg],
\end{equation}
subject to
\begin{equation}
\begin{cases}
    -dX_t =  \, f(t, X_t, q_t, z_t) \, dt - q_t \, dW_t-z_tdW_t^0\\  
X_0 =  \, a, \quad X_T = \xi,
\end{cases}
\end{equation}
admits a unique minimizer \( \hat \xi\in L^2_a(\mathcal{D})\). Let $( \hat X, \hat q, \hat z )$ denote the corresponding optimal state. Then, there exists at most one flow $\hat \mu^X=(\hat \mu_t^X)_{0\leq t\leq T}$ such that
\begin{equation}
\label{cons}
  \hat \mu_t^X=\mathcal{L}(\hat X_t|\mathcal{F}_t^0), \quad t \in [0, T]. 
\end{equation}
\end{prop}

\begin{rem}\label{Lasry-appl} 
We observe that in the applications of Section 4, the terminal cost functional
\[
\phi(x,\mu) = -x + \frac{\gamma}{2}\left(x - \theta \int_{\mathbb R}x\,d\mu(x)\right)^2
\]
is not monotone in the sense of Lasry-Lions. Consequently, Proposition \ref{Lasry} is not applicable. Hence, in order to get uniqueness we will rely on a contraction argument instead.
\end{rem}

A direct consequence of the previous two propositions is the following corollary, establishing a uniqueness result for the original MFG with constraint.
\begin{cor}
\label{uniquealpha}
Under the assumptions \textnormal{(C1)--(C5)}  there exists at most one solution $(\hat{\alpha}, \hat{\beta}, \hat{\mu})$ to the original MFG with terminal state constraint (Definition \ref{def-MFG-constraint}). 
\end{cor}

\begin{proof}
By Proposition \ref{equivalence}, there is a bijective correspondence between solutions $(\hat{\alpha}, \hat{\beta}, \hat{\mu})$ of the original MFG and solutions $(\hat{\xi}, \hat{\mu})$ of the auxiliary MFG. As established in Proposition \ref{equivalence} (based on the Lipschitz continuity and linear growth of the inverse maps $(\tilde{\sigma}, \tilde{\sigma}^0)$), the square integrability of the auxiliary state $(\hat{X}, \hat{q}, \hat{z}) \in \mathcal{S}^2_n \times \mathcal{H}^2_{n \times d} \times \mathcal{H}^2_{n \times d}$ automatically guarantees the admissibility of the reconstructed controls $(\hat{\alpha}, \hat{\beta}) \in \mathcal{A}$.

Regarding uniqueness, Proposition \ref{Lasry} yields the uniqueness of the equilibrium state flow $\hat{\mu}^X$. For a given state flow, the optimal terminal variable $\hat{\xi}$ is unique by Proposition \ref{uniquexi}, and the associated BSDE solution $(\hat{X}, \hat{q}, \hat{z})$ is unique by well-posedness. Since the inverse maps $(\tilde{\sigma}, \tilde{\sigma}^0)$ reconstruct the optimal controls $(\hat{\alpha}, \hat{\beta})$ uniquely, the full conditional law $\hat{\mu}_t=\mathcal{L}(\hat X_t,\hat\alpha_t,\hat\beta_t\mid \mathcal{F}_t^0)$ is uniquely determined for every $t$, and hence the uniqueness of the full conditional law $\hat{\mu}$ follows. Thus, the solution to the original MFG is unique.
\end{proof}

\subsection{Existence of solutions for the FBSDE system in the linear case}

We consider the FBSDE system \eqref{FBSDE MFG CN} with $\ell = 0$, where the coefficients are linear and independent of the measure variable, except for $\phi$, which depends on the measure through its conditional average. Since in the FBSDE system \eqref{FBSDE MFG CN} we have $h_0>0$, we may normalize the adjoint equation and the terminal variational inequality by dividing by $h_0$. Hence, without loss of generality, we take $h_0=1$. We will prove existence and uniqueness of solutions of the FBSDE system in the one-dimensional case,  where $\mathcal{D} =[c, \infty)$, for some $c \in \mathbb R$,  together with the  condition 
\begin{equation}
\label{consistency}
    m_T=\mathbb{E}(X_T\mid \mathcal F_T ^0).
\end{equation}
More specifically, to establish the existence and uniqueness result we need some more technical assumptions:
\begin{itemize}
    \item[(A6)] The drift  function $f$ is linear in $(x, q, z)$, i.e.,
    
    \[ f(t,x,q,z) = b_0(t)x + qc_0(t) + z \tilde c_0(t), \quad (t,x,q,z) \in [0, T]\times \mathbb{R}\times\mathbb{R}^{ d}\times \mathbb{R}^{d} ,\]
  where
    \[ (b_0 , c_0 ,\tilde c_0): [0,T] \to \mathbb R \times \mathbb R^{d} \times \mathbb R^{d}\] are bounded measurable functions.
    \item[(A7)] We assume that the function $\phi_x(x, m)$ is linear in both arguments and satisfies
\[
\phi_x(x, m) = k_0+k_1 x + k_2 m, \quad (x,m)\in \mathbb R^2,
\]
for some constants $k_i$, $i=0,1,2$, satisfying $k_1 >0$, $k_2 \leq 0$, and $k_1 + k_2 >0$.
  \end{itemize}
 Under these assumptions, the mean field FBSDE system \eqref{FBSDE MFG CN}  reduces to the following:
\begin{equation}
\label{FBSDE_Cont CN l}
    \begin{cases}
    dX_t=-[b_0(t)X_t+q_tc_0(t)^\top +z_t\tilde{c}_0(t)^\top]dt +q_tdW_t+z_tdW_t^0\\
      X_0 =a, \quad X_T=\xi\\
  
    dY_t=b_0(t)Y_tdt+c_0(t)Y_tdW_t
    +\tilde{c}_0(t)Y_tdW_t^0\\
   Y_0=h_1\\

    \big( Y_T+ \phi_x(\xi, m_T)\big)\big( \xi'-\xi \big) \geq 0 \ a.s., \ \forall \xi' \in L^2([c, \infty)).        
   \end{cases}
\end{equation}
We will be using the notation
\[\mathcal{E}_t(c,c') := \exp\left( \int_0 ^t (c(s)dW_s + c'(s)dW^0 _s)  - \frac{1}{2}\int_0 ^t (|c(s)|^2 + |c'(s)|^2)ds \right)\] for the Dol\'eans-Dade exponential of a Wiener integral with respect to the $2d$-dimensional Brownian motion $(W,W^0)$, where $(c,c'):[0,T] \to \mathbb R^d \times \mathbb R^d$ are measurable and bounded functions.
\begin{theo}
\label{EXUN} 
Let $a > c \ e^{\int_0^T b_0(t)\, dt}$.    Under Assumptions  (A1)-(A7), the FBSDE system \eqref{FBSDE_Cont CN l} has a unique  solution.
\end{theo}
\begin{proof}
  The adjoint SDE can be solved for $t \in [0,T]$ explicitly as
\begin{equation}
\label{solution of forward CN l}
Y_t= h_1 e^{\int_0^t b_0(s)\, ds} \mathcal{E}_t\left(c_0 ,\tilde{c}_0 \right).
\end{equation} 
Moreover, since $b_0, c_0, \tilde{c}_0$ are bounded, the stochastic exponential $\mathcal E(c_0,\tilde c_0)$ and its reciprocal have finite moments of every order. By Doob's inequality,
\[ \mathbb E\left[\sup_{t\in [0,T]} |Y_t|^p\right] \leq C_p |h_1|^p \mathbb E\left[ \mathcal E_T (c_0,\tilde c_0)^p\right] < \infty,\]
for some constant $C_p >0$, so that $Y \in \mathcal S^p$ for all $p \geq  2$. In particular, $Y^{(1)}, 1/Y^{(1)}\in \mathcal S^p $ for all $p \geq  2$.

Defining the set $\Omega_0= \{ \omega \in \Omega \ | \ \xi(\omega)=c\}$, the condition 
\[\big ( Y_T+ \phi_x(\xi, \mu)\big)\big( \xi'-\xi \big) \geq 0, \text{a.s.}, \ \forall \xi' \in L^2([c, \infty))\] 
can be written as 
\begin{equation}
\label{eq1main}
    Y_T+ \phi_x(\xi, m_T)=0 \ \text{a.s. on} \ \Omega_0^c,
\end{equation}
\begin{equation}
\label{eq2main}
   Y_T+ \phi_x(c, m_T) \geq 0 \ \text{a.s. on} \ \Omega_0. 
\end{equation}
From  \eqref{eq1main} and \eqref{eq2main}  we obtain  
\begin{equation}
\label{optimality general}
    \xi=\left(\dfrac{-Y_T- \phi_x(c, m_T)}{k_1}\right)^+ +c.
\end{equation}
We prove existence and uniqueness by separating the problem into two steps. First, for a fixed \(h_1\), we solve for the conditional mean \(m_T\). Second, we determine \(h_1\) from the initial constraint 
\(X_0=a\).
Let us introduce the following convenient notation, stressing the dependence on $h_1$:
\[ Y_t ^{(h_1)} := Y_t = h_1 e^{\int_0^t b_0(s) ds} \mathcal{E}_t\left( c_0,\tilde{c}_0\right), \]
so that $Y_t ^{(h_1)} = h_1 Y_t ^{(1)}$ , where $Y_t ^{(1)}$  does not depend on $h_1$. Moreover, for any process $Z \neq 0$ a.s., and any time instants $s \leq t$, let $Z_{s,t}:= Z_t /Z_s$. Consequently, we clearly have
\begin{equation}\label{eq:factorY} Y^{(h_1)} _t = h_1 Y_t ^{(1)}=h_1 Y_s^{(1)} Y_{s,t}^{(1)}, \end{equation}
where $Y_{s,t}^{(1)}$ is independent of $\mathcal F_s$, so in particular of $Y_s ^{(1)}$.
Therefore, under Assumption (A7) the representation \eqref{optimality general} becomes
\begin{equation}
\label{xirep}
    \xi=c+\left(\frac{-Y_T^{(h_1)}-k_0-k_1c-k_2m_T}{k_1}\right)^+.
\end{equation}
Hence, for fixed \(h_1\), the conditional mean \(m_T\) must satisfy
\begin{equation}
\label{conditionalmean}
 m_T=c+\mathbb E\left[\left(\frac{-Y_T^{(h_1)}-k_0-k_1c-k_2m_T}{k_1}\right)^+\;\middle|\;\mathcal F_T^0\right].
\end{equation}
For fixed \(h_1\), define the map
\[
\Gamma_{h_1}:L^2(\mathcal F_T^0)\longrightarrow L^2(\mathcal F_T^0)
\]
by
\begin{equation}\label{def:Gamma}
\Gamma_{h_1}(m):=c+\mathbb E\left[\left(\frac{
-Y_T^{(h_1)}-k_0-k_1c-k_2m}{k_1}\right)^+\;\middle|\;\mathcal F_T^0\right].
\end{equation}
We show that \(\Gamma_{h_1}\) is a contraction. Let \(m,m'\in L^2(\mathcal F_T^0)\). Since the map \(x\mapsto x^+\) is \(1\)-Lipschitz, we have
\begin{equation}
\label{contractionestimate}
  \begin{aligned}
\left\| \Gamma_{h_1}(m)-\Gamma_{h_1}(m')\right\|_{L^2}
&\le\left\|\left(\frac{-Y_T^{(h_1)}-k_0-k_1c-k_2m}{k_1}\right)^+-\left(\frac{-Y_T^{(h_1)}-k_0-k_1c-k_2m'}{k_1}\right)^+\right\|_{L^2}\\
&\le \frac{|k_2|}{k_1} \|m-m'\|_{L^2}.
\end{aligned}  
\end{equation}
Under assumption (A7), $k_1>0$, $k_2\leq 0$ and $k_1+k_2>0$,
therefore $|k_2|<k_1,$ and so
$\frac{|k_2|}{k_1}<1$.
Thus, \(\Gamma_{h_1}\) is a contraction on \(L^2(\mathcal F_T^0)\). By the Banach fixed-point theorem, for every fixed \(h_1\), there exists a unique $m_T^{(h_1)}\in L^2(\mathcal F_T^0)$, such that
\[
m_T^{(h_1)}=\Gamma_{h_1}\big(m_T^{(h_1)}\big).
\]
Once \(m_T^{(h_1)}\) is determined, define
\begin{equation}
\label{xih1}
\xi^{(h_1)}:=c+\left(\frac{-Y_T^{(h_1)}-k_0-k_1c-k_2m_T^{(h_1)}}{k_1}\right)^+.
\end{equation}
In order to determine \(h_1\) from the initial condition \(X_0=a\), we observe first that, since $X_0=\mathbb E\big[\xi^{(h_1)}Y_T^{(1)}\big]$,
the parameter \(h_1\) must satisfy
$a=\mathbb E [\xi^{(h_1)}Y_T^{(1)}]$,
or, equivalently, it must be a zero of the function
\[
G(h_1):=\mathbb E\left[\xi^{(h_1)}Y_T^{(1)} \right]-a.
\]
\begin{claim}\label{claim:h1}
The equation \(G(h_1)=0\) admits a unique solution.
\end{claim}
\noindent The proof of Claim \ref{claim:h1} is rather technical and is postponed to  \ref{appendix}.
 For the unique solution \(h_1\), the corresponding fixed point
\(m_T^{(h_1)}\) is also uniquely determined, and so
\[
\xi=\xi^{(h_1)}=c+\left(\frac{
-Y_T^{(h_1)}-k_0-k_1c-k_2m_T^{(h_1)}
}{k_1}
\right)^+
\]
is uniquely determined.

It remains to construct processes \((X,q,z)\) solving the first equation in \eqref{FBSDE_Cont CN l} and to verify that \((X,q,z)\in \mathcal S^2\times\mathcal H^2_{1\times d}\times\mathcal H^2_{1\times d}\). We proceed directly from the terminal variable. The fixed-point argument for \(m_T^{(h_1)}\) works in \(L^p(\mathcal F_T^0)\), for every \(p\ge2\), with the same contraction constant \(|k_2|/k_1<1\). Since \(Y_T^{(1)}\) has moments of all orders, it follows that \(m_T^{(h_1)}\in L^p(\mathcal F_T^0)\), \(\xi^{(h_1)}\in L^p(\mathcal F_T)\), and \(\xi^{(h_1)}Y_T^{(1)}\in L^p(\mathcal F_T)\), for every \(p\ge2\). Define
\[
M_t:=\mathbb E\big[\xi^{(h_1)}Y_T^{(1)}\mid \mathcal F_t\big], \qquad t\in[0,T].
\]
By Doob's inequality, \(M\in\mathcal S^p\) for every \(p\ge2\). Set
\[
X_t:=\frac{M_t}{Y_t^{(1)}}, \qquad t\in[0,T].
\]
Since \(1/Y^{(1)}\in\mathcal S^p\) for every \(p\ge2\), H\"older's inequality gives \(X\in\mathcal S^p\) for every \(p\ge2\). Moreover, \(X_T=\xi^{(h_1)}\) and, by the definition of the zero of \(G\),
\[
X_0=\mathbb E\big[\xi^{(h_1)}Y_T^{(1)}\big]=a.
\]
By the martingale representation theorem, there exist predictable processes \((\widetilde q,\widetilde z)\) such that
\[
M_t=a+\int_0^t\widetilde q_s\,dW_s+\int_0^t\widetilde z_s\,dW_s^0, \qquad t\in[0,T].
\]
The Burkholder-Davis-Gundy and Doob inequalities yield \((\widetilde q,\widetilde z)\in\mathcal H^p_{1\times d}\times\mathcal H^p_{1\times d}\), for every \(p\ge2\). Applying It\^o's formula to \(X_t=M_t/Y_t^{(1)}\), and using
\[
dY_t^{(1)}=b_0(t)Y_t^{(1)}dt+c_0(t)Y_t^{(1)}dW_t+\widetilde c_0(t)Y_t^{(1)}dW_t^0,
\]
we obtain
\[
dX_t=\big(-b_0(t)X_t-q_tc_0(t)^\top-z_t\widetilde c_0(t)^\top\big)dt+q_t\,dW_t+z_t\,dW_t^0,
\]
where
\[
q_t=\frac{\widetilde q_t}{Y_t^{(1)}}-X_tc_0(t),\qquad
z_t=\frac{\widetilde z_t}{Y_t^{(1)}}-X_t\widetilde c_0(t).
\]
Since \(c_0\) and \(\widetilde c_0\) are bounded, while \(X\in\mathcal S^p\), \(1/Y^{(1)}\in\mathcal S^p\), and \((\widetilde q,\widetilde z)\in\mathcal H^p\times\mathcal H^p\) for sufficiently large \(p\), H\"older's inequality gives \(q,z\in\mathcal H^2_{1\times d}\). Thus \((X,q,z)\) solves the first equation in \eqref{FBSDE_Cont CN l} and belongs to the required space.
\end{proof}

\begin{rem}
\label{remExun}
Similar results as in Theorem \ref{EXUN} (and its proof) hold also in the two ``extreme'' cases with no common noise and with only common noise with the corresponding modifications. They will be used in the applications of the next section.
\end{rem}

\section{Application to optimal investment for risk-neutral agents with quadratic relative performance criterion}\label{sec:appl}

We consider a MFG of optimal investment for risk-neutral investors with quadratic relative performance criterion and a terminal wealth constraint. Specifically, the investors require their terminal wealth to stay above a certain fixed constant threshold. This variant of the optimal investment problem received a lot of attention in the past mathematical finance literature; we refer especially to the papers by Cox and Huang \cite{Cox Huang}, El Karoui et al. \cite{ElK-J-L}, Korn \cite{Korn} and Tepl\'a \cite{Tepla}, to cite just a few. The idea is that the representative investor fears big losses, hence she incorporates the tolerated minimum wealth in her criterion. For the sake of simplicity, here we assume that the minimum level is constant and applies only at maturity. Moreover, in order to match our general setting, we consider investors who are willing to maximize their expected terminal wealth and, at the same time, compare their performance to their competitors. The latter is described by some quadratic relative performance criterion appearing in their objective.  
 Optimal investment problems under relative performance criteria have been studied rather extensively in the literature, i.e., by Espinosa, Touzi in \cite{Espinosa Touzi} and by Lacker, Soret, Zariphopoulou in  \cite{Lacker S}, \cite{Lacker Z} in the MFG setting.\medskip
 
We start by describing the $N$-player optimal investment game. Then, we will pass to MFG for $N \to \infty$, which we will set up in a more rigorous way. We assume that the interest rate is zero, so that the riskless asset is identically equal to one, and that each agent $i = 1,\ldots, N$ invests in a risky asset, whose price at time $t$ is denoted $S_t ^i$ and follows the dynamics
\[ \frac{dS^i _t}{S_t ^i} = bdt + \sigma dW^i + \sigma^0 dW_t ^0, \quad t\in [0,T], \]
where $W^0, W^1, \ldots, W^{N}$ are $N+1$ independent Brownian motions, $b\in \mathbb R$ and both $\sigma, \sigma^0$ are nonnegative constants with $\sigma + \sigma^0 >0$. The information of the agents is modelled by the filtration generated by all these Brownian motions and satisfying the usual conditions.
Each agent $i$ trades using a self-financing strategy $(\alpha_t^i)_{t \in [0,T]}$, which represents the proportion of wealth invested in the risky asset. Therefore, agent $i$'s wealth, denoted $X_t ^i$, evolves as
$$dX_t^i=\alpha_t^i X_t^i(bdt+\sigma dW_t^i + \sigma^0 dW_t ^0) , \ X_0^i=a,$$
where $a > 0 $ is the initial wealth. 

Each agent aims at maximizing her expected terminal wealth while considering her performance relative to the average terminal wealth of all investors. Hence, given the terminal wealth of all agent $i$'s competitors  $(X_T ^j)_{j \neq i}$, her optimization problem is
\[ \sup_{\alpha^i} \mathbb E\left[ X_T^i -\frac{\gamma}{2}\left(X_T^i - \dfrac{\theta}{N-1}\sum_{j\neq i}^N X_T^j \right)^2  \right], \]
under the constraint $X_T^i \geq c$, for a given threshold $0 < c < a$; where $\gamma>0$ is the agent's risk aversion coefficient and $\theta \in [0,1)$ models her relative performance concern.

Our goal is to solve the corresponding MFG using the methods developed in the previous sections with $\mathcal D = [c,+\infty)$. Therefore, we will consider only the two cases $\sigma^0=0$ (no common noise) and $\sigma =0$ (pure common noise). The first case corresponds to a situation where each investor is specialized in a specific asset, while in the second one all agents are investing in the same asset. In both cases below, $(\mathcal F_t)$ will denote the filtration generated by two independent real-valued Brownian motions $W^0,W$ defined in a common complete probability space $(\Omega, \mathcal F, \mathbb P)$. The filtration is assumed to be completed with all $\mathbb P$-null sets, so that the usual conditions are fulfilled. 
 In the intermediate situation where both 
$\sigma^0$ and $\sigma$ are not zero, the method developed in the paper can be used if we allow for different controls in the noise terms, which is not a natural assumption in this application. 

\paragraph{The case with no common noise.} In this case, i.e., $\sigma^0 =0$, the corresponding MFG is given by the following two-step procedure:
\begin{enumerate}

\item Fix a constant $ m_T \in \mathbb R$ and solve the following maximization problem
\[ \sup_{\alpha \in \mathcal A} \mathbb E\left[ X_T -\frac{\gamma}{2}(X_T - \theta m_T)^2  \right], \]
subject to
\begin{equation}
\label{sde oi}
    dX_t = \alpha_t X_t (b dt + \sigma dW_t ), \quad X_0=a, \quad X_T \geq c,
\end{equation}

where $\mathcal A$ is the set of all admissible investment strategies, i.e., all $(\mathcal F_t)$-progressively measurable processes $\alpha$ such that $\mathbb E\left[\int_0 ^T |\alpha|^2 dt \right] < + \infty$.

\item Find a fixed point $\hat m_T = \mathbb E[\hat X_T]$, where $\hat X$ is the state associated to the optimal control $\hat \alpha$ obtained in 1.
\end{enumerate}
Before moving to the resolution of the MFG, we need to fix some preliminary notation. Let $\rho := b/\sigma$ denote the market price of risk, let $Y^{(1)} := \mathcal E(-\rho,0)$.
\begin{prop}\label{prop:application1}
Let $(\hat h_1, \hat m_T) \in \mathbb R \times \mathbb R$ be the unique solution of the coupled system 
\begin{equation} \label{eq:h1}
\begin{cases}
 m_T = c + \mathbb{E} \left[ \left( \frac{1}{\gamma}-c+\theta m_T - \frac{h_1}{\gamma} Y_T^{(1)} \right)^{+} \right]\\
    G(h_1):= c -a+ \mathbb{E} \left[ Y_T^{(1)} \left( \frac{1}{\gamma}-c+\theta m_T - \frac{h_1}{\gamma} Y_T^{(1)} \right)^{+} \right]=0,\\
\end{cases}
\end{equation} 
 and let
\[ \hat g(t,y) := \mathbb{E}\!\left[ yY^{(1)}_{t,T} \left(\frac{1}{\gamma}-c+\theta \hat{m}_T- \frac{\hat h_1 }{\gamma}yY^{(1)}_{t,T} \right)^+ \right], \quad (t,y)\in [0,T]\times \mathbb R^+.\]
 The unique MFG solution is given by the couple $(\hat \alpha, \hat m_T)$ where
$$\hat\alpha_t=-\frac{\rho}{\sigma}
\dfrac{Y^{(1)}_t\partial_{y}\hat g(t, Y^{(1)}_t) - \hat g(t,Y^{(1)}_t)}{cY^{(1)}_t+\hat g(t, Y^{(1)}_t)}, \quad t\in [0,T).$$
\end{prop}

\begin{proof}
First, we reformulate the control problem in terms of the amount $\pi_t$ invested
in the risky asset at time $t$, i.e., by letting
\[
\pi_t:=\alpha_tX_t, \quad t\in [0,T],
\]
the wealth dynamics \eqref{sde oi} becomes
\[
dX_t=\pi_t(b\,dt+\sigma\,dW_t),\qquad X_0=a.
\]
In this reformulated problem the diffusion coefficient \(\pi \sigma\) satisfies assumption (A4), so that we can use the approach developed in the previous sections and set
\[
q_t:=\sigma\pi_t,
\qquad \rho:=\frac b\sigma,
\]
yielding
\[
dX_t=q_t(\rho\,dt+dW_t),\qquad X_0=a.
\]
After solving the reformulated problem, the original control will be recovered via
\[
\alpha_t=\frac{\pi_t}{X_t}=\frac{q_t}{\sigma X_t},
\]
whenever \(X_t>0\), which we will verify to be true at the equilibrium. Using Theorem \ref{verification}, we are reduced to solving the following FBSDE system:
\begin{equation}
\label{FBSDE general 1 oi}
\begin{cases}
    dX_t=  q_t (\rho dt+ dW_t)\\
    X_0=a, \quad X_T=\xi\\
     dY_t=-\rho Y_t dW_t\\
    Y_0=h_1\\

        \big(Y_T-1+\gamma(\xi -\theta m_T)\big)\big( \xi'-\xi \big) \geq 0 \ \text{a.s.}, \ \forall \xi' \in L^2([c,+\infty)) .
\end{cases}
\end{equation}
Applying the same computations as in the first part of Theorem \ref{EXUN} (see also Remark \ref{remExun}) with the parameters
\[
c_0(t) = -\rho, \quad b_0(t) = 0, \quad \tilde{c}_0(t) = 0, \quad k_0 = -1, \quad k_1 = \gamma, \quad k_2 =-\gamma\theta,
\]
we have that the terminal condition of system \eqref{FBSDE general 1 oi} takes the form 
\begin{equation}
\label{terminal_X}
\xi=c + \left( \frac{1}{\gamma} - c + \theta m_T - \frac{h_1}{\gamma} Y_T^{(1)} \right)^{+}.
\end{equation}
By the same arguments as in the proof of Theorem \ref{EXUN}, we deduce that there exists a unique  pair $(\hat h_1, \hat m_T)$  which solves the following  coupled system
\begin{align*}
m_T &= c + \mathbb{E} \left[ \left( \frac{1}{\gamma} - c + \theta m_T - \frac{h_1}{\gamma} Y_T^{(1)} \right)^{+} \right], \\
a &= c + \mathbb{E} \left[ Y_T^{(1)} \left( \frac{1}{\gamma} - c + \theta m_T - \frac{h_1}{\gamma} Y_T^{(1)} \right)^{+} \right].
\end{align*}
Since \(X Y^{(1)}\) is a martingale, we have
\[
X_t = \frac{1}{Y_t^{(1)}}
\mathbb{E}\left[ X_T Y_T^{(1)} \mid \mathcal{F}_t \right].
\]
Using the terminal condition
\[X_T=c+\left(\frac{1}{\gamma} - c + \theta m_T-\frac{h_1}{\gamma} Y_T^{(1)}\right)^+,
\]
we obtain
\[
X_t=c+\frac{1}{Y_t^{(1)}}\mathbb{E}\left[Y_T^{(1)}
\left(\frac{1}{\gamma} - c + \theta m_T-\frac{h_1}{\gamma} Y_T^{(1)}\right)^+\;\middle|\;\mathcal{F}_t\right].\]
We notice that, since $Y^{(1)}_T = Y^{(1)}_t Y^{(1)}_{t,T}$, where $Y^{(1)}_t$ and $Y^{(1)}_{t,T}$ are independent, we can express $X$ as follows
\[
X_t = c + \frac{1}{Y^{(1)}_t} \hat g(t, Y^{(1)}_t),
\]
where
\[
\hat g(t,y) := \mathbb{E}\!\left[ yY^{(1)}_{t,T} \left(\frac{1}{\gamma} - c + \theta \hat m_T- \frac{\hat h_1 }{\gamma}yY^{(1)}_{t,T} \right)^+ \right].\]
Observe that $Y^{(1)}_{t,T}$ follows a lognormal distribution with a smooth density on $(0,\infty)$. Hence, $\hat g(t,\cdot)$ is a lognormal convolution of a payoff with at most quadratic growth, so $\hat g \in C^{1,2}([0,T) \times (0, \infty))$. Since $Y^{(1)} > 0$, we can safely apply It\^o's formula to $X_t = c+  \hat g(t,Y^{(1)}_t)/Y^{(1)}_t$, $t<T$, yielding
\[
dX_t = \left(\partial_t \kappa (t,Y^{(1)}_t) + \frac12 \rho^2 (Y_t ^{(1)})^2 \, \partial_{yy} \kappa (t,Y^{(1)}_t) \right)dt - \rho Y^{(1)}_t \, \partial_y \kappa (t,Y^{(1)}_t) dW_t, \quad t<T,
\]
where $\kappa (t,y) :=y^{-1}\hat g(t,y)$. Therefore, we can identify the process $q$ as
\[
q_t=-\rho Y_t^{(1)} \partial_y \kappa (t,Y^{(1)}_t)=-\rho
\left(
\partial_{y}\hat g(t,Y^{(1)}_t) - \frac{1}{Y^{(1)}_t} \hat g(t,Y^{(1)}_t)\right),\quad   t<T.
\]
Finally, since
\[
X_t=c+\frac{\hat g(t,Y_t^{(1)})}{Y_t^{(1)}}>0,
\]
\(\hat\alpha_t\) is well defined, and recalling that
\(q_t=\alpha_tX_t\sigma\), we obtain the equilibrium strategy
\begin{equation}
\label{opt-alpha}
\hat \alpha_t= -\frac{\rho}{\sigma}
\dfrac{Y^{(1)}_t\partial_{y}\hat g(t,Y^{(1)}_t) - \hat g(t,Y^{(1)}_t)}{cY^{(1)}_t+\hat g(t,Y^{(1)}_t)}, \quad t\in [0,T).
\end{equation}
It remains to prove uniqueness. For notational convenience, we write $m$ instead of $m_T$.	
Fix $m\in \mathbb R$. By Proposition \ref{uniquexi} and Remark \ref{uniq-appl}, $\xi^m\in L^2_a([c, +\infty))$ given by \eqref{terminal_X} for the fixed $m$, is  the unique minimizer of 
\[
J_m(\xi) := \mathbb E\left[-\xi+\frac{\gamma}{2}(\xi-\theta m)^2\right]. 
\]
As observed in  Remark~\ref{Lasry-appl}, the terminal cost functional is not monotone in the sense of Lasry-Lions. Therefore, we establish uniqueness via the following contraction argument.
Define
\[
\Gamma(m):=\mathbb E[\xi^m].
\]
We show that \(\Gamma\) is a contraction.
Let \(m_1,m_2\in\mathbb R\), and set
$\xi_1:=\xi^{m_1}, \xi_2:=\xi^{m_2}$.
Since \(L^2_a([c, +\infty))\) is convex, for every \(\varepsilon\in(0,1)\),
\[
\xi_\varepsilon:=(1-\varepsilon)\xi_1+\varepsilon\xi_2\in L^2_a([c, +\infty)).
\]
Moreover, since \(\xi_1\) minimizes \(J_{m_1}\) over \(L^2_a([c, +\infty))\),
\[
J_{m_1}(\xi_1)\le J_{m_1}(\xi_\varepsilon).
\]
On the other hand, by the \(\lambda\)-convexity of \(J_{m_1}\),
\[
J_{m_1}(\xi_\varepsilon)\le(1-\varepsilon)J_{m_1}(\xi_1)+\varepsilon J_{m_1}(\xi_2)-\lambda\varepsilon(1-\varepsilon)\mathbb E\left[|\xi_2-\xi_1|^2 \right].
\]
Combining these last two inequalities gives
\[
J_{m_1}(\xi_1)\le(1-\varepsilon)J_{m_1}(\xi_1)+\varepsilon J_{m_1}(\xi_2)-\lambda\varepsilon(1-\varepsilon)\mathbb E\left[|\xi_2-\xi_1|^2 \right].
\]
Rearranging and dividing by \(\varepsilon>0\), we obtain
\[
J_{m_1}(\xi_2)-J_{m_1}(\xi_1)\ge \lambda(1-\varepsilon)\mathbb E\left[|\xi_2-\xi_1|^2 \right].
\]
Repeating the same argument with \(m_2\), minimizer \(\xi_2\), and competitor \(\xi_1\), yields
\[
J_{m_2}(\xi_1)-J_{m_2}(\xi_2)\ge\lambda (1-\varepsilon)\mathbb E\left[|\xi_2-\xi_1|^2 \right].
\]
Adding the last two inequalities gives
\begin{equation}\label{ineq-m1m2}
2\lambda (1-\varepsilon)\mathbb E \left[|\xi_2-\xi_1|^2 \right]\le
J_{m_1}(\xi_2)-J_{m_1}(\xi_1)+J_{m_2}(\xi_1)-J_{m_2}(\xi_2).
\end{equation}
Since
\[
J_m(\xi)=\mathbb E\left[-\xi+\frac{\gamma}{2}(\xi-\theta m)^2\right],
\]
the right-hand side in \eqref{ineq-m1m2} satisfies
\[
\begin{aligned}
&J_{m_1}(\xi_2)-J_{m_1}(\xi_1)+J_{m_2}(\xi_1)-J_{m_2}(\xi_2) \\
&=\frac{\gamma}{2}\mathbb E\Big[(\xi_2-\theta m_1)^2-(\xi_1-\theta m_1)^2+(\xi_1-\theta m_2)^2-(\xi_2-\theta m_2)^2\Big]\\
&=\gamma\theta(m_1-m_2)\mathbb E[\xi_1-\xi_2].
\end{aligned}
\]
Therefore,
\begin{align*}
2\lambda (1-\varepsilon)\mathbb E \left[|\xi_2-\xi_1|^2 \right] & \le \gamma\theta(m_1-m_2)\mathbb E[\xi_1-\xi_2]\\
& \le \gamma\theta |m_1-m_2|\left|\mathbb E[\xi_1-\xi_2]\right|.
\end{align*}
By Jensen's inequality,
$\left|\mathbb E[\xi_1-\xi_2]\right|^2 \le \mathbb E\left[|\xi_1-\xi_2|^2\right]$.
Hence, if \(\mathbb E[\xi_1-\xi_2]\neq0\), we obtain
\[
\left|\mathbb E[\xi_1-\xi_2]\right| \le\frac{\gamma\theta}{2\lambda(1-\varepsilon)}|m_1-m_2|.
\]
If \(\mathbb E[\xi_1-\xi_2]=0\), the same inequality trivially holds too. Thus
\[
|\Gamma(m_1)-\Gamma(m_2)|\le\frac{\gamma\theta}{2\lambda(1-\varepsilon)}|m_1-m_2|.
\]
By Remark \ref{uniq-appl}  for the present quadratic cost we may take \(\lambda=\gamma/2\). Therefore,
\[
|\Gamma(m_1)-\Gamma(m_2)| \le \frac{\theta}{1-\varepsilon} |m_1-m_2|.
\]
Since \(\theta\in[0,1)\), we may choose \(\varepsilon>0\) sufficiently small such that $\frac{\theta}{1-\varepsilon}<1$, yielding that \(\Gamma\) is a contraction on \(\mathbb R\). Therefore the consistency condition
$m=\mathbb E[\xi^m]$
has at most one solution. Finally, by Proposition \ref{equivalence} we have  uniqueness of the MFG equilibrium $(\hat \alpha, \hat m_T)$.
\end{proof}

\begin{rem}
   As a baseline for comparison, we consider the analytical solution to the unconstrained MFG, i.e., without the terminal constraint $X_T \geq c$.  The solution is obtained by applying the stochastic maximum principle and characterizing the optimal control through the adjoint processes $(Y, Z)$. By solving the resulting coupled system of Forward-Backward SDEs (FBSDEs), we obtain the following MFG solution: 
   \[
m_T = \frac{ \gamma a-1\   + \  \exp(\rho^2T)  }{ \gamma \left( (1 - \theta)\exp(\rho^2T) +\theta  \right) },
\]
and 
\[
\alpha_t=-\dfrac{\rho}{\sigma}\dfrac{(a\gamma(1-\theta)-1)\psi(t)}{(a\gamma(1-\theta)-1)\psi(t)+\exp(\rho^2 T)+a\gamma \theta},
\]
where $$\psi(t):=\exp \left(\rho^2(T-t) -\dfrac{1}{2}\rho^2 t - \rho  W_t\right).$$
The corresponding  wealth process ${X}_t$ is
\begin{equation*}
  X_t=\dfrac{(a\gamma(1-\theta)-1)\psi(t)+\exp(\rho^2T)+a\gamma\theta}{\gamma((1-\theta)\exp(\rho^2T )+\theta)}.
\end{equation*}

\end{rem}

\paragraph{Numerics.}
We discretize the time interval $[0,T]$ into $N=1000$ steps, with $T=1$ the terminal time, and simulate the stochastic process $Y_t^{(1)}$ using standard Brownian increments with volatility $\sigma=0.2$ and drift parameter $b=0.1$. For the unconstrained model, the optimal wealth process and risky allocation are computed analytically from their explicit formulas.

In the constrained mean field game setting, the unknown pair $(h_1,m_T)$ is determined numerically by solving the coupled nonlinear system \eqref{eq:h1}. This is done using the nonlinear least-squares routine \texttt{least\_squares} from SciPy, initialized from several starting guesses in order to improve robustness. The expectations in \eqref{eq:h1} are approximated by Monte Carlo simulations with $50\,000$ samples of the terminal random variable $Y_T^{(1)}$. Once $(h_1,m_T)$ has been computed, the constrained wealth paths and optimal allocations are evaluated along the simulated trajectories of $Y_t^{(1)}$, where the $\hat{g}, \partial_y \hat{g} $ are estimated through conditional Monte Carlo expectations.

The simulations consider $\gamma\in\{0.5,2\}$, $\theta\in\{0.25,0.75\}$, $a=5$, and $c=4.5$. Table \ref{tab:convergence_results} reports the resulting numerical values of $h_1$ and $m_T$ for the different parameter configurations. In all tested scenarios, the numerical procedure exhibits high stability, with residual errors of order $10^{-14}$ or smaller.

\begin{table}[H]
\caption{Numerical solution of the coupled system \eqref{eq:h1} using a nonlinear least-squares method, together with the achieved residual errors}
\label{tab:convergence_results}
\centering
\begin{tabular}{ccccc}
\toprule
$\gamma$ & $\theta$ & $h_1$ & $m_T$ & Residual \\
\midrule
0.5 & 0.25 & $-0.65$ & $4.75$ & $1.49 \times 10^{-16}$ \\
0.5 & 0.75 & $0.37$  & $5.16$ & $1.02 \times 10^{-15}$ \\
2.0 & 0.25 & $-4.23$ & $4.71$ & $2.29 \times 10^{-16}$ \\
2.0 & 0.75 & $-1.37$ & $4.81$ & $1.24 \times 10^{-14}$ \\
\bottomrule
\end{tabular}
\end{table}

A first clear effect of the terminal constraint (see Figure \ref{fig:alphapaths}) appears in the behavior of the optimal investment strategy near maturity. In the unconstrained case, the strategy $\alpha_t$ remains relatively stationary, as the agent maintains a consistent risk-reward profile throughout the entire horizon. By contrast, in the constrained case ($\hat{\alpha}_t$), the risky allocation decreases more visibly as $t$ approaches $T$, reflecting the need to control downside risk in order to satisfy the terminal condition $X_T \geq c$. This behavior shows that the terminal constraint becomes increasingly relevant near maturity and induces a more cautious investment policy. 
\begin{figure}[H]
    \centering
\includegraphics[width=0.8\textwidth]{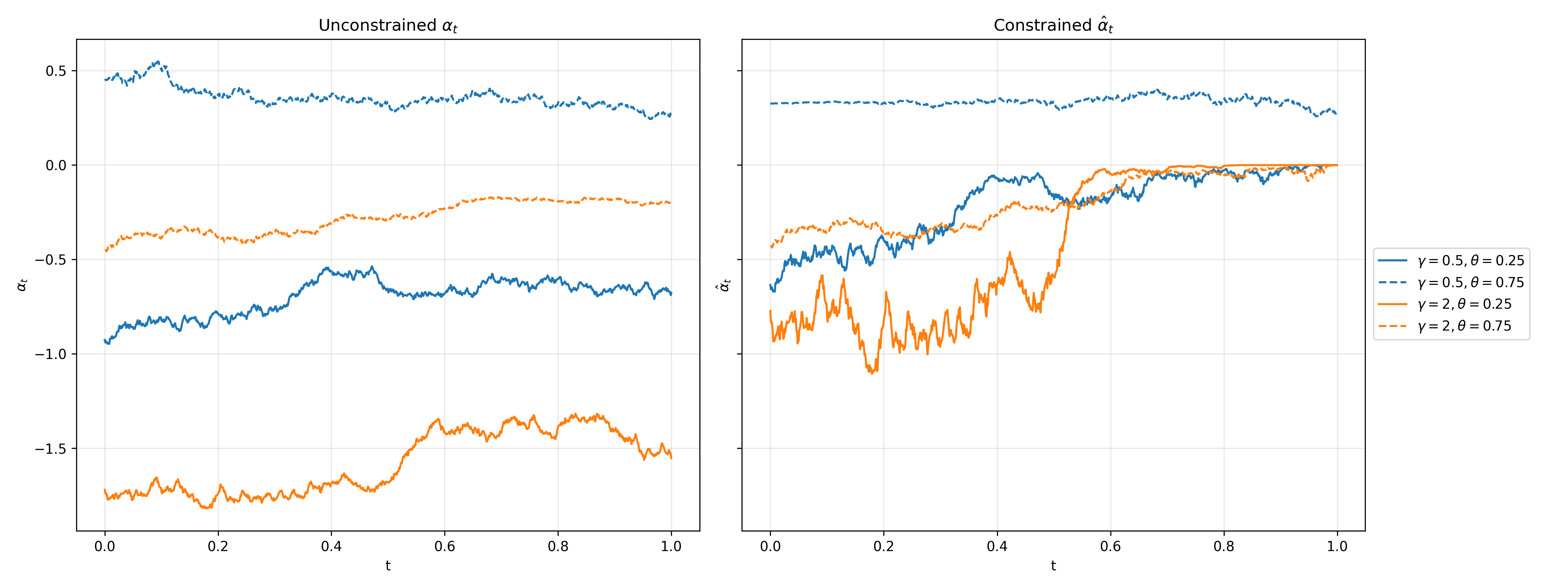}
    \caption{Comparison of optimal investment strategies in the constrained ($\hat{\alpha}_t$) and unconstrained ($\alpha_t$) MFG scenarios. Parameters: $b = 0.1$, $\sigma = 0.2$, $T = 1$, $a = 5$, $c = 4.5$, $\theta \in \{0.25, 0.75\}$, and $\gamma \in \{0.5, 2\}$. }
    \label{fig:alphapaths}
\end{figure}
 
The comparison between the unconstrained and constrained wealth processes (see Figure \ref{fig:Xpaths})  illustrates clearly the impact of the lower bound $c=4.5$. In the unconstrained model, the trajectories evolve freely,  frequently crossing the $4.5$ threshold and diverging downward. In the constrained case, the paths remain systematically above $c$.
This behavior can be explained as follows: since the terminal claim $X_T \geq c$ is non-negative relative to the floor, its value process $\hat{X}_t - c$ --- interpreted as a conditional price --- must remain non-negative for the entire horizon. 

\begin{figure}[H]
    \centering
\includegraphics[width=0.8\textwidth]{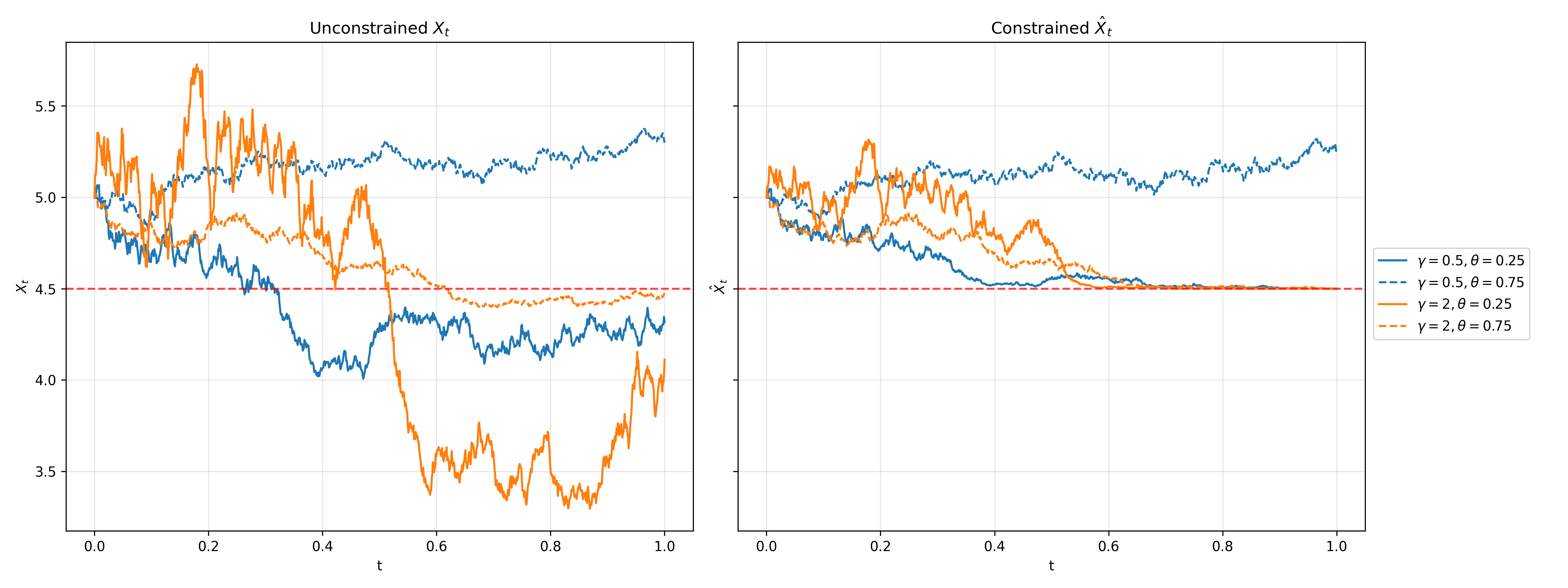}
    \caption{Comparison of optimal wealth paths in the constrained ($\hat{X}_t$) and unconstrained ($X_t$) MFG scenarios. The parameters are $b = 0.1$, $\sigma = 0.2$, $T = 1$, $a = 5$, $c = 4.5$, with variations in $\theta \in \{0.25, 0.75\}$ and $\gamma \in \{0.5, 2\}$.}
    \label{fig:Xpaths}
\end{figure}

For a fixed level of risk aversion $\gamma$, an increase in the relative performance parameter $\theta$ leads to a higher optimal investment $\hat{\alpha}_t$ (see Figure \ref{fig:theta_constrained_results}). This reflects the fact that stronger concerns about relative performance induce agents to adopt more aggressive investment strategies in order to outperform the average terminal wealth. 

For a fixed $\theta$, a decrease in $\gamma$ results in an increase in $\hat{\alpha}_t$ (see Figure \ref{fig:gamma_constrained_results}). Since $\gamma$ determines the weight assigned to the quadratic penalty on deviations from the benchmark, a smaller $\gamma$ corresponds to lower effective risk aversion. Consequently, agents are more willing to tolerate variability in terminal wealth and therefore allocate a larger fraction of their wealth to the risky asset.

Overall, these results indicate that agents who are either more competitive or less risk-averse optimally choose a higher exposure to the risky asset. 
\begin{figure}[H]
    \centering
    \begin{minipage}[t]{0.48\textwidth}
        \centering
        \includegraphics[width=\textwidth]{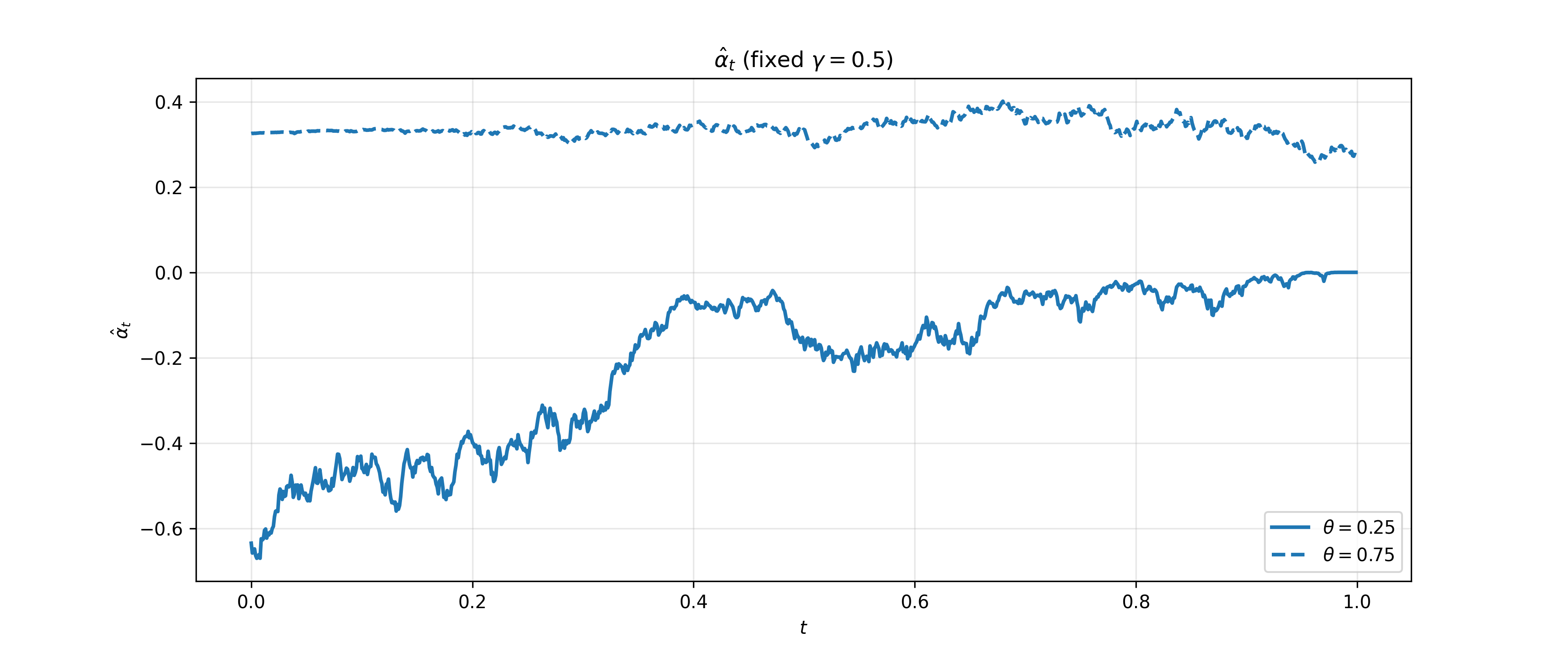}
    \end{minipage}
    \hfill
    \begin{minipage}[t]{0.48\textwidth}
        \centering
        \includegraphics[width=\textwidth]{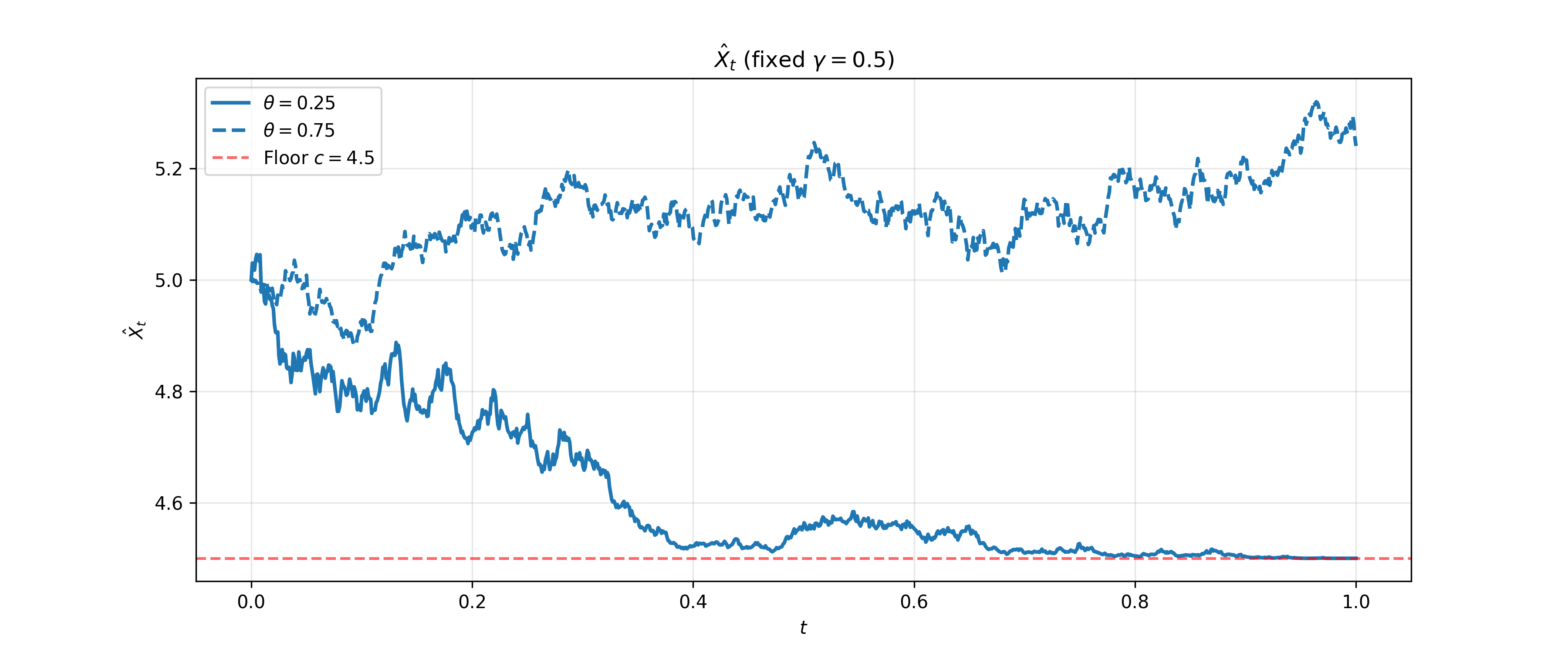}
    \end{minipage}

    \caption{Constrained optimal investment $\hat{\alpha}_t$ and corresponding wealth $\hat{X}_t$
    for parameters $b = 0.1$, $\sigma = 0.2$, $T = 1$, $a = 5$, $c = 4.5$, $\gamma = 0.5$
    and $\theta \in \{0.25,  0.75\}$.}
    \label{fig:theta_constrained_results}
\end{figure}

\begin{figure}[H]
    \centering
    \begin{minipage}[t]{0.48\textwidth}
        \centering
        \includegraphics[width=\textwidth]{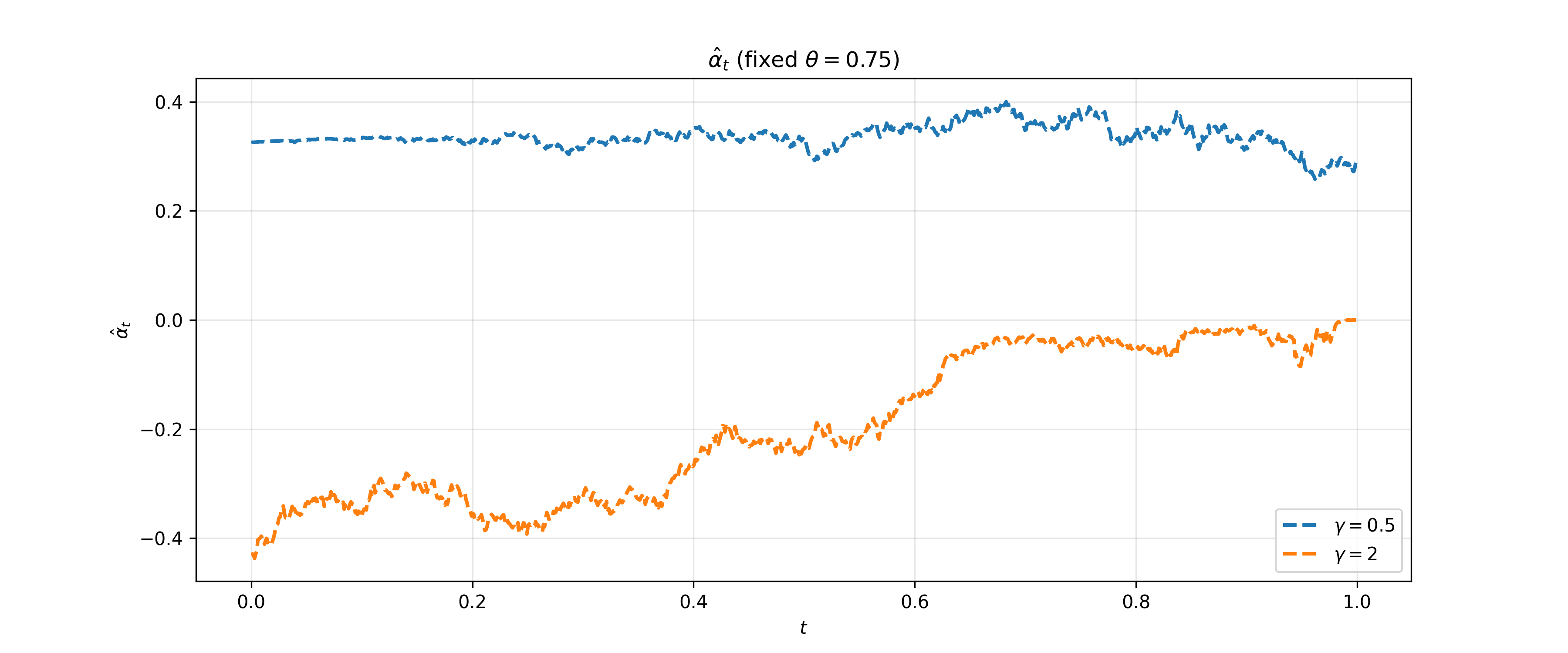}
    \end{minipage}
    \hfill
    \begin{minipage}[t]{0.48\textwidth}
        \centering
        \includegraphics[width=\textwidth]{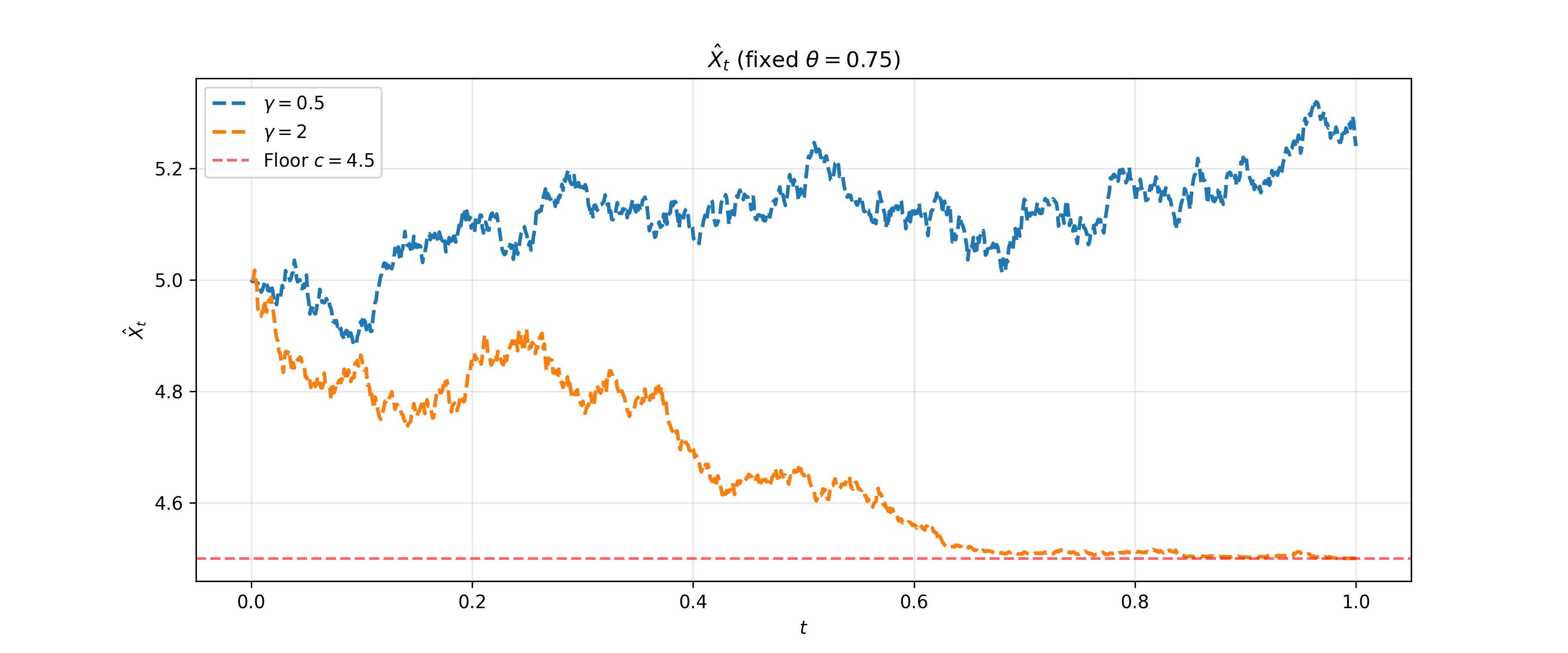}
    \end{minipage}

    \caption{Constrained optimal investment $\hat{\alpha}_t$ and corresponding wealth $\hat{X}_t$
    for parameters $b = 0.1$, $\sigma = 0.2$, $T = 1$, $a = 5$, $c = 4.5$,
    $\theta = 0.75$, and $\gamma \in \{0.5,  2\}$.}
    \label{fig:gamma_constrained_results}
\end{figure}

\begin{figure}[H]
    \centering
    
\includegraphics[width=0.5\textwidth]{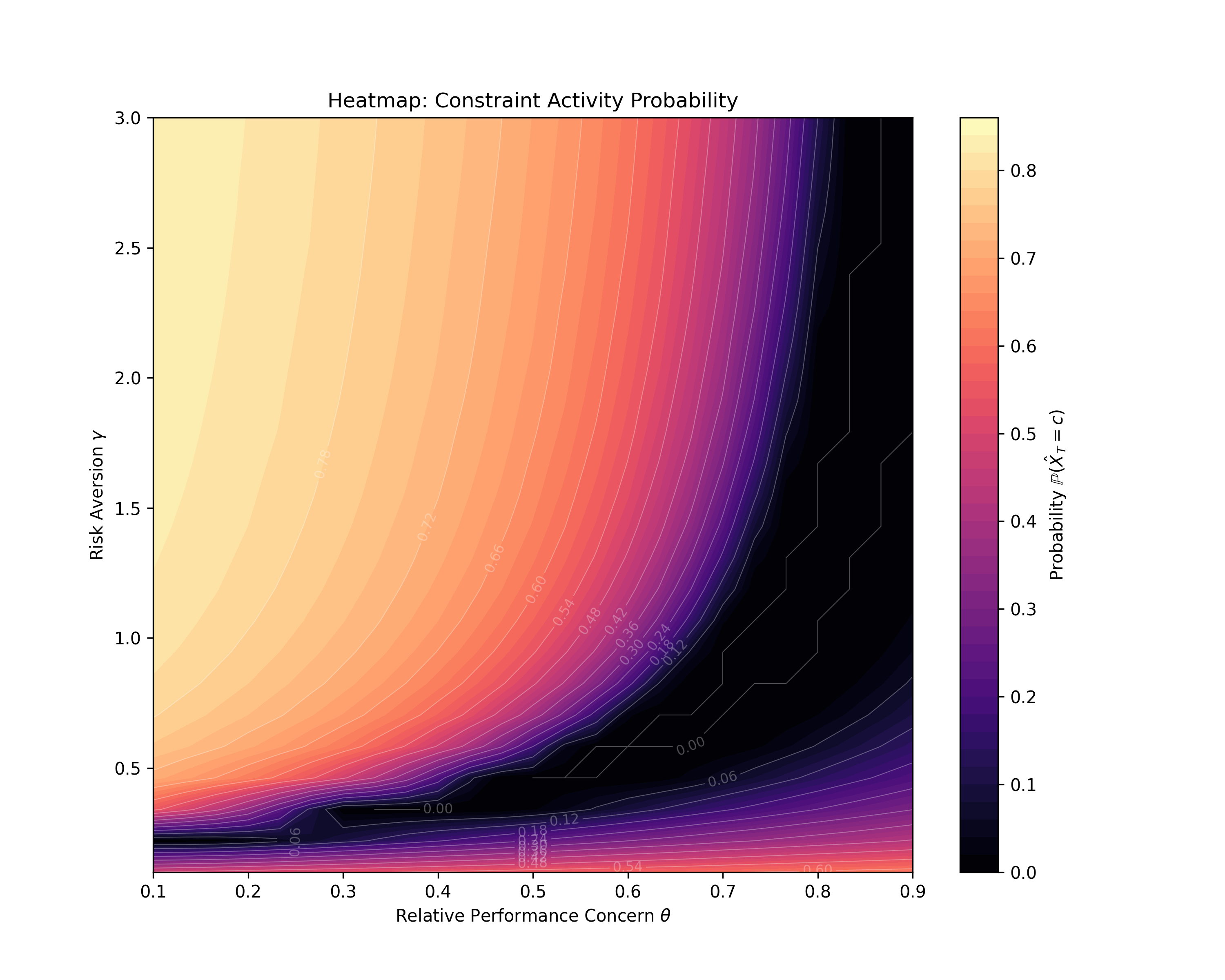}

    \caption{ Heatmap of constraint activity probability $\mathbb{P}(\hat{X}_T = c)$. This figure maps the probability $\mathbb{P}(\hat{X}_T = c)$ as a function of the risk aversion $\gamma$ and the relative performance concern $\theta$.}
    \label{fig:Heatmap}
\end{figure}

The heatmap in Figure~\ref{fig:Heatmap} provides a global sensitivity analysis of the terminal constraint activity probability $\mathbb{P}(\hat{X}_T = c)$ within the parameter space $(\theta, \gamma)$. A first feature of this plot is the diagonal corridor of low probability (near 0), identifying the stable regime where the optimal terminal state stays strictly above the threshold $c$. 
The figure identifies two regions where the constraint is more likely to bind: one corresponding to high risk aversion and low relative performance concern, and another corresponding to sufficiently strong relative performance concern. This confirms that the relevance of the terminal constraint depends on the balance between the parameters $\theta$ and $\gamma$.
\begin{figure}[H]
    \centering
\includegraphics[width=\linewidth,height=0.28\textheight,keepaspectratio]{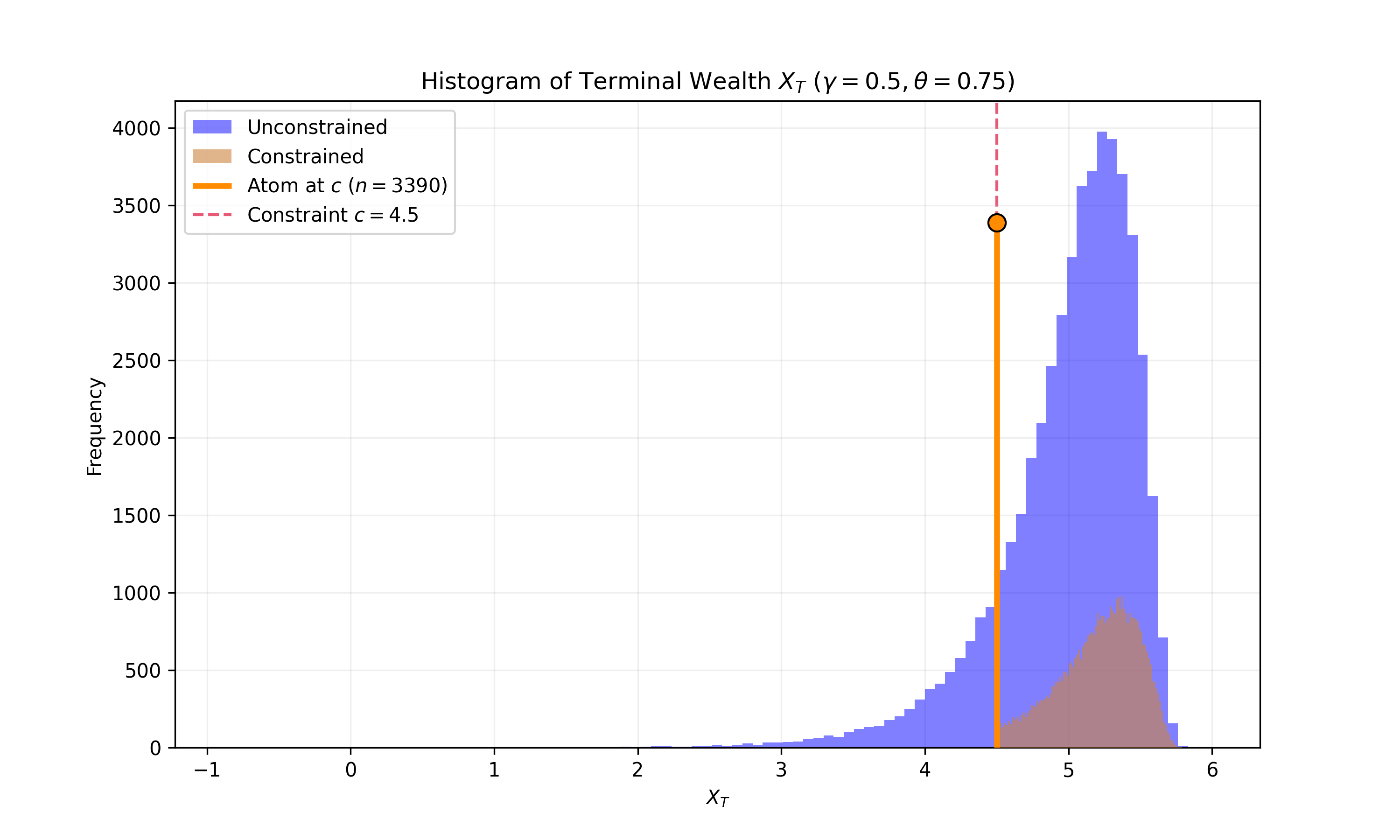}    

    \caption{ Distribution of terminal wealth $X_T$ under the constraint $c=4.5$. The histogram compares the unconstrained (blue) and constrained (orange) wealth distributions for the parameter set $\gamma=0.5, \theta=0.75$.}
    \label{fig:Histogram}
\end{figure}

Figure \ref{fig:Histogram} provides a useful comparison between the terminal wealth distributions under unconstrained and constrained regimes for $\gamma=0.5$ and $\theta=0.75$. The red dashed line denotes the terminal constraint $c=4.5$. The blue density represents the distribution of $X_T$ in the unconstrained case, where a significant portion of the lower tail violates the threshold. In contrast, the orange density represents the optimal terminal state $\hat{X}_T$ under the constraint. The figure clearly shows that the constraint truncates the lower tail of the distribution and creates a positive mass at the boundary $x=c$ (indicated by the vertical line and bullet). This  ensures that all realized trajectories terminate above $c$. The concentration of mass near the boundary 
$c$ reflects the fact that the constrained terminal state stays as close as possible to the unconstrained one while remaining above the threshold.

\begin{figure}[H]
    \centering
    \begin{minipage}[t]{0.48\textwidth}
        \centering
        \includegraphics[width=\textwidth]{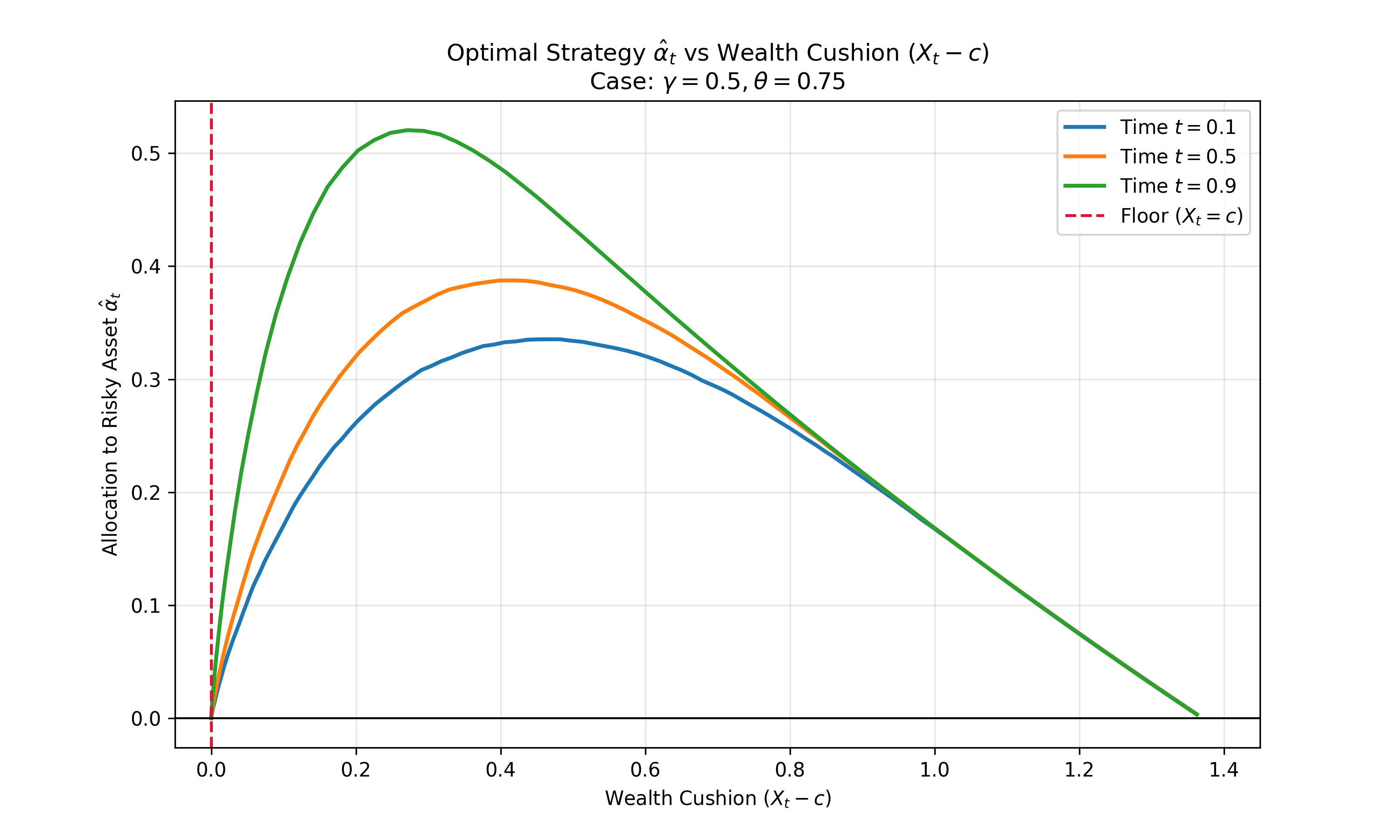}
    \end{minipage}
    \hfill
    \begin{minipage}[t]{0.48\textwidth}
        \centering
        \includegraphics[width=\textwidth]{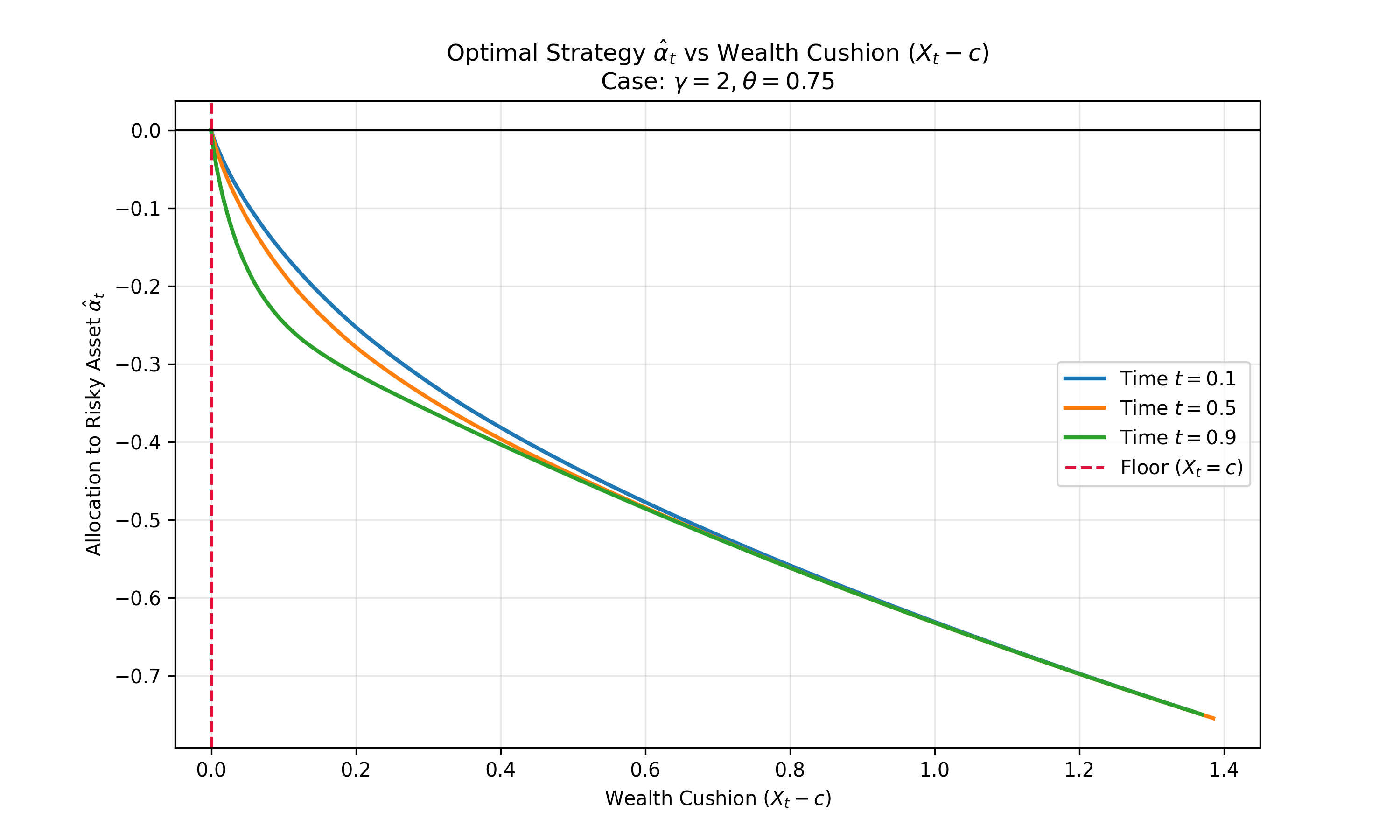}
    \end{minipage}

    \caption{Sensitivity of optimal strategy $\hat \alpha_t$ to Risk Aversion $\gamma$.
These panels illustrate the optimal allocation as a function of the wealth cushion $X_t-c$ for low risk aversion ($\gamma=0.5$, left) and high risk aversion ($\gamma=2$, right), with relative performance concern fixed at $\theta=0.75$. Each curve represents a fixed time slice $t \in \{0.1, 0.5, 0.9\}$.} 
    \label{fig:cushion}
\end{figure}

Figure 7 shows that risk aversion has a strong effect on the optimal investment strategy. When risk aversion is low $(\gamma = 0.5)$, the investor takes a positive position in the risky asset and is more willing to follow the group's performance. The position is largest when the wealth buffer $X_t-c$ is at a middle level, where the investor still has room to take risk but must also avoid falling below the minimum wealth level.
When risk aversion is high $(\gamma = 2)$, the investor focuses more on protecting wealth, so the risky position becomes negative across all time horizons. As the buffer gets smaller, the investor reduces risk, especially close to maturity, when there is less time to recover from losses. Overall, $\theta$ affects the desire to move away from the minimum level $c$, while $\gamma$ determines how cautious or aggressive the strategy is.

\paragraph{The case with pure common noise.}
We conclude by considering the case with common noise only, i.e., all agents trade in the same risky asset. In this case $\sigma=0$ and solving the corresponding MFG consists in the following procedure: 
\begin{enumerate}
\item For each fixed $\mathcal F_T ^0$-measurable random variable $m_T$, we solve the following maximization problem
\[ \sup_{\alpha \in \mathcal A}\mathbb E\left[ X_T -\frac{\gamma}{2}(X_T - \theta m_T)^2  \right], \]
subject to
\begin{equation}
    \label{sde oic}
    dX_t = \alpha_t X_t (b dt + \sigma^0 dW_t^0 ), \quad X_0=a >0,  
\end{equation}
under the constraint $X_T \geq c$, for a given threshold $c >0$.

\item Find a fixed point $\hat m_T = \mathbb E[\hat X _T| \mathcal{F}_T^0]$, where $\hat X$ is the state associated to the optimal control $\hat\alpha$ obtained in 1.
\end{enumerate}

In the following proposition we state a characterization of the unique MFG solution, which is the analogue of Proposition \ref{prop:application1} in the setting where all agents trade in the same asset.

Let us introduce some preliminary notation: let $\rho^0 := b/\sigma^0$ be the market price of risk, $Y^{(1)} := \mathcal E (0, -\rho^0)$.
\begin{prop}
\label{prop:pure_common_noise}
Let $\hat h_1 \in \mathbb R$ be the unique solution of the equation 
\[
G(h_1):= c -a+ \mathbb E\left[Y_T^{(1)}\left(\frac{1}{\gamma(1-\theta)}-c-\frac{h_1Y_T^{(1)}}{{\gamma(1-\theta)}}\right)^+\right]=0,\]
and let
\[\hat g^0(t, y):=\mathbb{E}\!\bigg[yY^{(1)}_{t,T} \bigg(\frac{1}{\gamma(1-\theta)}-c - \dfrac{\hat h_1}{\gamma(1-\theta)}yY^{(1)}_{t,T}\bigg)^+\bigg], \quad (t,y)\in [0,T] \times \mathbb R^+.\]
The unique MFG solution is given by the couple $(\hat \alpha, \hat m_T)$ with
\begin{equation}
    \label{opt-com-alpha}
    \hat\alpha_t=-\frac{\rho^0}{\sigma^0}
\dfrac{Y^{(1)}_t\partial_{y}\hat g^0(t,Y^{(1)}_t) - \hat g^0(t,Y^{(1)}_t)}{cY^{(1)}_t+\hat g^0(t,Y^{(1)}_t)} , \quad t\in [0,T),
\end{equation}
\begin{equation}
    \label{opt-com-m}
\hat m_T= c + \left(\frac{1}{\gamma(1-\theta)}-c - \frac{\hat h_1}{\gamma(1-\theta)} Y^{(1)}_T \right)^+ .
\end{equation}
\end{prop}
\begin{proof}
Arguing as at the beginning of the proof of Theorem \ref{EXUN}, we obtain formula \eqref{xirep} with parameters $k_0 = -1$, $k_1 = \gamma$, $k_2 =-\gamma\theta$. This leads to the following expression for the terminal condition:
\begin{equation}
\label{xicommon}
    \xi=c+\left(\frac{-Y_T^{(h_1)}+1-\gamma c+\gamma \theta m_T}{\gamma}\right)^+.
\end{equation}
Since \(\sigma=0\), the state process generated by an admissible feedback depending on the common Brownian motion only is \(\mathcal F_t^0\)-adapted. Hence, at the fixed point, \(m_T=\mathbb E[X_T\mid\mathcal F_T^0]=X_T=\xi\). Therefore the conditional fixed-point equation reduces pathwise to
\[
\xi=c+\left(\frac{-Y_T^{(h_1)}+1-\gamma c+\gamma\theta\xi}{\gamma}\right)^+.
\]
Solving this scalar projection equation yields
\[
\xi^{(h_1)}=c+\left(\frac{1-Y_T^{(h_1)}}{\gamma(1-\theta)}-c\right)^+.
\]
By the same arguments as in the proof of Theorem \ref{EXUN}, we deduce that there exists a unique   $\hat h_1$  which solves the following  equation
\[G(h_1):=\mathbb E\left[\xi^{(h_1)}Y_T^{(1)} \right]-a=0.
\]  
The rest of the proof is analogous to that of Proposition \ref{prop:application1} and is omitted.
\end{proof}
\begin{rem}
   The solution of MFG with  common noise only and without the constraint $X_T \geq c$, is given by 
   \begin{equation*}
 m_T=\bigg(a-\frac{1}{\gamma(1-\theta)}\bigg)\psi^0(T)+\frac{1}{\gamma(1-\theta)},
\end{equation*}
and 
\begin{equation*}
    \begin{aligned}
       \alpha_t=-\dfrac{\rho^0 (a \gamma(1-\theta)-1)\psi^0(t)}{\sigma^0 ((a\gamma(1-\theta)-1)\psi^0(t)+1)},
    \end{aligned}
\end{equation*}
where
$$\psi^0(t):=\exp \left(-\frac{3}{2}(\rho^0)^2t - \rho^0  W_t^0\right).$$ The corresponding  wealth process ${X}_t$ is
 $$X_t=\bigg(a-\frac{1}{\gamma(1-\theta)}\bigg) \psi^0(t) +\frac{1}{\gamma(1-\theta)}.$$
\end{rem}

\paragraph{Numerics.}

For the pure common noise case, we use the same discretization and parameter values as in the previous subsection, replacing $\sigma$ by $\sigma^0=0.2$. Since all trajectories are driven by the same Brownian motion $W^0$, the simulated common-noise path is reused across all parameter configurations $(\gamma,\theta)$.

In the constrained MFG, the unknown parameter $\hat h_1$ is obtained by solving the scalar equation $G(h_1)=0$ with a bisection method. The expectations are approximated by Monte Carlo simulation with $50\,000$ samples. Once $\hat h_1$ is computed, the constrained wealth paths and feedback controls are evaluated along the simulated paths of $W^0$, with $\hat g^0$ and $\partial_y \hat g^0$ estimated by conditional Monte Carlo.

The simulations use $\gamma \in \{0.5,2\}$, $\theta \in \{0.25,0.75\}$, $a=5$, and $c=4.5$. Table \ref{tab:h1_residuals_common} reports the resulting values of $\hat h_1$; the residuals are of order $10^{-12}$ or smaller.

\begin{table}[h!]
\centering
\caption{Solution of $\hat h_1$ using a monotone bracketed method (bisection) with achieved residuals for the common noise case}
\begin{tabular}{ccccc}  
\hline
$\gamma$ & $\theta$ & $\hat h_1$ & Residual  \\
\hline
0.5 & 0.25 & -0.60 & $1.13\times 10^{-12}$ \\
0.5 & 0.75 &  0.39 & $ 1.86\times 10^{-12}$  \\
2.0 & 0.25 & -3.94 & $1.67\times 10^{-13}$  \\
2.0 & 0.75 & -0.98 & $ 2.19\times 10^{-13}$  \\
\hline
\end{tabular}
\label{tab:h1_residuals_common}
\end{table}

The comparison between the unconstrained control $\alpha_t$ and the constrained control $\hat{\alpha}_t$ under pure common noise (see Figure \ref{fig:commonalphapaths}) shows that the terminal constraint reduces the dispersion of the optimal policies near maturity. In the unconstrained case, the optimal strategies for various $(\gamma, \theta)$ scenarios remain approximately parallel over time, where the relative distances between policy responses remain constant over time due to the absence of idiosyncratic volatility. In contrast, the constraint $X_T \geq c$ disrupts this synchronization. As systemic shocks drive wealth toward the floor $c$, the optimal controls for all parameterizations are forced to converge toward a common limit to ensure solvency at maturity. This effect illustrates that the terminal constraint reduces these differences as the wealth approaches the floor $c$. 
\begin{figure}[H]
    \centering
\includegraphics[width=0.8\textwidth]{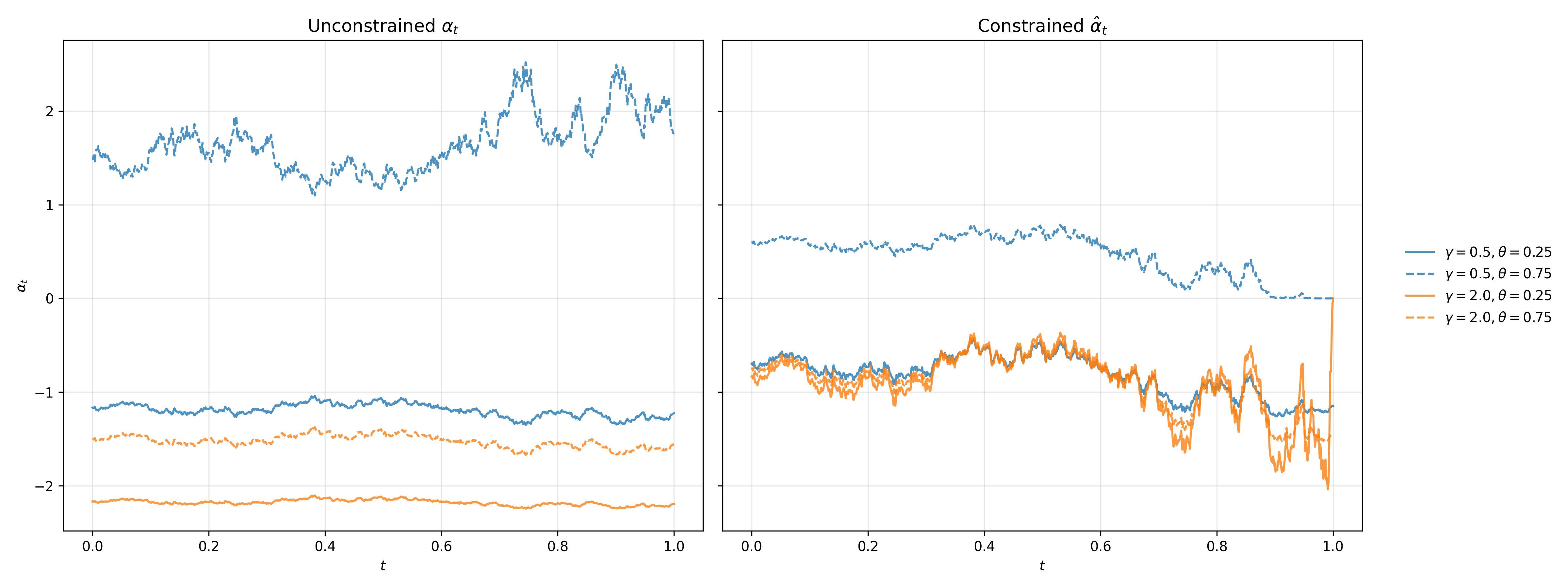}
   \caption{Comparison of optimal investment strategies in the constrained ($\hat{\alpha}_t$) and unconstrained ($\alpha_t$) MFG scenarios with common noise. Parameters: $b = 0.1$, $\sigma^0 = 0.2$, $T = 1$, $a = 5$, $c = 4.5$, $\theta \in \{0.25, 0.75\}$, and $\gamma \in \{0.5, 2\}$. }
    \label{fig:commonalphapaths}
\end{figure}

Figure \ref{fig:commonXpaths} highlights the stronger co-movement of wealth paths in the pure common noise case.
Since the same shock affects both individual wealth and the benchmark $m_T$, deviations from
the benchmark are smaller than under idiosyncratic noise. As a result, increasing risky exposure
raises terminal wealth while generating a smaller increase in the quadratic relative-performance
penalty, which helps explain why the optimal risky allocation can be higher in the common-noise
setting.

\begin{figure}[H]
    \centering
\includegraphics[width=0.8\textwidth]{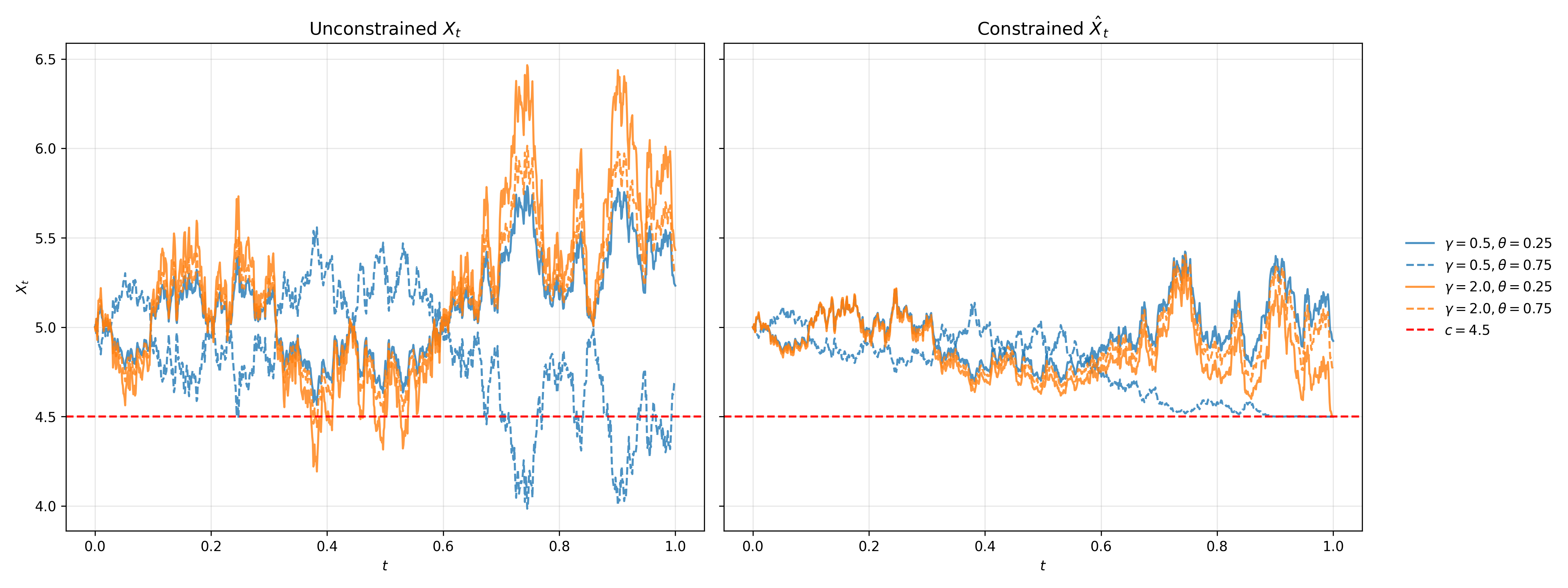}

    \caption{Comparison of optimal wealth paths in the constrained ($\hat{X}_t$) and unconstrained ($X_t$) MFG scenarios with pure common noise. The parameters are $b = 0.1$, $\sigma^0 = 0.2$, $T = 1$, $a = 5$, $c = 4.5$, with variations in $\theta \in \{0.25, 0.75\}$ and $\gamma \in \{0.5, 2\}$.}
    \label{fig:commonXpaths}
\end{figure}
In the pure common noise case, the comparative statics with respect to $\theta$ and $\gamma$ are the same as in the no common noise case: higher relative performance concern and lower risk aversion lead to larger risky allocations. Likewise, the constraint activity probability exhibits the same pattern  as in the no common noise case: there is a diagonal region where the constraint is mostly inactive, while outside this region the floor binds with positive probability. This again reflects the balance between relative performance incentives and the terminal wealth constraint. The same truncation effect is observed for the terminal wealth distribution. The constrained distribution has no mass below $c=4.5$ and displays an atom at the boundary, consistent with the constraint activity probability described above.

\begin{figure}[H]
    \centering
    \begin{minipage}[t]{0.48\textwidth}
        \centering
        \includegraphics[width=\linewidth,height=0.28\textheight,keepaspectratio]{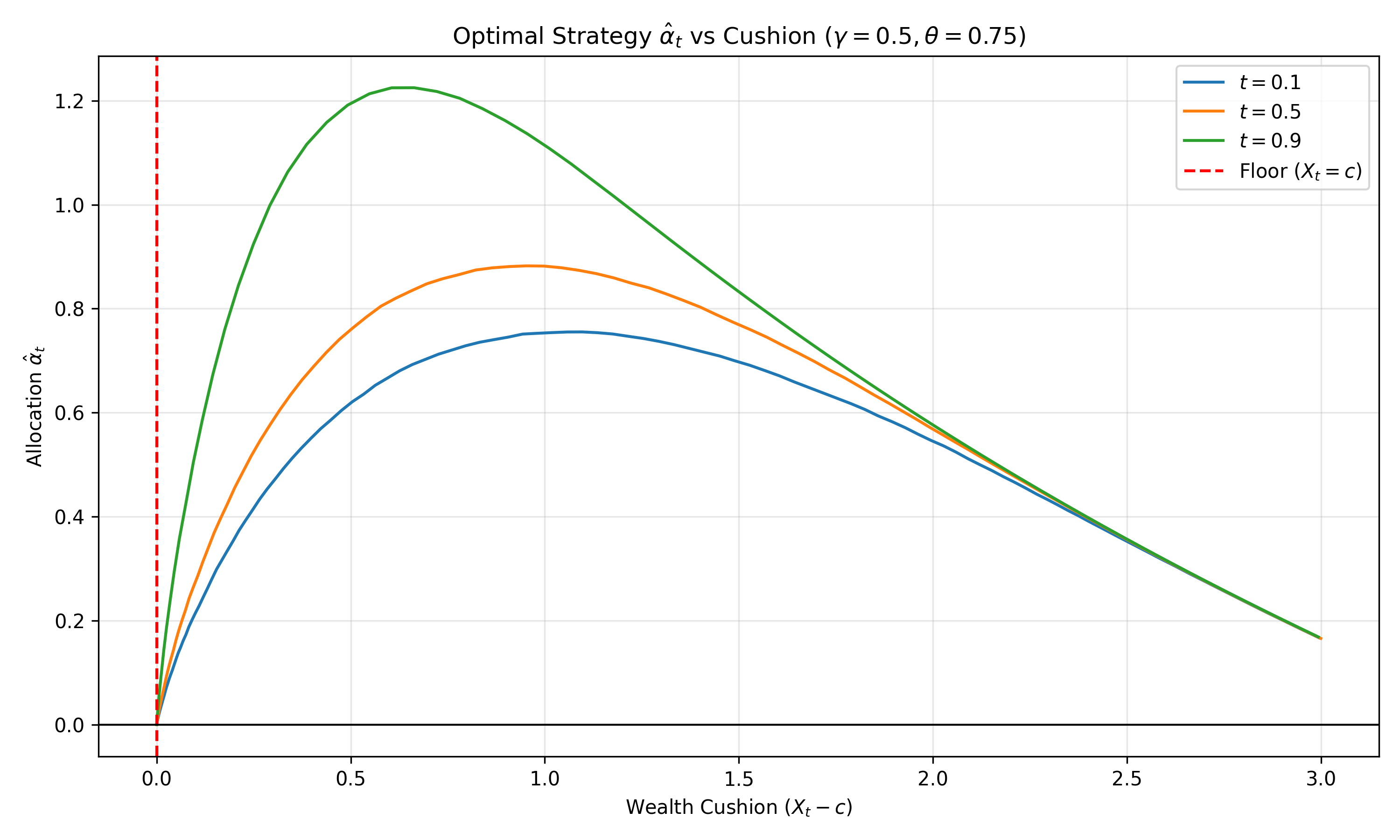}
    \end{minipage}
    \hfill
    \begin{minipage}[t]{0.48\textwidth}
        \centering
        \includegraphics[width=\textwidth]{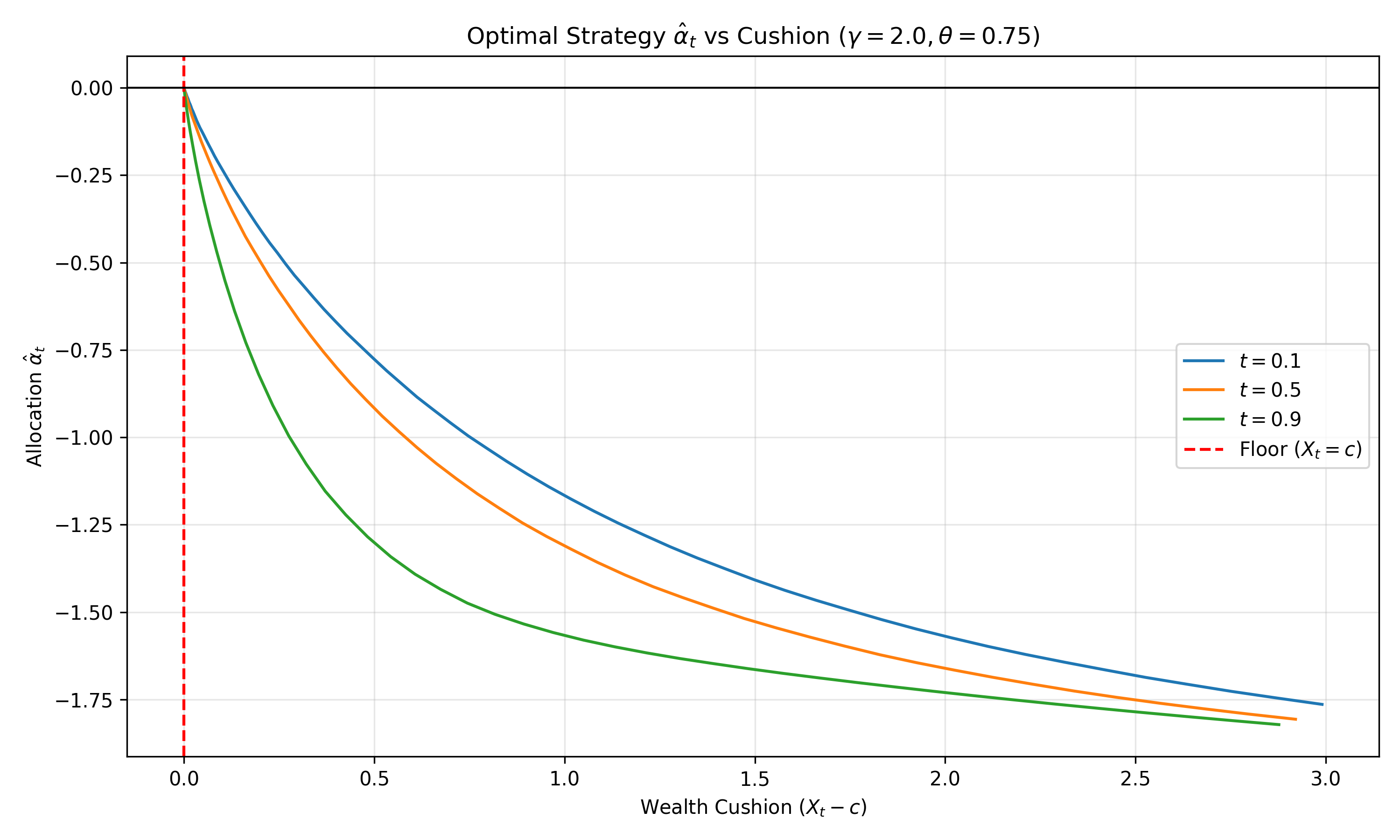}
    \end{minipage}

    \caption{Sensitivity of optimal strategy $\hat \alpha_t$ to Risk Aversion $\gamma$.
These panels illustrate the optimal allocation as a function of the wealth buffer $X_t-c$ for low risk aversion ($\gamma=0.5$, left) and high risk aversion ($\gamma=2$, right), with relative performance concern fixed at $\theta=0.75$, in the common noise case. Each curve represents a fixed time slice $t \in \{0.1, 0.5, 0.9\}$.}
    \label{fig:commoncushion}
\end{figure}
Finally, a comparison of Figures \ref{fig:cushion} and \ref{fig:commoncushion} shows that, for the same values of $\gamma$ and $\theta$,
risky allocations are higher in the pure common-noise case than in the idiosyncratic case.
This is consistent with the stronger co-movement between individual wealth and the benchmark
in the common-noise setting.

\appendix 
\renewcommand{\thesection}{Appendix A} 
\section{} 
\label{appendix}
\renewcommand{\thesection}{A} 
In this appendix, we provide the detailed proofs of the necessary and sufficient conditions in Theorem \ref{Necessary Sufficient Condition CN} together with the auxiliary results used in that proof, and the proof of Claim \ref{claim:h1}.\medskip

\noindent We start by observing that $(L^2 (\mathcal D), d(\cdot,\cdot))$ is a complete metric space, where
\[
d(\xi^1, \xi^2)=\left(\mathbb E|\xi^1-\xi^2|^2\right)^{\frac{1}{2}}, \quad \xi^1,\xi^2 \in L^2 (\mathcal D).
\]
Now, let $\hat \xi$ be $a$-optimal and let the corresponding state processes as in \eqref{BSDE CN} be denoted by $(\hat  X, \hat  q, \hat  z)$. For an arbitrary $\xi \in L^2 (\mathcal D)$ and for each $0 \leq \rho \leq 1$, $\hat  \xi+\rho(\xi-\hat  \xi) \in L^2 (\mathcal D)$. The state processes in \eqref{BSDE CN} associated to $\hat  \xi+\rho(\xi-\hat  \xi)$ are given by $(X^\rho, q^\rho, z^\rho)$.
The variational equation is defined as 
\begin{equation}
\label{variational equation CN}
    \begin{cases}
        -d(\Delta{X}_t)=\left(\hat f_x(t)\Delta{X}_t+\hat f_q(t)\Delta{q}_t+\hat f_z(t)\Delta{z}_t \right)dt-\Delta{q}_tdW_t-\Delta{z}_tdW_t^0\\
        \Delta{X}_T=\xi-\hat \xi
        \end{cases}
\end{equation}
where $\hat f_x(t)=f_x(t, \hat X_t, \hat q_t, \hat z_t, \mu_t)$, \  similarly for $\hat f_q(t)$ and $\hat f_z(t)$.

Under Assumptions (A1)-(A5) mentioned above, the BSDE \eqref{variational equation CN} has a unique solution   $(\Delta X, \Delta q,$ $ \Delta z) \in \mathcal S^2 _n \times \mathcal H^2 _{n\times d} \times \mathcal H^2 _{n\times d}$  \cite[Theorem 2.1, p.18]{Karoui Peng Quenez}, \cite[Theorem 3.1, p.58]{Pardoux}.
The associated variational cost functional is defined as
\begin{equation}
\label{variational cost CN}
\Delta{J}(\hat \xi;\mu)=\mathbb E\left[\int_0^T \left (\hat \ell_x(t)\Delta{X}_t+\hat \ell_q(t)\Delta{q}_t+\hat  \ell_z(t)\Delta{z}_t\right)dt+\phi_x(\hat \xi,\mu_T^X)\Delta{X}_T\right],
\end{equation}
where $\hat \ell_p (t)=\ell_p (t, \hat X_t, \hat q_t, \hat z_t, \mu_t)$ for $p\in \{x,q,z\}$. Moreover we set

$$\tilde{X}^\rho_t=\rho^{-1}({X}^\rho_t- \hat {X}_t)-\Delta{X}_t,$$
$$\tilde{q}^\rho_t=\rho^{-1}({q}^\rho_t-\hat {q}_t )-\Delta{q}_t,$$
$$\tilde{z}^\rho_t=\rho^{-1}(z^\rho_t-\hat z_t)-\Delta{z}_t .$$

\begin{lemma}
\label{convergence lemma}
Under Assumptions \emph{(A1)-(A5)}, we have the following convergence result:  
\end{lemma}
\begin{equation}
\label{convergence results CN}
    \begin{aligned}
\lim_{\rho \to 0} \left( \sup_{0 \leq t \leq T} \mathbb E\left[(\tilde{X}^\rho_t)^2 \right] + \mathbb E\left[\int_0^T|\tilde{q}^\rho_t|^2dt \right] + \mathbb E\left[\int_0^T|\tilde{z}^\rho_t|^2dt \right] \right) =0 .
    \end{aligned}
\end{equation}
\begin{proof} We consider the BSDE 
\begin{equation}
\label{tilde}
\begin{cases}
\begin{aligned}
    -d\tilde{X}_t^\rho &= \rho^{-1}(f(t, X_t^\rho,  q_t^\rho, z_t^\rho, \mu_t)-f(t, \hat X_t, \hat q_t, \hat z_t, \mu_t)\\
    &-\rho \hat f_x(t)\Delta{X}_t-\rho \hat f_q(t)\Delta{q}_t-\rho \hat f_z(t)\Delta{z}_t )dt-\tilde{q}_t^\rho dW_t-\tilde{z}_t^\rho dW_t^0\\
\end{aligned}\\
    \tilde{X}_T^ \rho=0
\end{cases}
\end{equation} 
which has a unique solution under the properties (B1)-(B2) and (B4) as shown in \cite{Karoui Peng Quenez}, \cite{Pardoux}.  
Setting
$$\Ab_t^\rho=\int_0^1 f_x(t, \hat X_t+\lambda\rho(\Delta{X}_t+\tilde{X}_t^\rho), \hat q_t+\lambda\rho(\Delta{q}_t+\tilde{q}_t^\rho), \hat z_t+\lambda\rho(\Delta{z}_t+\tilde{z}_t^\rho), \mu_t)d \lambda,$$
$$\Bb_t^\rho=\int_0^1 f_q(t, \hat X_t+\lambda\rho(\Delta{X}_t+\tilde{X}_t^\rho),  \hat q_t+\lambda\rho(\Delta{q}_t+\tilde{q}_t^\rho), \hat z_t+\lambda\rho(\Delta{z}_t+\tilde{z}_t^\rho), \mu_t )d \lambda,$$
$$\Cb_t^\rho=\int_0^1 f_z(t, \hat X_t+\lambda\rho(\Delta{X}_t+\tilde{X}_t^\rho), \hat q_t+\lambda\rho(\Delta{q}_t+\tilde{q}_t^\rho), \hat z_t+\lambda\rho(\Delta{z}_t+\tilde{z}_t^\rho), \mu_t)d \lambda,$$
\begin{equation}
\label{D}
    \Db_t^\rho=(\Ab_t^\rho- \hat f_x(t))\Delta {X}_t+(\Bb_t^\rho- \hat f_q(t))\Delta{q}_t+(\Cb_t^\rho- \hat f_z(t))\Delta{z}_t,
\end{equation}
the BSDE \eqref{tilde} can be written 
\begin{equation*}
\begin{cases}
 -d\tilde{X}_t^\rho =(\Ab_t^\rho \tilde{X}_t^\rho+\Bb_t^\rho \tilde{q}_t^\rho+\Cb_t^\rho \tilde{z}_t^\rho+ \Db_t^\rho )dt-\tilde{q}_t^\rho dW_t-\tilde{z}_t^\rho dW_t^0\\
\tilde{X}_T^ \rho=0.
\end{cases}
\end{equation*}  
By Ito's formula applied to $|\tilde{X}_s^\rho|^2$ over the interval $[t,T]$ and by Young's inequality, we have
\begin{align*}
    \mathbb E\left[|\tilde{X}_t^\rho|^2\right]&+\mathbb E \left[\int_t^T |\tilde{q}_s^\rho|^2ds \right] +\mathbb E \left[\int_t^T |\tilde{z}_s^\rho|^2ds\right]\\
    &=2\mathbb E \left[\int_t^T \tilde{X}_s^\rho (\Ab_s^\rho \tilde{X}_s^\rho+\Bb_s^\rho \tilde{q}_s^\rho+\Cb_s^\rho \tilde{z}_s^\rho+ \Db_s^\rho) ds \right]\\
    &\leq C \mathbb E \left[\int_t^T |\tilde{X}_s^\rho|^2ds \right] +\frac{1}{2}E\left[\int_t^T |\tilde{q}_s^\rho|^2ds\right]+\frac{1}{2}\mathbb E\left[\int_t^T |\tilde{z}_s^\rho|^2ds \right]+\frac{1}{2}\mathbb E \left[\int_t^T |\Db_s^\rho|^2ds\right],
\end{align*}
leading to
\begin{equation}
\label{estimatetilde}
    \mathbb E\left[|\tilde{X}_t^\rho|^2\right] + \frac{1}{2}\mathbb E \left[\int_t^T |\tilde{q}_s^\rho|^2ds \right]+\frac{1}{2} \mathbb E \left[\int_t^T |\tilde{z}_s^\rho|^2ds\right] \leq C \mathbb E \left[\int_t^T |\tilde{X}_s^\rho|^2ds\right] +\frac{1}{2} \mathbb E \left[\int_t^T |\Db_s^\rho|^2ds \right],
\end{equation} 
where $C>0$ is a constant that might change from line to line.
Notice that the inequality above and Fubini-Tonelli's theorem clearly imply
\[
\mathbb E\left[|\tilde X_t^\rho|^2 \right] \le C\int_t^T\mathbb E\left[|\tilde X_s^\rho|^2\right]ds + \frac12\mathbb E \left[\int_t^T|\Db_s^\rho|^2ds\right].
\]
Hence, applying the backward Gronwall inequality yields
\begin{equation}\label{eq:bound-intX}
\sup_{0\le t\le T}\mathbb E\left[|\tilde X_t^\rho|^2 \right] \le C\mathbb E\left[\int_0^T|\Db_s^\rho|^2ds \right].
\end{equation}
Next, taking \(t=0\) in the original estimate \eqref{estimatetilde} and using the bound \eqref{eq:bound-intX}, we obtain
\begin{align*}
\frac12\mathbb E\left[\int_0^T|\tilde q_t^\rho|^2dt \right]+ \frac12\mathbb E \left[\int_0^T|\tilde z_t^\rho|^2dt \right]& \le C\mathbb E \left[\int_0^T|\tilde X_t^\rho|^2 dt\right] + \frac12 \mathbb E \left[\int_0^T|\Db_t^\rho|^2dt \right]\\ &\leq C\mathbb E \left[\int_0^T|\Db_t^\rho|^2dt \right].
\end{align*}
Therefore,
\begin{equation}\label{eq:bound-supX}
\sup_{0\le t\le T}\mathbb E \left[|\tilde X_t^\rho|^2 \right] + \mathbb E \left[\int_0^T|\tilde q_t^\rho|^2dt \right]+\mathbb E \left[\int_0^T|\tilde z_t^\rho|^2dt \right]\le C\mathbb E \left[\int_0^T|\Db_t^\rho|^2dt \right].
\end{equation}
We recall the following estimate for Lipschitz BSDEs (see, e.g., \cite[Theorem 4.2.3, p.84]{Zhang}): There exists a constant $C>0$ such that
\[
\mathbb E\left[
\sup_{0\le t\le T}|X_t^\rho-\hat X_t|^2
+\int_0^T|q_t^\rho-\hat q_t|^2dt
+\int_0^T|z_t^\rho-\hat z_t|^2dt
\right]
\le
C\rho^2\mathbb E\left[|\xi-\hat\xi|^2\right].
\]
Hence, as $\rho\to0$, we have the following strong convergence in the respective solution spaces:
\[
X^\rho\to \hat X \quad \text{in } \mathcal S^2_n, \qquad
q^\rho\to \hat q \quad \text{in } \mathcal H^2_{n \times d}, \qquad
z^\rho\to \hat z \quad \text{in } \mathcal H^2_{n \times d}.
\]
After recalling
\[
\rho(\Delta X_t+\tilde X_t^\rho)=X_t^\rho-\hat X_t, \quad
\rho(\Delta q_t+\tilde q_t^\rho)=q_t^\rho-\hat q_t, \quad
\rho(\Delta z_t+\tilde z_t^\rho)=z_t^\rho-\hat z_t,
\] 
we can rewrite $\Ab_t ^\rho$ explicitly as:
\[
\Ab_t^\rho
=
\int_0^1
f_x\bigl(
t,
\hat X_t+\lambda(X_t^\rho-\hat X_t),
\hat q_t+\lambda(q_t^\rho-\hat q_t),
\hat z_t+\lambda(z_t^\rho-\hat z_t),
\mu_t
\bigr)d\lambda,
\]
and analogously for $\Bb_t^\rho$ and $\Cb_t^\rho$. 
To address the convergence of these integrals, we note that, for any sequence $\rho_n \to 0$, there exists a subsequence $\rho_{n_k} \to 0$ along which the following $dt \otimes d\mathbb{P}$-a.s. convergence holds:
\[
X_t^{\rho_{n_k}}\to \hat X_t, \qquad q_t^{\rho_{n_k}}\to \hat q_t, \qquad z_t^{\rho_{n_k}}\to \hat z_t \quad \text{as } k\to\infty.
\]
By the continuity of the partial derivatives $f_x, f_q, f_z$ with respect to the state variables, we obtain the pointwise convergence of the integrands. Applying the bounded convergence theorem to the $d\lambda$-integral yields:
\[
\Ab_t^{\rho_{n_k}}\to \hat f_x(t),\qquad
\Bb_t^{\rho_{n_k}}\to \hat f_q(t),\qquad
\Cb_t^{\rho_{n_k}}\to \hat f_z(t),
\quad dt \otimes d\mathbb{P}\text{-a.s. as } k\to\infty.
\]
By definition of $\Db_t^\rho$ in \eqref{D}, this immediately implies that for the chosen subsequence:
\[
\lim_{k\to\infty} |\Db_t^{\rho_{n_k}}|^2 = 0, \quad dt \otimes d\mathbb{P}\text{-a.s.}
\]
Furthermore, since the derivatives of $f$ are  bounded there exists a constant $C >0$ such that:
\[
|\Db_t^\rho|^2
\le C\left( |\Delta X_t|^2+|\Delta q_t|^2+|\Delta z_t|^2\right).
\]
Since the standard first-order variations satisfy $(\Delta X,\Delta q,\Delta z) \in \mathcal S_n^2 \times \mathcal H_{n\times d}^2 \times \mathcal H_{n\times d}^2$, the right-hand side of the inequality just above is $dt \otimes d\mathbb P$-integrable over $[0,T] \times \Omega$. 
Therefore, by the dominated convergence theorem, the subsequence converges to zero in $L^2(dt \otimes d\mathbb{P})$:
\[
\lim_{k\to\infty} \mathbb E\left[\int_0^T |\Db_t^{\rho_{n_k}}|^2dt\right] = 0.
\]
Since this holds for every chosen subsequence $\rho_{n_k}$, the entire directed family must converge to the same limit as $\rho \to 0$. Thus, we conclude:
\[
\lim_{\rho\to0}
\mathbb E\left[\int_0^T |\Db_t^\rho|^2dt\right]
=0.
\]
It follows from \eqref{eq:bound-supX} that
\[
\lim_{\rho\to0} \left( \sup_{0\le t\le T}\mathbb E\left[|\tilde X_t^\rho|^2\right] + \mathbb E\left[\int_0^T|\tilde q_t^\rho|^2dt\right]
+ \mathbb E\left[\int_0^T|\tilde z_t^\rho|^2dt\right]
\right) =0,
\]
which completes the proof.
\end{proof}

At this point, we use Ekeland's variational principle \cite{Ekeland} to study the optimization problem \eqref{backward CN}.
Given an optimal $\hat \xi$ as above, we introduce the mapping $F_\varepsilon :L^2 (\mathcal D) \to \mathbb R$ defined by 
\begin{equation}\label{F-epsilon}
F_\varepsilon(\xi)=\left(|X_0^\xi-a|^2+(J(\xi)-J(\hat \xi)+\varepsilon)_+ ^2\right)^{\frac{1}{2}},
\end{equation}
where $a$ is the given initial state and $\varepsilon$ is an arbitrary positive constant.
\begin{theo}
    We suppose \emph{(A1)-(A5)}. Let $\hat \xi$ be an $a$-optimal solution to \eqref{backward CN}. Then there exist $h_1 \in \mathbb{R}^n$ and $h_0 \in \mathbb{R}$ with $h_0 \geq 0$ and $|h_0|+|h_1|\neq 0$  such that the following variational inequality holds
    \begin{equation}
    \label{variational inequality CN}
        h_1 \Delta{X}_0 +h_0 \Delta{J}(\hat \xi) \geq 0.
    \end{equation}
\end{theo} 
\begin{proof}
   It is straightforward to verify that $F_\varepsilon $ is continuous on $L^2 (\mathcal D)$ and it satisfies 
\begin{align*} F_\varepsilon(\hat \xi) &=\varepsilon,\\
    F_\varepsilon(\xi)& >0, \ \textrm{for all } \xi \in L^2 (\mathcal D),\\
    F_\varepsilon(\hat \xi) & \leq \inf_{\xi \in L^2 (\mathcal D)}F_\varepsilon(\xi)+\varepsilon .
    \end{align*}
Then, by Ekeland's variational principle \cite{Ekeland}, there exists $\xi^\varepsilon \in L^2 (\mathcal D)$ such that
\begin{enumerate}[label=(\roman*)]
    \item $F_\varepsilon(\xi^\varepsilon)\leq F_\varepsilon(\hat \xi);$
    \item $d(\hat \xi,\xi^\varepsilon)\leq \sqrt{\varepsilon};$
    \item $F_\varepsilon(\xi)+\sqrt{\varepsilon}d(\xi,\xi^\varepsilon) \geq F_\varepsilon(\xi^\varepsilon)$ for all $\xi \in L^2 (\mathcal D)$.
\end{enumerate}
For any $\xi \in L^2 (\mathcal D)$, set  $\xi^{\varepsilon, \rho}=\xi^\varepsilon+\rho(\xi-\xi^\varepsilon)$, $ 0 \leq\rho \leq 1$. Let $(X^{\varepsilon,\rho},q^{\varepsilon,\rho}, z^{\varepsilon,\rho} )$ (resp. $(X^\varepsilon,q^\varepsilon, z^\varepsilon ) $) be the solution of \eqref{BSDE CN} under $\xi^{\varepsilon,\rho}$ (resp. $\xi^\varepsilon$ ), and $(\Delta{X}^\varepsilon,\Delta{q}^\varepsilon, \Delta{z}^\varepsilon )$ be the solution of \eqref{variational equation CN}, where $\hat \xi$ is replaced by $\xi^\varepsilon$. 

From (iii) we conclude
\begin{equation}
\label{third assertion CN}
F_\varepsilon(\xi^{\varepsilon, \rho})-F_\varepsilon(\xi^\varepsilon)+\sqrt{\varepsilon}d(\xi^{\varepsilon, \rho},\xi^\varepsilon) \geq 0.
\end{equation}
By  convergence results \eqref{convergence results CN} in Lemma \ref{convergence lemma},  we have
$$\lim_{\rho \to 0}\sup_{0 \leq t \leq T}\mathbb E\left[\rho^{-1}(X^{\varepsilon, \rho}_t-{X}^\varepsilon_t)-\Delta{X}^\varepsilon_t\right]=0,$$
and therefore
$$X^{\varepsilon, \rho}_0-{X}^\varepsilon_0\equiv X_0^{\xi^{\varepsilon, \rho}}-X_0^{\xi^\varepsilon}=\rho\Delta{X}^\varepsilon_0+o(\rho),$$
where $o(\rho)$ is understood in $L^1$-convergence.
By Taylor expansion, we have
\begin{equation}
\label{Taylor1}
    |X_0^{\xi^{\varepsilon, \rho}}-a|^2-|X_0^{\xi^\varepsilon}-a|^2=2( X_0^{\xi^\varepsilon}-a)(X_0^{\xi^{\varepsilon, \rho}}-X_0^{\xi^\varepsilon})+o(\rho)=2 \rho( X_0^{\xi^\varepsilon}-a) \Delta{X}^\varepsilon_0+o(\rho),
\end{equation}
\begin{equation}
\label{Taylor2}
   (J(\xi^{\varepsilon, \rho})-J(\hat \xi)+\varepsilon)_+^2-(J(\xi^\varepsilon)-J(\hat \xi)+\varepsilon)_+^2=2\rho (J(\xi^\varepsilon)-J(\hat \xi)+\varepsilon)_+\Delta{J}(\xi^\varepsilon)+o(\rho). 
\end{equation}
By \eqref{third assertion CN}, we have 
\begin{equation}
\label{10 CN}
    \frac{(F_\varepsilon(\xi^{\varepsilon, \rho}))^2-(F_\varepsilon(\xi^\varepsilon))^2}{F_\varepsilon(\xi^{\varepsilon, \rho})+F_\varepsilon(\xi^\varepsilon)}=F_\varepsilon(\xi^{\varepsilon, \rho})-F_\varepsilon(\xi^\varepsilon)\geq -\sqrt{\varepsilon}d(\xi^{\varepsilon, \rho},\xi^\varepsilon)=-\rho\sqrt{\varepsilon}d(\xi,\xi^\varepsilon).
\end{equation}
Dividing \eqref{10 CN} by $\rho$ and letting $\rho \to 0$, we obtain
\begin{align*}
    \lim_{\rho \to 0}\frac{F_\varepsilon(\xi^{\varepsilon, \rho})-F_\varepsilon(\xi^\varepsilon)}{\rho}
    &=\lim_{\rho \to 0}\frac{1}{F_\varepsilon(\xi^{\varepsilon, \rho})+F_\varepsilon(\xi^\varepsilon)}\frac{(F_\varepsilon(\xi^{\varepsilon, \rho}))^2-(F_\varepsilon(\xi^\varepsilon))^2}{\rho}&\\
    &=\lim_{\rho \to 0}\frac{2 \rho( X_0^{\xi^\varepsilon}-a) \Delta{X}^\varepsilon_0+2\rho (J(\xi^\varepsilon)-J(\hat \xi)+\varepsilon)_+\Delta{J}(\xi^\varepsilon)+o(\rho)}{\rho [F_\varepsilon(\xi^{\varepsilon, \rho})+F_\varepsilon(\xi^\varepsilon)]}\\
    &=\frac{( X_0^{\xi^\varepsilon}-a) \Delta{X}^\varepsilon_0+ (J(\xi^\varepsilon)-J(\hat \xi)+\varepsilon)_+\Delta{J}(\xi^\varepsilon)}{F_\varepsilon(\xi^\varepsilon)}\\
    &\geq -\sqrt{\varepsilon}d(\xi,\xi^\varepsilon)=-\sqrt{\varepsilon}\left(\mathbb E\left[|\xi-\xi^\varepsilon|^2\right]\right)^\frac{1}{2},
\end{align*}
where the second equality follows from the definition of $F_\varepsilon$, and equations \eqref{Taylor1}, \eqref{Taylor2}.
Hence 
\begin{equation}
\label{optimality condition}
    h_1^\varepsilon \Delta{X}_0^\varepsilon+h_0^\varepsilon \Delta{J}(\xi^\varepsilon) \geq -\sqrt{\varepsilon}\left(\mathbb E \left[|\xi-\xi^\varepsilon|^2\right]\right)^\frac{1}{2},
\end{equation}
where $$h_0^\varepsilon =\frac{(J(\xi^\varepsilon)-J(\hat \xi)+\varepsilon)_+}{F_\varepsilon(\xi^\varepsilon)}\quad \textrm{and}\quad h_1^\varepsilon =\frac{X_0^{\xi^\varepsilon}-a}{F_\varepsilon(\xi^\varepsilon)}.$$
Moreover, we have  $h_0^\varepsilon \geq 0$ and
\[
|h_0^\varepsilon |^2+|h_1^\varepsilon|^2=1,
\]
by the definition of $F_\varepsilon$. Then there exists a convergent subsequence of $(h_1^\varepsilon,h_0^\varepsilon)$, for $\varepsilon \to 0$, whose limit is denoted by $(h_1,h_0)$. Finally, when $\varepsilon \to 0$ we have $\Delta{X}_0^\varepsilon \to \Delta{X}_0$ and $\Delta{J}(\xi^\varepsilon) \to \Delta{J}(\hat \xi)$, which completes the proof.
\end{proof}
\begin{proof}[\textup{\textbf{Proof of Theorem~\ref{Necessary Sufficient Condition CN} (i)(Necessary Conditions)}}]

We apply Ito's lemma to $Y_t\Delta{X}_t$, which yields 
   \begin{align}
       d(Y_t \Delta{X}_t) 
       &= -Y_t \Big(\hat f_x(t)\Delta{X}_t+ \hat f_q(t)\Delta{q}_t+ \hat f_z(t)\Delta{z}_t\Big) dt+\Delta{q}_tY_t dW_t+\Delta{z}_t Y_t  dW_t^0 \nonumber\\
       &\quad +\Delta{X}_t \Big(\hat f_x(t) Y_t+h_0 \hat \ell_x(t)\Big) dt+\Delta{X}_t\Big( \hat f_q(t) Y_t+h_0 \hat \ell_q(t)\Big) dW_t\nonumber \\
       &\quad + \Delta{X}_t\Big( \hat f_z(t)Y_t+h_0 \hat \ell_z(t)\Big) dW_t^0+ \Delta{q}_t\Big( \hat f_q(t) Y_t+h_0 \hat \ell_q(t)\Big) dt \nonumber\\
       &\quad +\Delta {z}_t\Big( \hat f_z(t) Y_t+h_0\hat \ell_z(t)\Big) dt. \label{eq:YDeltaX}
   \end{align}
To handle the stochastic integrals above, we use a standard localization argument. Define the stopping times
\[
\tau_N= \inf\Big\{ t\in[0,T]:|X_t|+|\hat X_t|+|q_t|+|\hat q_t|+|z_t|+|\hat z_t|+|Y_t|+|M_2(\mu_t)|\geq N \Big\}\wedge T .
\]
Integrating both side of equality \eqref{eq:YDeltaX} over $[0,\tau_N]$, and taking the expectation, we have
\begin{equation}
\label{stopped}
   \mathbb E\left[Y_{\tau_N}\Delta X_{\tau_N}-Y_0\Delta X_0\right]
= \mathbb E \left[\int_0^{\tau_N}
\left(h_0\Delta X_t\hat \ell_x(t)+h_0\Delta q_t\hat \ell_q(t)+h_0\Delta z_t\hat \ell_z(t)\right)dt \right]. 
\end{equation}
To pass to the limit as $N \to \infty$ we show that each term in the stopped identity \eqref{stopped} is dominated by an integrable random variable.\\ 
Since $\tau_N\uparrow T$ a.s. and $Y,\Delta X$ have continuous sample paths,
\[Y_{\tau_N}\Delta X_{\tau_N}\to Y_T\Delta X_T \quad\text{a.s.}
\]
Moreover, the sequence $\{Y_{\tau_N} \Delta X_{\tau_N}\}_N$ is dominated by the product of the suprema of the processes:
\[
|Y_{\tau_N} \Delta X_{\tau_N}| \leq \left( \sup_{0 \leq t \leq T} |Y_t| \right) \left( \sup_{0 \leq t \leq T} |\Delta X_t| \right) =: G_1.
\]
By the Cauchy-Schwarz inequality and the fact that $Y, \Delta X \in \mathcal{S}^2_n$, we have
\[
\mathbb{E}[G_1] \leq \left( \mathbb{E} \left[\sup_{0\leq t\leq T} |Y_t|^2\right] \right)^{1/2} \left( \mathbb{E} \left[\sup_{0\leq t\leq T} |\Delta X_t|^2 \right] \right)^{1/2} = \|Y\|_{\mathcal{S}^2_n} \|\Delta X\|_{\mathcal{S}^2_n} < \infty.
\]
Since the sequence is dominated by the $L^1$ random variable $G_1$, the dominated convergence theorem yields
\[
\lim_{N \to \infty} \mathbb{E}[Y_{\tau_N} \Delta X_{\tau_N}] = \mathbb{E}[Y_T \Delta X_T].
\]
Consider the term involving $\hat{\ell}_q$. First notice that
\[
\left| \int_0^{\tau_N} h_0 \Delta q_t \hat{\ell}_q(t) dt \right| \leq \int_0^T |h_0 \Delta q_t \hat{\ell}_q(t)| dt =: G_2.
\]
Applying Cauchy-Schwarz and taking expectation lead to
\[
\mathbb{E}[G_2] \leq h_0 \left( \mathbb{E} \left[\int_0^T |\Delta q_t|^2 dt\right] \right)^{1/2} \left( \mathbb{E} \left[ \int_0^T |\hat{\ell}_q(t)|^2 dt \right] \right)^{1/2}.
\]
The first term on the right-hand side is finite since $\Delta q \in \mathcal{H}^2_{n \times d}$. The second term is finite due to the linear growth condition (B3) and the fact that $X \in \mathcal{S}^2_n, q, z \in \mathcal{H}^2_{n \times d}$, and $\mu$ satisfies $\mathbb E \left[\int_0^T M_2(\mu_t)^2 dt\right]<\infty$. Thus, $G_2 \in L^1$, and hence the dominated convergence theorem applies, yielding
\[
\mathbb E\left[\int_0^{\tau_N} h_0\Delta q_t\hat\ell_q(t)\,dt \right]
\to
\mathbb E \left[\int_0^T h_0\Delta q_t\hat\ell_q(t)\,dt\right].
\]
By treating similarly the terms in \eqref{stopped} involving $\hat\ell_z$  and $\hat\ell_x$, we obtain
\[
\mathbb{E}\left[Y_T \Delta X_T - Y_0 \Delta X_0\right] = \mathbb{E}\left[\int_0^T h_0 \Big(\Delta X_t \hat \ell_x(t) + \Delta q_t \hat \ell_q(t) + \Delta z_t \hat \ell_z(t) \Big)dt\right].
\]
Hence, by  \eqref{variational cost CN},
\[
\mathbb E\left[ Y_T(\xi-\hat \xi)-h_1\Delta{X}_0\right]=h_0\Delta{J}(\hat \xi)-\mathbb E\left[h_0\phi_x(\hat \xi,\mu_T^X)\Delta{X}_T\right].
\]
By the variational inequality \eqref{variational inequality CN}, there exist $h_1 \in \mathbb{R}^n$ and $h_0 \in \mathbb{R}$, with $h_0 \geq 0$ and $|h_0|+|h_1|\neq 0$, such that
\begin{equation}\label{ineq}
\mathbb E \left[ \big(Y_T +h_0\phi_x(\hat \xi,\mu_T^X)\big) \big(\xi-\hat \xi\big) \right]= h_1 \Delta{X}_0+h_0\Delta{J}(\hat \xi) \geq 0, \quad \textrm{for all }\xi \in L^2 (\mathcal D).
\end{equation}
It remains to show that the above inequality implies
\[
\big(Y_T +h_0\phi_x(\hat \xi,\mu_T^X)\big)\big( \xi-\hat \xi\big) \geq 0, \quad \textrm{for all }\xi \in L^2(\mathcal{D}).
\]
Assume, by contradiction, that there exist $\xi \in L^2(\mathcal D)$ and an event $A_0$, with $\mathbb P(A_0)>0$, such that $$\big(Y_T +h_0\phi_x(\hat \xi,\mu_T^X)\big)\big( \xi-\hat \xi\big) < 0$$ on $A_0$. By convexity of $\mathcal{D}$, the random variable $$\xi'=\xi \mathbbm{1} _{A_0}+\hat \xi \mathbbm{1} _{A_0^c} = \mathbbm{1} _{A_0}\xi+(1-\mathbbm{1} _{A_0})\hat \xi$$ 
belongs to $L^2(\mathcal{D})$ as well. Moreover, we have
\begin{equation*}
    \begin{aligned}
    \mathbb{E} \left[ \big( Y_T +h_0\phi_x(\hat \xi,\mu_T^X)\big)\big( \xi'-\hat \xi\big)\right]&=
     \mathbb{E} \left[ \big(Y_T +h_0\phi_x(\hat \xi,\mu_T^X)\big)\big( \mathbbm{1} _{A_0}\xi+(1-\mathbbm{1} _{A_0})\hat \xi-\hat \xi\big) \right]\\
     &=\mathbb{E} \left[ \big( Y_T +h_0\phi_x(\hat \xi,\mu_T^X)\big) \big(\xi-\hat\xi\big) \mathbbm{1} _{A_0} \right]< 0,  
    \end{aligned}
\end{equation*}
which contradicts inequality \eqref{ineq}. The proof is therefore complete.
\end{proof}

\begin{proof} [\textup{\textbf{Proof of Theorem~\ref{Necessary Sufficient Condition CN} (ii)(Sufficient conditions)}}]
By convexity of the Hamiltonian, we have
\begin{equation}
\label{Convexity Hamiltonian CN}
\begin{split}
     \int_0^T &\left(H(t,X_t,  q_t, z_t, \mu_t, Y_t, h_0)-H(t, \hat X_t,  \hat q_t, \hat z_t, \mu_t, Y_t,h_0)\right)dt \\ 
   & \geq \int_0^T  H_x(t, \hat X_t,  \hat q_t, \hat z_t, \mu_t,  Y_t, h_0)( X_t- \hat X_t ) dt \\
   & \quad  +\int_0^T  H_q(t, \hat X_t,  \hat q_t, \hat z_t, \mu_t, Y_t,h_0)( q_t- \hat q_t) dt\\
   & \quad +\int_0^T H_z(t, \hat X_t,  \hat q_t, \hat z_t, \mu_t, Y_t,h_0)( z_t- \hat z_t)dt\\
    & = \int_0^T\left( (X_t- \hat X_t)( \hat f_x(t) Y_t+h_0 \hat \ell_x(t))+(q_t-\hat q_t) (\hat f_q(t) Y_t+h_0 \hat \ell_q(t)) \right .\\
   & \quad \left. +(z_t-\hat z_t)(\hat f_z(t)Y_t+h_0 \hat \ell_z(t)) \right)dt. 
\end{split}
\end{equation}
    We define $\Delta X_t = X_t-\hat X_t $, which satisfies
\begin{equation}
\label{difference CN}
\begin{cases}
\begin{aligned}
    d(\Delta X_t)= \left(f(t, \hat X_t, \hat q_t, \hat z_t, \mu_t)-f(t,  X_t, q_t, z_t, \mu_t) \right)dt\\ +\,(q_t-\hat q_t)dW_t+(z_t-\hat z_t)dW_t^0
\end{aligned}\\
  
  \Delta X_T=\xi-\hat \xi, \quad \Delta X_0 =0.
\end{cases}
 \end{equation} 
By considering the adjoint SDE \eqref{adjoint equation CN} and the SDE \eqref{difference CN}, and applying integration by parts, we obtain 
\begin{align*}
    d(Y_t\Delta X_t )
    & = \bigg(\Delta X_t \big(\hat f_x(t) Y_t+h_0 \hat \ell_x(t)\big)+ (q_t-\hat q_t )\big( \hat f_q(t) Y_t+h_0 \hat \ell_q(t)\big)\\
    &\quad +(z_t-\hat z_t) \big( \hat f_z(t) Y_t+h_0 \hat \ell_z(t)\big)
    + Y_t \big(f(t, \hat X_t, \hat q_t, \hat z_t, \mu_t)-f(t,  X_t, q_t, z_t, \mu_t)\big)\bigg)dt\\
    &\quad +\bigg(\Delta X_t \big(\hat f_q(t) Y_t+h_0 \hat \ell_q(t)\big)+Y_t( q_t-\hat q_t)\bigg) dW_t \\ 
    &\quad +\bigg(\Delta X_t\big(\hat f_z(t) Y_t+h_0 \hat \ell_z(t)\big)+Y_t(z_t-\hat z_t)\bigg) dW_t^0.
\end{align*}
Arguing as in the proof of the necessary conditions, we obtain that the stochastic integral terms with respect to $dW$ and $dW^0$ are martingales. Hence, taking expectations on both sides leads to the following:
\begin{align}
\mathbb E \left[ Y_T ( \xi-\hat \xi ) \right] & =\mathbb E\left[ h_1( X_0^\xi-X_0^{\hat \xi})\right]+\mathbb E \left[\int_0^T\Delta X_t\big(\hat f_x(t) Y_t+h_0 \hat \ell_x(t)\big)dt  \right] \nonumber \\
 & \quad +\mathbb E\left[\int_0^T(q_t-\hat q_t )\big( \hat f_q(t) Y_t+h_0 \hat \ell_q(t)\big)dt \right] \nonumber\\
 &\quad +\mathbb E \left[\int_0^T(z_t-\hat z_t) \big( \hat f_z(t) Y_t+h_0 \hat \ell_z(t)\big) dt\right] \nonumber\\
 & \quad +\mathbb E \left[\int_0^T Y_t( f(t, \hat X_t, \hat q_t, \hat z_t, \mu_t)-f(t,  X_t, q_t, z_t, \mu_t) ) dt \right] \nonumber\\
 &\leq \mathbb E \left[\int_0^T \left(H(t,X_t,  q_t, z_t, \mu_t, Y_t,h_0)-H(t,\hat X_t, \hat q_t, \hat z_t, \mu_t,  Y_t,h_0) \right)dt \right] \nonumber\\
 &\quad +\mathbb E \left[\int_0^T Y_t \Big(f(t, \hat X_t, \hat q_t, \hat z_t, \mu_t)- f(t,  X_t, q_t, z_t, \mu_t)  \Big) dt \right] \nonumber\\
 &=\mathbb E \left[\int_0^T h_0 \Big( \ell(t,  X_t, q_t, z_t, \mu_t) - \ell(t, \hat X_t, \hat q_t, \hat z_t, \mu_t) \Big) dt \right]. \label{ineq2}
\end{align}
Recall that the cost functional is minimized over the set of $a$-feasible random variables, i.e., over all $\xi \in L^2_a (\mathcal D)$. This implies that  $X_0^\xi=a$, so that $\mathbb E\left[h_1 (X_0^\xi-X_0^{\hat \xi}) \right]=0$. Condition \eqref{optimality condition CN} and convexity of $\phi$ yield the following:
\begin{equation}
\label{inequality 1 CN}
\mathbb E\left[Y_T( \xi-\hat \xi ) \right] \geq -h_0\mathbb E \left[\phi_x(\hat \xi,\mu_T^X)(\xi-\hat \xi ) \right] \geq -h_0 \left(\mathbb E \left[\phi(\xi,\mu_T^X)\right] - \mathbb E \left[\phi(\hat \xi,\mu_T^X)\right]\right).
\end{equation}
Using the inequality \eqref{ineq2} above, we obtain
\begin{equation}
\label{inequality 2 CN}
 \mathbb E \left[ Y_T( \xi-\hat \xi ) \right]\leq
 \mathbb E \left[\int_0^T h_0 \Big( \ell(t, X_t, q_t, z_t, \mu_t) - \ell(t, \hat X_t, \hat q_t, \hat z_t, \mu_t)\Big) dt\right].
\end{equation}
Combining \eqref{inequality 1 CN} and \eqref{inequality 2 CN}, and using the fact that $h_0 > 0$, we finally get
$J(\hat \xi) \leq J(\xi)$, which ends the proof.
\end{proof}

\begin{proof}[\textup{\textbf{Proof of Claim~\ref{claim:h1}}}]
We prove that the equation \(G(h_1)=0\) admits exactly one solution. We first establish continuity. Let \(h_1^n\to h_1\), and write \(m_n:=m_T^{(h_1^n)}\), \(m:=m_T^{(h_1)}\), then , since \(Y_T^{(h_1^n)}=h_1^nY_T^{(1)}\to h_1Y_T^{(1)}=Y_T^{(h_1)}\) in \(L^2\), the fixed-point identities associated with \(\Gamma_{h_1^n}\) and \(\Gamma_{h_1}\), together with the contraction estimate \eqref{contractionestimate}, give
\[
(1-\rho)\|m_n-m\|_{L^2}
\le
\|\Gamma_{h_1^n}(m)-\Gamma_{h_1}(m)\|_{L^2}
\le
\frac1{k_1}\|Y_T^{(h_1^n)}-Y_T^{(h_1)}\|_{L^2}\to0\,,
\]
where \(\rho:=|k_2|/k_1<1\). Hence \(m_T^{(h_1^n)}\to m_T^{(h_1)}\) in \(L^2\), and the Lipschitz property of the positive part yields \(\xi^{(h_1^n)}\to \xi^{(h_1)}\) in \(L^2\). Since \(Y_T^{(1)}\in L^2\), Cauchy-Schwarz inequality gives \(G(h_1^n)\to G(h_1)\).

We next compute the two limits. In particular, let $h_1>0$ and set $r_{h_1}:=(m_T^{(h_1)}-c)/h_1\ge0$, then, dividing the fixed-point equation \textcolor{blue}{\eqref{conditionalmean}} by $h_1$ we obtain
\[
r_{h_1}=\mathbb E\left[\left(-\frac{Y_T^{(1)}}{k_1}-\frac{k_0+(k_1+k_2)c}{k_1h_1}-\frac{k_2}{k_1}r_{h_1}\right)^+\middle|\mathcal F_T^0\right],
\]
and since \(Y_T^{(1)}>0\) and \(\rho=-k_2/k_1<1\), it follows that
\[
0\le r_{h_1}\le \rho r_{h_1}+\frac{|k_0+(k_1+k_2)c|}{k_1h_1}\,,
\]
therefore \(\|r_{h_1}\|_{L^2}\to0\) as \(h_1\to+\infty\), and, by the formula defining \(\xi^{(h_1)}\), this implies \(\xi^{(h_1)}\to c\) in \(L^2\), consequently
\[
\lim_{h_1\to+\infty}G(h_1)=c\mathbb E[Y_T^{(1)}]-a<0.
\]
On the other hand, as \(h_1\to-\infty\), setting \(r=-h_1\to+\infty\), since \(m_T^{(h_1)}\ge c\) and \(-k_2\ge0\), there exists a deterministic constant \(C>0\) such that
\[
\xi^{(h_1)}\ge c+\left(\frac{rY_T^{(1)}-C}{k_1}\right)^+\,.
\]
Choosing \(\varepsilon>0\) such that \(A_\varepsilon:=\{Y_T^{(1)}>\varepsilon\}\) has positive probability, we have
\[
\mathbb E[\xi^{(h_1)}Y_T^{(1)}]
\ge
\left(c+\frac{r\varepsilon-C}{k_1}\right)\mathbb E[Y_T^{(1)}\mathbf 1_{A_\varepsilon}]\to+\infty\,,
\]
therefore \(\lim_{h_1\to-\infty}G(h_1)=+\infty\), and continuity gives the existence of at least one zero because \(a>c\mathbb E[Y_T^{(1)}]\).

Finally, to prove the uniqueness of the zero, we have that, if \(h_1<h_1'\), then \(\Gamma_{h_1}(m)\ge \Gamma_{h_1'}(m)\), for every \(m\in L^2(\mathcal F_T^0)\), and the maps \(\Gamma_h\) are increasing in \(m\) because \(k_2\le0\). Iterating \(\Gamma_{h_1'}\) from the fixed point of \(\Gamma_{h_1}\) gives
\[
m_T^{(h_1)}\ge m_T^{(h_1')}\quad\text{a.s.}\,,
\]
consequently \(\xi^{(h_1)}\ge \xi^{(h_1')}\) a.s., and \(G\) is non-increasing. Now, supposing that \(\widehat h_1\) is a zero, we have
\[
\mathbb E[\xi^{(\widehat h_1)}Y_T^{(1)}]=a>c\mathbb E[Y_T^{(1)}]\,,
\]
so that \(\mathbb P(\xi^{(\widehat h_1)}>c)>0\). If \(h_1>\widehat h_1\), the argument of the positive part defining \(\xi^{(h_1)}\) is strictly smaller than the one defining \(\xi^{(\widehat h_1)}\) on the set \(\{\xi^{(\widehat h_1)}>c\}\), hence
\[
G(\widehat h_1)-G(h_1)=\mathbb E[(\xi^{(\widehat h_1)}-\xi^{(h_1)})Y_T^{(1)}]>0\,,
\]
and \(G(h_1)<0\) for every \(h_1>\widehat h_1\). Analogously, if \(h_1<\widehat h_1\), then \(\xi^{(h_1)}>\xi^{(\widehat h_1)}\) on \(\{\xi^{(\widehat h_1)}>c\}\), so that \(G(h_1)>0\), and consequently the zero is unique.
\end{proof}

\end{document}